\documentclass[a4paper,11pt,twoside]{amsart}

\usepackage{graphicx}
\usepackage{amsfonts}
\usepackage{amsmath}
\usepackage{amsthm}
\usepackage{amssymb}
\usepackage[all]{xy}
\usepackage{hyperref}
\usepackage{aliascnt}
\usepackage[mathscr]{euscript}
\usepackage[dvipsnames]{xcolor}

\usepackage{graphicx}
\usepackage{verbatim}

\usepackage{tikz}
\usepackage{tikz-cd}
\usetikzlibrary{arrows,calc,matrix,patterns,shadings,backgrounds,fit}

\newtheorem{thm}{Theorem}[section]

\newaliascnt{prop}{thm}
\newtheorem{prop}[prop]{Proposition}
\aliascntresetthe{prop}

\newaliascnt{lem}{thm}

\aliascntresetthe{lem}

\newaliascnt{cor}{thm}
\newtheorem{cor}[cor]{Corollary}
\aliascntresetthe{cor}

\theoremstyle{definition}

\newaliascnt{definition}{thm}
\newtheorem{definition}[definition]{Definition}
\aliascntresetthe{definition}

\newaliascnt{remark}{thm}
\newtheorem{remark}[remark]{Remark}
\aliascntresetthe{remark}

\newaliascnt{ex}{thm}
\newtheorem{ex}[ex]{Example}
\aliascntresetthe{ex}

\newaliascnt{qn}{thm}
\newtheorem{qn}[qn]{Question}
\aliascntresetthe{qn}

\newaliascnt{stp}{thm}

\aliascntresetthe{stp}

\newaliascnt{conj}{thm}
\newtheorem{conj}[conj]{Conjecture}
\aliascntresetthe{conj}

\newaliascnt{setup}{thm}

\aliascntresetthe{setup}

\newtheorem*{plan}{Plan of the paper}

\numberwithin{equation}{section}

\DeclareMathOperator{\rk}{rk} 

\newcommand{\rest}[1]{|_{#1}} 

\newcommand{\id}{\mathrm{id}}
\newcommand{\Hom}{\mathrm{Hom}}
\newcommand{\Ext}{\mathrm{Ext}}

\newcommand{\Aut}{\mathrm{Aut}}
\newcommand{\Ob}{\mathrm{Ob}}

\newcommand{\mor}[1]{\xrightarrow{#1}} 
\newcommand{\isomor}{\mor{\sim}} 

\newcommand{\ZZ}{\mathbb{Z}}

\newcommand{\EE}{\mathbb{E}}
\newcommand{\FF}{\mathbb{F}}

\newcommand{\UU}{\mathbb{U}}

\newcommand{\RHom}{\rd\Hom} 

\newcommand{\cat}[1]{{\mathscr{#1}}} 

\newcommand{\cA}{\cat{A}}

\newcommand{\cB}{\cat{B}}

\newcommand{\cD}{\cat{D}}

\newcommand{\cE}{\cat{E}}
\newcommand{\cF}{\cat{F}}

\newcommand{\cl}{\cat{L}}
\newcommand{\cL}{\cat{L}}

\newcommand{\cM}{\cat{M}}

\newcommand{\cN}{\cat{N}}

\newcommand{\cP}{\cat{P}}

\newcommand{\cS}{\cat{S}}
\newcommand{\ct}{\cat{T}}
\newcommand{\cT}{\cat{T}}

\def\Aut{\mathop{\mathrm{Aut}}\nolimits}

\def\ch{\mathop{\mathrm{ch}}\nolimits}

\def\Coh{\mathop{\mathrm{Coh}}\nolimits}

\def\deg{\mathop{\mathrm{deg}}}

\def\dim{\mathop{\mathrm{dim}}\nolimits}

\def\Ext{\mathop{\mathrm{Ext}}\nolimits}

\def\Hom{\mathop{\mathrm{Hom}}\nolimits}

\def\RHom{\mathop{\mathbf{R}\mathrm{Hom}}\nolimits}
\def\id{\mathop{\mathrm{id}}\nolimits}
\def\Id{\mathop{\mathrm{Id}}\nolimits}

\def\coker{\mathop{\mathrm{coker}}\nolimits}

\def\Ob{\mathop{\mathrm{Ob}}}

\def\rk{\mathop{\mathrm{rk}}}

\def\Spec{\mathop{\mathrm{Spec}}}

\newcommand{\scat}[1]{{\mathbf{#1}}} 

\newcommand{\D}{\mathrm{D}} 

\newcommand{\Db}{\D^b}

\newcommand{\Dp}{\scat{Perf}} 
\newcommand{\K}{\scat{K}} 

\newcommand{\Ku}{\K^+}

\newcommand{\C}{\scat{C}} 

\newcommand{\fun}[1]{\mathsf{#1}} 

\newcommand{\fF}{\fun{F}}

\newcommand{\fO}{\fun{O}}

\newcommand{\fS}{\fun{S}}

\def\C{\ensuremath{\mathbb{C}}}

\def\P{\ensuremath{\mathbb{P}}}
\def\Q{\ensuremath{\mathbb{Q}}}
\def\R{\ensuremath{\mathbb{R}}}
\def\Z{\ensuremath{\mathbb{Z}}}
\def\L{\ensuremath{\mathbb{L}}}

\def\AA{\ensuremath{\mathcal A}}
\def\BB{\ensuremath{\mathcal B}}

\def\EE{\ensuremath{\mathcal E}}
\def\FF{\ensuremath{\mathcal F}}

\def\HH{\ensuremath{\mathcal H}}
\def\II{\ensuremath{\mathcal I}}

\def\MM{\ensuremath{\mathcal M}}

\def\OO{\ensuremath{\mathcal O}}

\def\UU{\ensuremath{\mathcal U}}

\def\ZZ{\ensuremath{\mathcal Z}}

\def\vv{{\mathbf v}}

\def\llra{\hbox to 10mm{\rightarrowfill}}
\def\lllra{\hbox to 15mm{\rightarrowfill}}

\def\llla{\hbox to 10mm{\leftarrowfill}}
\def\lllla{\hbox to 15mm{\leftarrowfill}}

\def\K{\mathbb K}

\def\llra{\hbox to 10mm{\rightarrowfill}}
\def\lllra{\hbox to 15mm{\rightarrowfill}}
\def\Ku{\mathcal{K}\!u}

\def\wH{\widetilde{H}}
\def\llambda{\ensuremath{\boldsymbol{\lambda}}}

\def\Stab{\mathop{\mathrm{Stab}}}

\begin{document} 

	\title[Categorical Torelli theorems: results and open problems]{Categorical Torelli theorems: results and open problems}

	\author[Laura Pertusi and Paolo Stellari]{Laura Pertusi and Paolo Stellari}
	
	\address{L.P.: Dipartimento di Matematica ``F.\
		Enriques'', Universit{\`a} degli Studi di Milano, Via Cesare Saldini
		50, 20133 Milano, Italy}
	\email{laura.pertusi@unimi.it}
	\urladdr{\url{http://www.mat.unimi.it/users/pertusi}}

	\address{P.S.: Dipartimento di Matematica ``F.\
	Enriques'', Universit{\`a} degli Studi di Milano, Via Cesare Saldini
	50, 20133 Milano, Italy}
	\email{paolo.stellari@unimi.it}
    \urladdr{\url{https://sites.unimi.it/stellari}}
    
    \thanks{L.P.~was partially supported by the research project PRIN 2017 ``Moduli and Lie Theory''. P.S.~was partially supported by the ERC Consolidator Grant ERC-2017-CoG-771507-StabCondEn, by the research project PRIN 2017 ``Moduli and Lie Theory'', and by the research project FARE 2018 HighCaSt (grant number R18YA3ESPJ)}

	\keywords{Derived categories, semiorthogonal decompositions, Torelli theorems}

	\subjclass[2010]{14J45, 14J28, 14F05, 18E30}

\begin{abstract}
We survey some recent results concerning the so called Categorical Torelli problem. This is to say how one can reconstruct a smooth projective variety up to isomorphism, by using the homological properties of special admissible subcategories of the bounded derived category of coherent sheaves of such a variety. The focus is on Enriques surfaces, prime Fano threefolds and cubic fourfolds.
\end{abstract}

\maketitle

\setcounter{tocdepth}{1}
\tableofcontents

\section*{Introduction}

During the last decades, derived categories of coherent sheaves on smooth projective varieties have played a special role in algebraic geometry. In particular, their use in birational geometry and for the study of the geometry of moduli spaces has produced important and unexpected results.

One natural and related question is if a smooth projective variety can be reconstructed, up to isomorphism, from its derived category. Due to the seminal work by Bondal and Orlov \cite{BO} we know that this is indeed a theorem when the variety has canonical bundle which is either ample or anti-ample (meaning that its dual is ample). On the other hand, Mukai \cite{Muk:ab} showed that this is no longer the case when the canonical bundle is trivial.

Of course, one may start wondering how one can study the derived category of coherent sheaves and how one can extract geometric information from it. This is a fast growing research area where several fruitful ideas have come into the picture. Important results in this direction are due to the \emph{Russian school}. The idea is to decompose the derived category in smaller pieces provided by nontrivial admissible subcategories which naturally generate the derived category and whose meaning is intrinsically connected to the geometry of the variety. This led to the notion of \emph{semiorthogonal decomposition} which is certainly one of the main characters in this survey.

Semiorthogonal decompositions are not always available and when available they are not, in general, canonical. For example, again when the canonical bundle is trivial, the derived category is indecomposable. Nonetheless, when a semiorthogonal decomposition is given, then its components turn out to be extremely interesting. One special case, which is prominent in this paper, is when the derived category $\Db(X)$ of a smooth projective variety $X$ contains a bunch of very simple objects, which are called \emph{exceptional}, and a geometrically meaningful residual category, which we call \emph{Kuznetsov component}.

In this paper we focus our attention on these components. Indeed, the problem we want to deal with can be now formulated in the following slightly vague form:

\medskip

\noindent{\it Categorical Torelli problem.} Let $X_1$ and $X_2$ be smooth projective varieties over a field, in the same deformation class and with Kuznetsov components $\Ku(X_1)$ and $\Ku(X_2)$. Is it true that $X_1$ and $X_2$ are isomorphic if and only if their Kuznetsov components are equivalent?

\medskip

As we mentioned above, semiorthogonal components are not in general canonical. Thus when such a problem has a positive answer, the corresponding Kuznetsov components have to emerge from very special geometric situations. This will be extensively explained in the examples of interest.

The \emph{Categorial Torelli theorems} discussed in this paper are indeed the results that provide a positive answer to the Categorical Torelli problem above. As it turns out, we need to be more precise about the equivalence between the Kuznetsov component in the sense that, in some cases, it has to satisfy some additional property. We will discuss this along the paper and discover that some of these assumptions are probably removable once some related conjectures are proved. Others, unfortunately, cannot be avoided. Just to give a short summary, the Categorical Torelli theorems that we will review are the following geometric situations:
\begin{enumerate}
\medskip
    \item[(CT1)] Enriques surfaces with an equivalence between the Kuznetsov components which is of Fourier\textendash Mukai type (\autoref{thm:derived_torreli});
\medskip
    \item[(CT2)] Cubic threefolds with no further assumptions on the equivalence between the Kuznetsov components (\autoref{thm;CTTcubic3folds});
\medskip
    \item[(CT3)] Several additional prime Fano threefolds (\autoref{subsect:moreFano});
\medskip
    \item[(CT4)] Cubic fourfolds with an equivalence between the Kuznetsov components which is compatible with the so called degree shift functor (\autoref{thm:catTT4folds}).
\medskip
\end{enumerate}

One important feature of the above results is the variety of techniques that are used to prove them. Indeed, (CT1) is a consequence of a general statement (see \autoref{prop:extension}) which allows us to extend a Fourier\textendash Mukai equivalence between the Kuznetsov components of two Enriques surfaces to an equivalence of their bounded derived categories and then to apply what we call Derived Torelli theorem. This is somehow related to the possibility to enhance exact functors at the dg level. On the other hand, (CT2) and (CT3) use in an extensive way the notion of stability conditions which only recently were constructed on Kuznetsov components (see \autoref{subsect:cubic3folds1}). The case of cubic fourfolds (CT4) can be handled either by using Hodge theoretic techniques as in \cite{HR} or, again, by using stability conditions \cite{BLMS}. Both approaches will be discussed in \autoref{subsect:catThm4folds}. The reader may actually view this survey as an occasion to review the most recent developments in so many different directions and to appreciate their power in combination with the theory of semiorthogonal decompositions.

In a negative direction, the Categorical Torelli theorem in the above formulation does not hold in the case of Fano threefolds of index $1$. For instance, a consequence of the results in \cite{KP_cones} provides the existence of non-isomorphic but birational Gushel\textendash Mukai threefolds with equivalent Kuznetsov components. A refined version, which takes into account the preservation of special objects in the Kuznetsov components, has been recently proved in \cite{JLLZ, JLZ} (see Section \ref{subsect:moreFano} for more details). It becomes then natural to ask whether the existence of an isomorphism between the given varieties with equivalent Kuznetsov components in (CT1)--(CT4) is just a special instance of the following more general problem:

\medskip

\noindent{\it Birational Categorical Torelli problem.} Let $X_1$ and $X_2$ be smooth projective varieties over a field, in the same deformation class and with Kuznetsov components $\Ku(X_1)$ and $\Ku(X_2)$. Is it true that if $\Ku(X_1)$ and $\Ku(X_2)$ are equivalent, then $X_1$ and $X_2$ are birational?

\medskip

This will be carefully discussed in the paper but it is worth mentioning that a converse to the Birational Categorical Torelli problem should not hold true (see \autoref{rmk_onlyonedirection}).

As a related line of investigation, we recommend the reader to consult \cite{LieOls}, where the authors study the problem of recovering the birational class of a smooth projective variety from its bounded derived category. More precisely, they conjecture that two smooth projective varieties are birational if there exists a (strongly) filtered exact equivalence between their bounded derived categories (see \cite[Section 10]{LieOls} for some motivic foundations of this conjecture and more generally \cite{KLOS_book} for reconstruction problems). Another very interesting research topic which is tightly related to the discussion in this paper but which will not be covered in this survey is the so called infinitesimal version of the Categorical Torelli theorems for Fano threefolds (see \cite{JLLZ,JLLZ2}).

We conclude this presentation by pointing out that the paper is accompanied by a list of open problems in the form either of questions or of conjectures. Their relation to the existing results will be carefully explained, but we take the opportunity to stress that their role in this paper is as important as the one of the main results.  

\begin{plan}
The survey is organized as follows. In \autoref{sect:preliminaries} we recollect some basic definitions and examples of semiorthogonal decompositions, and of Fourier--Mukai functors. 

In \autoref{sect:sodgeoex} we focus on some special examples of smooth projective varieties of low dimension having a semiorthogonal decomposition with a nontrivial component, known as the Kunzetsov component. In particular, we consider Enriques surfaces, prime Fano threefolds of index $1$ and $2$ and cubic fourfolds, and we recall the properties of their Kuznetsov component. 

\autoref{sect:EnriquesSurf} is devoted to the proof of the Categorical Torelli theorem for Enriques surfaces following \cite{LNSZ,LSZ}. The proof makes use of a general criterion, explained in \autoref{subsect:extending}, which allows us to extend Fourier--Mukai equivalences among admissible subcategories appearing in semiorthogonal decompositions under suitable assumptions. Then in \autoref{subsect:specialobjects} we characterize $3$-spherical objects in the Kuznetsov component; this is used in the proof of the main theorem given in \autoref{subsect:RDTTEnriques}. 

\autoref{sect:introstab1} provides a quick introduction to the notion of (weak) stability conditions, the definition of stability manifold together with the associated group actions, the construction via tilt-stability on surfaces and on some threefolds, and the conjectural approach through generalized Bogomolov inequalities.

\autoref{sect:cubic3folds} is devoted to the case of cubic threefolds. In \autoref{subsect:cubic3folds1} we explain the first method to construct stability conditions on the Kuznetsov component of a cubic threefold and how to apply this result to prove the Categorical Torelli theorem; the main reference is \cite{BMMS}. In \autoref{subsect:introstab2} we review the more recent method, introduced in \cite{BLMS}, to induce stability conditions on admissible subcategories which are left orthogonal components to an exceptional collection in a triangulated category with a Serre functor. Then we introduce the notion of Serre-invariant stability conditions. This is applied to cubic threefolds in \autoref{subsect:cubic3folds2} to construct stability conditions on the associated Kuznetsov component which are Serre-invariant. In \autoref{subsec:cubictthreefolds3} we explain some applications of this result on Serre-invariant stability conditions to the study of the geometry of moduli spaces and to give an alternative proof of the Categorical Torelli theorem; the main references are \cite{soheylaetal,FP,PY}. In \autoref{subsect:moreFano} we recall the state of art about these questions on Serre-invariant stability conditions and Categorical Torelli theorem for the Kuznetsov component of prime Fano threefolds of index $2$ and $1$. 

Finally, in \autoref{sec:cubic4folds} we analyze the higher dimensional case of cubic fourfolds. We recall the construction of stability conditions on the Kuznetsov component with the method of \cite{BLMS}. Then we explain the two known ways to prove the Categorical Torelli theorem from \cite{HR} and \cite{BLMS}, and how to deduce the Classical Torelli theorem from it. We end by discussing the analogous questions in the case of Gushel--Mukai fourfolds.
\end{plan}


{\Small\subsection*{Statements \& Declarations}

\noindent \textit{Competing interests.}  The authors read and 
approved the final manuscript. The author has no relevant financial or 
non-financial interests to disclose.

\noindent \textit{Availability of data and material.} This paper has 
no associated data and material.}

\section{Semiorthogonal decompositions: general results}\label{sect:preliminaries}

This section is devoted to a quick discussion about some basic facts concerning semiorthogonal decompositions (see \autoref{subsec:gen}) and Fourier\textendash Mukai functors (see \autoref{subsect:FM}). We also recall some preliminary examples in \autoref{subsect:sodgeoex}.

\subsection{The main definitions}\label{subsec:gen}
Despite its relatively simple definition and even if it is one of the simplest examples of triangulated category, the bounded derived category $\Db(\cA)$ of an abelian category $\cA$ has a very rich and often mysterious structure. This remains true when $\cA$ is the category $\Coh(X)$ of coherent sheaves on a smooth projective variety $X$ defined over a field $\K$.

In the latter case, there are several approaches to the study of the structure of $\Db(X):=\Db(\Coh(X))$:
\begin{itemize}
    \item We can look at the way $\Db(X)$ is generated;
    \item Mimicking representation theory, we can look at the action of the autoequivalences group of $\Db(X)$ on a suitable vector space or lattice (e.g.\ the total cohomology of $X$);
    \item We can decompose $\Db(X)$ into smaller pieces and try to understand how the geometry of $X$ is encoded by those pieces.
\end{itemize}
The first strategy has a long history, initiated by the beautiful  paper \cite{BVdB} (see also \cite{R}). More recently, it was discovered that this is intimately related to the way $\Db(X)$ can be enhanced to higher categorical structures (see, for example, \cite{BLL,CNS,CSUni1,LO}).

The second viewpoint has been widely adopted in the case of K3 surfaces, abelian varieties or, more generally, varieties with trivial canonical bundle (see, for example the seminal papers \cite{Mu,Or}). The reason being that, when the canonical bundle is trivial, $\Db(X)$ is indeed indecomposable by \cite{BrEq} and thus the third strategy cannot be pursued. On the contrary, the autoequivalences group is usually very rich and intimately related to the topology of the stability manifold which we will discuss later.

In this paper, we are mainly interested in the third approach which can be made precise by introducing the following definition. Let $\cT$ be a triangulated category which, for simplicity, we assume from now on to be linear over a field $\K$. A \emph{semiorthogonal decomposition} for $\cT$, denoted by
	\begin{equation*}
	\cT = \langle \cD_1, \dots, \cD_m \rangle,
	\end{equation*}
is a sequence of full triangulated subcategories $\cD_1, \dots, \cD_m$ of $\cT$ such that: 
	\begin{enumerate}
		\item $\Hom(F, G) = 0$, for all $F \in \cD_i$, $G \in \cD_j$ and $i>j$;	\item For any $F \in \cT$, there is a sequence of morphisms
		\begin{equation*}  
		0 = F_m \to F_{m-1} \to \cdots \to F_1 \to F_0 = F,
		\end{equation*}
		such that $\pi_i(F):=\mathrm{Cone}(F_i \to F_{i-1}) \in \cD_i$ for $1 \leq i \leq m$. 
	\end{enumerate}
We call the subcategories $\cD_i$ \emph{components} of the decomposition.

\begin{remark}\label{rmk:funsod}
It is a nice and relatively easy exercise to verify that (1) above yields that the factors $\pi_i(F)$ in (2) are uniquely determined and functorial, for all $F\in\cT$ and for all $i=1,\ldots,m$. Hence, in presence of a semiorthogonal decomposition, we get the \emph{$i$-th projection functor}  $\pi_i\colon\cT\to\cD_i$.
\end{remark}

Given a semiorthogonal decomposition for $\cT$ as above, denote by $\alpha_i\colon\cD_i\hookrightarrow\cT$ the inclusion. We say that $\cD_i$ is \emph{admissible} if $\alpha_i$ has left adjoint $\alpha_{i}^*$ and right adjoint $\alpha_{i}^!$. In presence of a semiorthogonal decomposition
\[
\cT=\langle\cD_1,\cD_2\rangle,
\]
then $\pi_1$ and $\pi_2$ coincide with the left adjoint $\alpha_1^*$ and the right adjoint $\alpha_2^!$. Furthermore, if $\cD$ is an admissible subcategory of $\cT$, we set
\begin{equation}\label{eqn:orth}
{}^\perp\cD:=\{E\in\cT:\Hom(E,\cD)=0\}\qquad\cD^\perp:=\{E\in\cT:\Hom(\cD,E)=0\}
\end{equation}
to be the \emph{left orthogonal} and \emph{right orthogonal} subcategories of $\cD$, respectively.
These triangulated subcategories yield semiorthogonal decompositions
\[
\cT=\langle\cD^\perp,\cD\rangle=\langle\cD,{}^\perp\cD\rangle,
\]
when $\cD$ is admissible.

One special feature of the triangulated category $\cT=\Db(X)$, for $X$ a smooth projective variety, is that it has Serre functor. Recall that a \emph{Serre functor} of a triangulated category $\cT$ is an exact autoequivalence $\fS_\cT\colon\cT\to\cT$ inducing, for all $A$ and $B$ in $\cT$, an isomorphism
\[
\Hom(A,B)\cong\Hom(B,\fS_\cT(A))^\vee
\]
which is natural in both arguments. Such a functor is unique up to isomorphism of exact functors and thus we will refer to $\fS_\ct$ as `the' Serre functor of $\cT$.

\begin{ex}\label{ex:Serre}
If $X$ is a smooth projective variety, then the Serre functor $\fS_X:=\fS_{\Db(X)}$ takes the following explicit form
\[
\fS_X(-):=(-)\otimes\omega_X[\dim(X)],
\]
where $\omega_X$ is the dualizing sheaf of $X$. If $\alpha\colon\cD\hookrightarrow\Db(X)$ is an admissible subcategory, then $$\fS_\cD\cong\alpha^!\circ\fS_X\circ\alpha.$$ 
The latter is a general fact: if $\cT$ has Serre functor $\fS_\cT$, and it contains an admissible triangulated subcategory $\cD$, then $\cD$ has a Serre functor $\fS_\cD$ as well. The shape of $\fS_\cD$ is exactly the one above with $\fS_X$ replaced by $\fS_\cT$ (see \cite{B}).
\end{ex}

Recall that $\cT$ is $\K$-linear. An object $E\in\cT$ is \emph{exceptional} if $\Hom(E,E[p])=0$, for all integers $p\neq0$, and $\Hom(E,E)\cong \K$. A set of objects $\{E_1,\ldots,E_m\}$ in $\cT$ is an \emph{exceptional collection} if $E_i$ is an exceptional object, for all $i$, and $\Hom(E_i,E_j[p])=0$, for all $p$ and all $i>j$.
An exceptional collection $\{E_1,\ldots,E_m\}$ is
\begin{itemize}
\item \emph{orthogonal} if $\Hom(E_i,E_j[p])=0$, for all $i,j=1,\ldots,m$ with $i\neq j$ and for all integers $p$;
\item \emph{full} if the smallest full triangulated subcategory of $\cT$ containing the exceptional collection is equal to $\cT$;
\item \emph{strong} if $\Hom(E_i,E_j[p])=0$, for all $p\neq 0$ and all $i,j=1,\ldots,m$ with $i\neq j$.
\end{itemize}

Finally, assume that $\cT$ is a \emph{proper} $\K$-linear triangulated category. This means that 
\[
\dim_\K\left(\oplus_i\Hom(F,G[i])\right)<+\infty,
\]
for all $F$ and $G$ in $\cT$. Its \emph{numerical Grothendieck group} $\cN(\cT)$ is defined as the quotient
\[
\cN(\cT):=K(\cT)/ \ker \chi.
\]
Here $\chi$ denotes the Euler form on $K(\cT)$ defined by
\begin{equation} \label{eq_defeulerform}
\chi(-,-)=\sum_i(-1)^i \dim \Hom(-, -[i])
\end{equation}
and $K(\cT)$ stands for the \emph{Grothendieck group} of $\cT$, which is the free abelian group generated by isomorphism classes $[F]$ of objects $F \in \cD$ modulo the relation $[F]=[E]+[G]$ for every exact triangle $E \to F \to G$. Note that $K(-)$ and $\cN(-)$ are additive with respect to semiorthogonal decompositions.

\subsection{Basic geometric examples}\label{subsect:sodgeoex}

In this section we discuss the first examples of semiorthogonal decompositions of the derived categories of simple smooth projective varieties.

\begin{ex}[Points and exceptional objects]\label{ex:excobj}
It is an easy exercise to show that if $E$ is an exceptional object in a triangulated category $\cT$, then the smallest full triangulated subcategory $\langle E\rangle$ of $\cT$ containing $E$ is equivalent to the bounded derived category of a point. On the other hand, if $\cT$ is a proper triangulated category and $E\in\cT$ is exceptional, then $\langle E\rangle$ is an admissible subcategory of $\cT$ (see, for example, \cite[Proposition 2.6]{MSLectNotes}).
\end{ex}

In higher dimension we have two important and classical examples.

\begin{ex}[Projective spaces]\label{ex:Pn}
In the case of the $n$-dimensional projective space $\P^n$, a classical result of Beilinson \cite{Bei:Pn} shows that the set of line bundles
\begin{equation}\label{eqn:collPn}
\{\OO_{\P^n}(-n),\OO_{\P^n}(-n+1),\ldots,\OO_{\P^n}\}
\end{equation}
forms a full exceptional collection and so it yields a semiorthogonal decomposition
\[
\Db(\P^n)=\langle \OO_{\P^n}(-n),\OO_{\P^n}(-n+1),\ldots,\OO_{\P^n}\rangle.
\]
It should be noted that the collection \eqref{eqn:collPn} is also full and strong.
\end{ex}

\begin{ex}[Quadrics]\label{ex:quadrics}
Assume now that $Q$ is an $n$-dimensional smooth quadric in $\P^{n+1}$ defined by an equation $\{ q=0\}$. We assume $\mathrm{char}(\K)\neq 2$. According to \cite{Kap:Grass}, the category $\Db(Q)$ has a semiorthogonal decomposition by exceptional bundles whose explicit form depends on the parity of $n$. More precisely, if $n=2m+1$ is odd,
\[
\Db(Q)=\langle S,\OO_{Q},\OO_{Q}(1),\ldots,\OO_{Q}(n-1)\rangle,
\]
where $S$ is the spinor bundle on $Q$ defined as $\coker(\phi|_{Q})(-1)$ and $\phi\colon \OO_{\P^{n+1}}(-1)^{2^{m+1}}\to \OO_{\P^{n+1}}^{2^{m+1}}$ is such that  $\phi\circ(\phi(-1))=q\cdot \Id\colon \OO_{\P^{n+1}}(-2)^{2^{m+1}}\to \OO_{\P^{n+1}}^{2^{m+1}}$.

If $n=2m$ is even, then we get
\[
\Db(Q)=\langle S^-,S^+,\OO_{Q},\OO_{Q}(1),\ldots,\OO_{Q}(n-1)\rangle,
\]
where $S^-:=\coker(\phi|_{Q})(-1)$, $S^+:=\coker(\psi|_{Q})(-1)$, and $\phi,\psi\colon \OO_{\P^{n+1}}(-1)^{2^{m}}\to \OO_{\P^{n+1}}^{2^{m}}$ are such that $\phi\circ(\psi(-1))=\psi\circ(\phi(-1))=q\cdot \Id$. See \cite{Ottaviani} for more details on spinor bundles.
\end{ex}

\subsection{Fourier\textendash Mukai functors}\label{subsect:FM}
We now briefly recall how to define special classes of exact functors between admissible subcategories. The reader can have a look to \cite{CS:SurvFM} for a survey or to \cite{Huy} for an extensive treatment. Note that all functors are derived.

Let $X_1$ and $X_2$ be smooth projective varieties over a field $\K$ with admissible embeddings 
\[
\alpha_i:\cD_i\hookrightarrow\Db(X_i),
\]
for $i=1,2$.

\begin{definition}\label{def:FM}
An exact functor $\fF\colon\cD_1\to\cD_2$ is \emph{of Fourier\textendash Mukai type} (or a \emph{Fourier\textendash Mukai functor}) if there exists $\mathcal{E}\in\Db(X_1\times X_2)$ such that the composition $\alpha_2\circ \mathsf{F}$ is isomorphic to the restriction 
\[
\Phi_\EE|_{\cD_1}\colon\cD_1\to\Db(X_2).
\]
Here the exact functor $\Phi_{\EE}$ is given by
\[
\Phi_\mathcal{E}(-):=p_{2*}(\mathcal{E}\otimes p_1^*(-)),
\]
where $p_i$ is the $i$-th natural projection.
\end{definition}

By \cite[Theorem 7.1]{Kuz11}, the projection functor onto an admissible subcategory $\cD\hookrightarrow\Db(X)$ is of Fourier\textendash Mukai type. 
This motivates \cite[Conjecture 3.7]{Kuz07} which says that any exact equivalence $\fF\colon\cD_1\to\cD_2$ between admissible subcategories $\cD_i\hookrightarrow\Db(X_i)$, for $X_i$ smooth projective over $\K$, is of Fourier\textendash Mukai type. Indeed, \cite[Theorem 7.1]{Kuz11} is a special case of this conjecture for $\fF=\id$.

Here we propose the following restatement:

\begin{qn}\label{qn:FM1}
Is any exact fully faithful functor $\fF\colon\cD_1\to\cD_2$ between admissible subcategories $\cD_i\hookrightarrow\Db(X_i)$, for $X_i$ smooth projective over $\K$, of Fourier\textendash Mukai type?
\end{qn}

This is motivated by Orlov's result that any fully faithful exact functor $\mathsf{F}\colon\Db(X_1)\to\Db(X_2)$ is of Fourier\textendash Mukai type when $X_i$ is smooth projective over a field $\K$ (see \cite{Or}). It should be noted that in \cite{COS,CS} the assumptions on $\fF$ were weakened. In particular, assuming full is enough. More recently, \cite{Ol} extended Orlov's result to the smooth and proper case. Note that \autoref{qn:FM1} is related to \cite[Conjecture 3.7]{Kuz07}. 

Motivated by the recent work \cite{CSUni1,LO}, we can actually state a weaker version of \autoref{qn:FM1}:

\begin{qn}\label{qn:FM2}
Let $\fF\colon\cD_1\to\cD_2$ be an exact equivalence between admissible subcategories $\cD_i\hookrightarrow\Db(X_i)$, for $X_i$ smooth projective over $\K$. Is there a Fourier\textendash Mukai equivalence $\cD_1\cong\cD_2$?
\end{qn}

The problem above is mainly motivated by the recent developments about uniqueness of enhancements \cite[Corollary 9.12]{LO} (see also \cite[Section 7.2]{CSUni1} and the recent improvements in \cite{CNS}). The idea is that, given an equivalence $\fF\colon\Dp(X_1)\to\Dp(X_2)$ between the categories of perfect complexes on quasi-compact and quasi-separated schemes, one can replace $\fF$ with another equivalence which can be lifted to any dg model of $\Dp(X_i)$. Under additional suitable assumptions on $X_i$ (e.g.\ if $X_i$ is noetherian), the latter condition is equivalent to being of Fourier\textendash Mukai type, due to \cite{LS,To}.

We expect \autoref{qn:FM1} and \autoref{qn:FM2} to have negative answers in general also due to the results in \cite[Section 6.4]{CNS} (see, in particular, Corollary 6.12 and Remark 6.13 there). We will discuss later when one can positively answer them in the geometric settings we are interested in.

\section{Semiorthogonal decompositions for special projective varieties}\label{sect:sodgeoex}
In this section we would like to go beyond \autoref{ex:Pn} and \autoref{ex:quadrics} and consider semiorthogonal decompositions for $\Db(X)$ for special but very interesting smooth projective varieties. In most of the examples discussed in this section, the derived category will contain a small set of exceptional objects of geometric origin and a nontrivial admissible subcategory right orthogonal as in \eqref{eqn:orth} to the exceptional objects and which we will call \emph{Kuznetsov component}.

\subsection{Smooth projective curves}\label{subsect:curve}
It is well known that if $C$ is a smooth complex projective curve, then $\Db(C)$ determines $C$ up to isomorphism:

\begin{thm}[Derived Torelli theorem for curves]\label{thm:dertordercatcurves}
	Let $C_1$ and $C_2$ be smooth complex projective curves. Then $\Db(C_1)\cong\Db(C_2)$ if and only if $C_1\cong C_2$.
\end{thm}

As it is explained in the proof of \cite[Corollary 5.46]{Huy}, the delicate case which requires using the cohomology of the curve and then the classical Torelli theorem is when the genus of the curves is $1$. All the other cases follow from the following beautiful result by Bondal and Orlov \cite{BO}.

\begin{thm}[Derived Torelli theorem for (anti)Fano manifolds]\label{thm:dertordercatFano}
	Let $X_1$ and $X_2$ be smooth projective varieties with ample or antiample canonical bundle (i.e.\ either $\omega_{X_i}$ or $\omega_{X_i}^\vee$ is an ample line bundle). Then $\Db(X_1)\cong\Db(X_2)$ if and only if $X_1\cong X_2$.
\end{thm}

Apart from the case of $\P^1$ which is covered by \autoref{ex:Pn}, in genus greater than $0$ the triangulated category $\Db(C)$ does not have nontrivial semiorthogonal decompositions by \cite{Ok}.

\begin{remark}\label{rmk:curvesdiffields}
Note that \autoref{thm:dertordercatcurves} remains true over any algebraically closed field $\K$ when the genus of the curves is not $1$. This is because \autoref{thm:dertordercatFano} holds in this more general setting. On the other hand, when the genus is $1$, the assumption $\K=\C$ cannot be removed as there are nonisomorphic smooth projective curves of genus $1$ with equivalent derived categories (see \cite{AKW, ShinderZhang}).
\end{remark}

\subsection{Enriques surfaces}\label{subsect:Enriquesgeom}
Let $\K$ be an algebraically closed field of characteristic different from $2$. An \emph{Enriques surface} is a smooth projective surface $X$ defined over $\K$ such that $H^{1}(X,\OO_X)=0$ and the dualizing line bundle $\omega_X$ is nontrivial but $2$-torsion. An Enriques surface $X$ can be equivalently characterized as a quotient of a K3 surface by an involution acting without fixed points.

The derived category $\Db(X)$ of an Enriques surface determines the surface up to isomorphism in view of the following result which is a rewriting of  \cite[Proposition 6.1]{BM01} and \cite[Theorem 1.1]{HLT17}:

\begin{thm}[Derived Torelli theorem for Enriques surfaces]\label{thm:dertordercatEnr}
Let $X_1$ and $X_2$ be smooth projective surfaces defined over an algebraically closed field $\K$ of characteristic different from $2$. If $X_1$ is an Enriques surface and there is an exact equivalence $\Db(X_1)\cong\Db(X_2)$, then $X_1\cong X_2$.
\end{thm}

Clearly the situation becomes much more involved when $\K$ has characteristic $2$. In this case the definition has to be slightly modified: an Enriques surface is a minimal smooth projective surface whose canonical bundle is numerically trivial and such that the second Betti number is $10$. If $\mathrm{char}(\K)\neq 2$, this definition coincides with the one above. But when $\mathrm{char}(\K)=2$, then one gets three families:
\begin{itemize}
    \item \emph{Classical Enriques surfaces:} they are characterized by the fact that $\dim(H^{1}(X,\OO_X))=0$;
    \item \emph{Singular Enriques surfaces:} in this case $\dim(H^{1}(X,\OO_X))=1$ and such a cohomology group carries a nontrivial action of the Frobenius;
    \item \emph{Supersingular Enriques surfaces:} in this case $\dim(H^{1}(X,\OO_X))=1$ and this cohomology group carries a trivial action of the Frobenius.
\end{itemize}
In the first case the canonical bundle is nontrivial and $2$-torsion while, in the latter two cases, the canonical bundle is trivial. Singular Enriques surfaces are again realized as quotients of K3 surfaces. The reader can have a look at \cite{CD89} and \cite{DK} for an extensive treatment of Enriques surfaces and \cite{D} for a shorter but informative one.

It is then natural to raise the following question:

\begin{qn}\label{qn:extDerTorChar2}
Is \autoref{thm:dertordercatEnr} still true when $\mathrm{char}(\K)=2$ for some/all of the three families above?
\end{qn}

\begin{remark}\label{rm:moreFMpart}
It is clear that we should not expect an analogue of \autoref{thm:dertordercatEnr} to hold for all smooth projective surfaces. Indeed, already for abelian surfaces \cite{Muk:ab} and K3 surfaces \cite{Og,St}, this is known to be false.
\end{remark}

For the rest of this section we stick to the case where $\K$ is algebraically close and $\mathrm{char}(\K)\neq 2$. We can further analyze $\Db(X)$ by means of the following result which is certainly well-known to experts.

\begin{prop}[{\cite[Proposition 3.5]{LNSZ}}]\label{prop:exceptionalcollectionexists}
Let $X$ be an Enriques surface over $\K$ as above. Then $\Db(X)$ contains an admissible subcategory $\cL=\langle\cL_1,\dots,\cL_c\rangle$, where $\cL_1,\dots,\cL_c$ are orthogonal admissible subcategories and
\[
\cL_i=\langle L^i_1,\dots,L^i_{n_i}\rangle.
\]
Here:
\begin{itemize}
\item[{\rm (1)}] $L_j^i$ is a line bundle such that $L^i_j=L^i_1\otimes\OO_X(R^i_1+\dots+R^i_{j-1})$, where $R^i_1,\dots,R^i_{j-1}$ form a chain of $(-2)$ rational curves of $A_{j-1}$ type\footnote{This is a pretty compact and standard way to say that the rational curves $R^i_1,\dots,R^i_{j-1}$ form a root basis of type $A_{j-1}$.}; 
\item[{\rm (2)}] $\{L^i_1,\dots,L^i_{n_i}\}$ is an exceptional collection; and
\item[{\rm (3)}] $n_1+\dots+n_c=10$.
\end{itemize}
\end{prop}

We can illustrate the geometry attached to the collection $\cL$ in some interesting cases.

\begin{ex}\label{ex:genricEnr}
A generic Enriques surface does not contain $(-2)$-curves. Thus the collection in \autoref{prop:exceptionalcollectionexists} gets much simplified. In particular, with the above notation, we have $n_i=1$ for every $i=1, \dots, c$, and we get $10$ completely orthogonal blocks $\cL_i=\langle L_i\rangle$, where $L_i:= L^{i}_{1}$ is an exceptional line bundle, for all $i=1,\dots,10$.

If $\K=\C$ and $X$ does not contain $(-2)$-curves, then \cite{Zu} gives a very geometric interpretation of these orthogonal line bundles. Indeed, any ample polarization on $X$ of degree $10$ yields $10$ elliptic pencils each containing $2$ double fibers. Denote them by $F_i^+$ and $F_i^-$. Then we can take $L_i:=\OO_X(-F_i^+)$. One can prove that, for all $i$, we have the relation $F_i^+=F_i^-+K_X$, where $K_X$ is the canonical class. Using the Serre functor (see \autoref{ex:Serre}), one immediately sees that it is possible to change any $L_i$ to $\OO_X(-F_i^-)$ and still get a completely orthogonal collection of $10$ line bundles. In particular, $\Db(X)$ contains many distinct collections of exceptional objects, as we have at least $2^{10}=1024$ possible choices of orthogonal exceptional collections of line bundles in $\Db(X)$ (see \cite[Example 3.4]{LNSZ} for a more detailed discussion).
\end{ex}

Thus, if $X$ is an Enriques surface and $\cL$ is a collection of exceptional line bundles as in \autoref{prop:exceptionalcollectionexists}, then we get a semiorthogonal decomposition
\begin{equation}\label{eqn:sodEnriques}
\Db(X)=\langle\Ku(X,\cL),\cL\rangle.
\end{equation}
The admissible subcategory $\Ku(X,\cL):=\cL^\perp$ is referred to as the \emph{Kuznetsov component of $X$}. As the notation suggests, it is important to keep in mind that $\Ku(X,\cL)$ depends on $\cL$ and not just on $X$.

The Kuznetsov component is, at the moment, quite a mysterious subcategory. On one hand, it is certainly easy to show that it is nonzero. Indeed, its numerical Grothendieck group can be easily described:
\[
\cN(\Ku(X,\cL))\cong\Z\oplus\Z.
\]
Furthermore, if
\[
\kappa\colon\Ku(X,\cL)\hookrightarrow\Db(X)
\]
denotes the embedding with right adjoint $\kappa^!$, then the object
\begin{equation}\label{eqn:sphericalEnriques}
S_i:=\kappa^!(L^i_1)
\end{equation}
is nontrivial in $\Ku(X,\cL)$, for $L_1^i$ the line bundles in \autoref{prop:exceptionalcollectionexists}. On the other hand, three potentially interesting and related open problems are summarized by the following:

\begin{qn}\label{qn:KuzEnriques}
(i) Does the Serre functor $\fS_{\Ku(X,\cL)}$ have an explicit and computable description (other than the abstract one in \autoref{ex:Serre})?

(ii) Is $\Ku(X,\cL)$ indecomposable\footnote{This would mean that $\Ku(X,\cL)$ does not admit a nontrivial semiorthogonal decomposition.}?

(iii) Does \autoref{qn:FM1} or \autoref{qn:FM2}, with $\cD_i:=\Ku(X_i,\cL_i)$, have a positive answer for the Kuznetsov component $\Ku(X,\cL)$?
\end{qn}

In \autoref{subsect:specialobjects} we will comment on the action of $\fS_{\Ku(X,\cL)}$ on some special objects of $\Ku(X,\cL)$ and we will mention why a positive answer to \autoref{qn:KuzEnriques} (ii) may be interesting to provide yet another counterexample to the Jordan\textendash H\"older conjecture. A positive answer to \autoref{qn:KuzEnriques} (iii) (in either of the two forms) would yield a simpler statement for the Categorical Torelli theorem discussed later.

\begin{remark}\label{rmk:sodothersurfaces}
As we observed in the introduction, not all surfaces admit nontrivial semiorthogonal decompositions. Already for surfaces of Kodaira dimension $0$, K3 and abelian surfaces have indecomposable derived categories. Additional interesting results in dimension $2$ are contained in \cite[Section 5]{KO}.
\end{remark}

\subsection{Prime Fano threefolds}\label{subsect:Fano3foldsgeo}

We turn now to the case of Fano threefolds, i.e. smooth projective threefolds $X$ defined over an algebraically closed field $\K$ and such that $\omega_X^\vee=\OO_X(-K_X)$ is ample. We stick to the examples where the rank $\rho_X$ of the Picard group $\mathrm{Pic}(X)$ of $X$ is $1$, which are called \emph{prime Fano threefolds}. Under these assumptions, we denote by $H_X$ the primitive ample generator of $\mathrm{Pic}(X)$.

The classification of prime Fano threefolds was achieved in \cite{Is,MU} in characteristic $0$. These classification results have been extended to positive characteristic in \cite{SB}. These threefolds are classified by two numerical invariants. The first one is the \emph{index} which is the positive integer $i_X$ such that
\[
K_X=-i_X H_X.
\]
The second one is the \emph{degree} which is the positive integer $d:=H_X^3$.

It turns out that $i_X\in\{1,2,3,4\}$. The cases $i_X=3$ and $i_X=4$ correspond to $X=Q$ and $X=\P^3$, respectively, where $Q$ is a $3$-dimensional quadric. In both cases, we know all about $\Db(X)$ in view of \autoref{ex:Pn} and \autoref{ex:quadrics} (when $\mathrm{char}(\K) \neq 2$), respectively. Thus we can stick to $i_X=1,2$.

Prime Fano threefolds with index $1$ are organized in $10$ deformation types. Moreover, the degree of these Fano threefolds $d_X=2g_X-2$ is even and the deformation type is characterized by the choice of $2\leq g_X\leq 12$ but $g_X\neq 11$, where $g_X$ is called the \emph{genus} of $X$.

Now, for $g_X=2,3,4,5$ we consider the semiorthogonal decomposition
\begin{equation}\label{eqn:sodFano1}
\Db(X)=\langle\Ku(X),\OO_X\rangle.
\end{equation}
On the other hand, if $g_X$ is even and greater than $4$, then we have
\begin{equation}\label{eqn:sodFano2}
\Db(X)=\langle\Ku(X),\EE_2,\OO_X\rangle,
\end{equation}
where $\EE_2$ is a rank-$2$ stable vector bundle on $X$ whose existence is claimed in \cite{Muk92} (see \cite[Section 6]{BLMS} for a careful proof of this fact which is valid for algebraically closed fields of characteristic either $0$ or sufficiently large). For odd genus $g_X=7,9$, we use again \cite{Muk92} which yields a rank-$5$ and rank-$3$ vector bundle $\EE_5$ and $\EE_3$ and semiorthogonal decompositions
\begin{equation}\label{eqn:sodFano3}
\Db(X)=\langle\Ku(X),\EE_5,\OO_X\rangle\qquad\Db(X)=\langle\Ku(X),\EE_3,\OO_X\rangle,
\end{equation}
when $X$ has genus $7$ and $9$, respectively.

In all the above cases, the residual category $\Ku(X)$ is called \emph{Kuznetsov component}. It is worth pointing out that $\Ku(X)$ can be better understood in some interesting cases. What is known is summarized in the following table, where the third column indicates the reference where the semiorthogonal decomposition in the second column is provided. As above, the base field has characteristic either zero or sufficiently large.

\medskip

\begin{center}
	\begin{tabular}{r|l|l}
		\hline
		\multicolumn{3}{ c }{$\rho_X=1$ \& $i_X=1$}\\
		\hline
		$g_X$ &Semiorthogonal decomposition & Reference\\
		\hline
		12&$\Db(X_{22})=\langle\EE_4,\EE_3,\EE_2,\OO\rangle$ &{\cite[Thm.~4.1]{Kuz:Fano3folds}}\\
		10&$\Db(X_{18})=\langle\Db(C_2),\EE_2,\OO\rangle$ &
		{\cite[\S6.4]{Kuznetsov:Hyperplane}}\\
		9&$\Db(X_{16})=\langle\Db(C_3),\EE_3,\OO\rangle$ &{\cite[\S6.3]{Kuznetsov:Hyperplane}}\\
		8&$\Db(X_{14})=\langle\Ku(X_{14}),\EE_2,\OO\rangle$&
		{\cite{Kuz:V14}}\\
		7&$\Db(X_{12})=\langle\Db(C_7),\EE_5,\OO\rangle$ &{\cite[\S6.2]{Kuznetsov:Hyperplane}}\\
		6&$\Db(X_{10})=\langle\Ku(X_{10}),\EE_2,\OO\rangle$ &{\cite[Lem.~3.6]{Kuz:Fano3folds}}\\
		5&$\Db(X_{8})=\langle\Ku(X_8),\OO\rangle$ &\\
		4&$\Db(X_{6})=\langle\Ku(X_{6}),\OO\rangle$& \\
		3&$\Db(X_{4})=\langle\Ku(X_4),\OO\rangle$ &\\
		2&$\Db(X_{2})=\langle\Ku(X_2),\OO\rangle$ &\\
		\hline
	\end{tabular}
\end{center}

\medskip

\noindent In the table, $X_d$ denotes a prime Fano threefold of index $1$ and degree $d=2g-2$, $C_g$ denotes a smooth curve of genus $g$ while $\EE_i$ refer to vector bundles which are explicitly described in the references in the third column. It is worth to point out that in some cases the semiorthogonal decomposition of $\Db(X_6)$ can be refined. More precisely, note that $X_6$ is a complete intersection of a quadric hypersurface and a cubic hypersurface in $\P^5$. When the quadric cutting $X_6$ is smooth, there is a semiorthogonal decomposition
$$\Db(X_6)= \langle \cA_X, \cS, \OO \rangle,$$
where $\cS$ is the restriction to $X$ of a spinor bundle on the quadric.

Prime Fano threefolds with index $2$ are usually referred to as \emph{del Pezzo threefolds}. They all have a canonical semiorthogonal decomposition
\begin{equation}\label{eqn:sodFano4}
\Db(X)=\langle\Ku(X),\OO_X,\OO_X(H_X)\rangle,
\end{equation}
where $\Ku(X)$ is, as usual, the Kuznetsov component. The degree $d_X$ is subject to the bounds $1\leq d_X\leq 5$ and in some cases the Kuznetsov component can be further analyzed according to the following table:

\medskip

\begin{center}
	\begin{tabular}{r|l|l}
		\hline
		\multicolumn{3}{ c }{$\rho_X=1$ \& $i_X=2$}\\
		\hline
		$d_X$
		&Semiorthogonal decomposition &Reference\\
		\hline
		5 &$\Db(X_5)=\langle\FF_3, \FF_2,\OO, \OO(H_{X_5})\rangle$ &{\cite{orlov:Y5}}\\
		4 &$\Db(X_4)=\langle\Db(C_2),\OO,\OO(H_{X_4})\rangle$ &{\cite[Thm.~2.9]{BondalOrlov:Main}}\\
		3 &$\Db(X_3)=\langle\Ku(X_3),\OO,\OO(H_{X_3})\rangle$&\\
		2 &$\Db(X_2)=\langle\Ku(X_2),\OO,\OO(H_{X_2})\rangle$&\\
		1 &$\Db(X_1)=\langle\Ku(X_1),\OO,\OO(H_{X_1})\rangle$&\\
		\hline
	\end{tabular}
\end{center}

\medskip

In the table, $C_2$ is a smooth curve of genus $2$ and $\cF_3, \cF_2$ are vector bundles of rank $3$ and $2$, respectively, described explicitely in the reference in the third column. In the above list, when $i_X=2$ and $d_X=3$ we get the celebrated case of \emph{cubic threefolds} (i.e.\ smooth degree $3$ hypersurfaces in $\P^4$) which are going to be important examples that we will analyze in full detail. The Kuznetsov components of cubic threefolds yield examples of the following special class of triangulated categories.

\begin{definition}
(i) A triangulated category $\cT$ is a \emph{fractional Calabi\textendash Yau category} if $\cT$ has Serre functor $\fS_\cT$ and there exist positive integers $p$ and $q\neq 0$ such that $\fS_\cT^q=[p]$. The fraction $\frac{p}{q}$ is called the \emph{fractional dimension} of $\cT$.

(ii) A fractional Calabi\textendash Yau category where $q=1$ is a \emph{$p$-Calabi\textendash Yau category}. 
\end{definition}

The following result is going to be relevant later in the paper:

\begin{prop}\label{prop:cubic3foldseq}
Let $X$ be a cubic threefold. Then the Kuznetsov component $\Ku(X)$ in \eqref{eqn:sodFano4} is a fractional Calabi\textendash Yau of fractional dimension $\frac{5}{3}$.
\end{prop}

\begin{proof}
The admissible subcategory $\Ku(X)$ has Serre functor by \autoref{ex:Serre}. We can then simply apply \cite[Corollary 4.4]{Kuz:V14} (see also \cite[Corollary 4.3]{Kuz:V14} for a more general statement).
\end{proof}

The following question is then natural.

\begin{qn}\label{qn:indeccubic3folds}
Is $\Ku(X)$ indecomposable, when $X$ is a cubic threefold?
\end{qn}

\begin{remark}\label{rmk:fractCY}
A list of other prime Fano threefolds with Kuznetsov component which is a fractional Calabi\textendash Yau category can be deduced from \cite{KLect} (see Section 2.4 therein) and \cite{KP}. One interesting example is provided by \emph{quartic threefolds}, i.e.\ smooth hypersurfaces of degree $4$ in $\P^4$. According to the first table above ($\rho=i=1$ and $g=3$), if $X$ is a quartic threefold, we have a semiorthogonal decomposition as in \eqref{eqn:sodFano1}. In this case, $\Ku(X)$ is a fractional Calabi\textendash Yau category of fractional dimension $\frac{10}{4}$. We will comment on this case later.
\end{remark} 

After discussing the Categorical Torelli theorem for cubic threefolds, we will explain that \autoref{qn:FM2} has a positive answer for $\Ku(X)$, when $X$ is a cubic threefold (see \autoref{cor:eqbecFM}).

\subsection{Higher dimensional Fano manifolds: cubic fourfolds}\label{subsect:cubic4folds}

A \emph{cubic fourfold} is a smooth cubic hypersurface in $\P^5$. Here we assume that $X$ is defined over an algebraically closed field $\K$ such that $\mathrm{char}(\K)\neq 2$.

The derived category of a cubic fourfold $X$ has again a natural semiorthogonal decomposition
\begin{equation}\label{eqn:sodcubic4folds}
\Db(X)=\langle\Ku(X),\OO_X,\OO_X(H),\OO_X(2H)\rangle,
\end{equation}
where $H$ is the class of a hyperplane section of $X$. The admissible subcategory $\Ku(X)$ is the Kuznetsov component of $X$.

\begin{prop}\label{prop:cubic4foldKu1}
If $X$ is a cubic fourfold, then the Kuznetsov component $\Ku(X)$ is an indecomposable $2$-Calabi\textendash Yau category for which \autoref{qn:FM1} has a positive answer.
\end{prop}

\begin{proof}
The fact that $\Ku(X)$ is a $2$-Calabi\textendash Yau category is again a consequence of \cite[Corollary 4.3]{Kuz:V14}. Moreover, it is an indecomposable admissible subcategory due to the well-known argument in \cite[Section 2.6]{KLect}. The fact that \autoref{qn:FM1} has a positive answer is the content of \cite{LPZStrong} (based on \cite{CNS,CSSupp}).
\end{proof}

It was observed by Kuznetsov in \cite{Kuzcubics} that, in many cases, one can realize the Kuznetsov component of a cubic fourfold as the derived category of a K3 surface. For example, this happens for Pfaffian cubic fourfolds. Subsequent recent work carried out in \cite{AT} and \cite{BLMNPS} completely classified all cubic fourfolds whose Kuznetsov component is equivalent to the derived category of a K3 surface. This body of work is very much related to the following very influential conjecture (see \cite[Conjecture 1.1]{Kuzcubics}) which we state even though it is not directly related to the rest of this paper.

\begin{conj}[Kuznetsov]\label{conj:Kuzration}
A cubic fourfold $X$ is rational if and only if $\Ku(X)$ is equivalent to the bounded derived category of a K3 surface.	
\end{conj}

None of the two implications is clear but the conjecture perfectly matches the classical Hodge theoretic still conjectural characterization of rational cubic fourfolds due to Harris and Hassett. Some important results in this direction are contained in \cite{RS1,RS2}.

\begin{remark}\label{rmk:twistK3s}
As it was first pointed out in \cite{Kuzcubics}, there are cubic fourfolds $X$ such that $\Ku(X)$ is equivalent to the derived category $\Db(S,\alpha)$ of twisted coherent sheaves, where $S$ is a K3 surface and $\alpha$ is an element in the Brauer group $\mathrm{Br}(S):=H^2(S,\OO_S^*)_\mathrm{tor}$ of $S$ (see \cite[Chapter 1]{Cal:thesis} for an extensive introduction to twisted coherent sheaves and their derived categories). The complete classification of all cubic fourfolds for which this is true was carried out in \cite{BLMNPS,Huy:cubics}. In view of \autoref{conj:Kuzration} it is certainly interesting to understand which geometric property of $X$ corresponds to having an equivalence $\Ku(X)\cong\Db(S,\alpha)$.
\end{remark}

Because of the many similarities with the derived categories of K3 surfaces---only few of which have been discussed here---the Kuznetsov component $\Ku(X)$ is an example of so called \emph{noncommutative K3 surfaces}.

\begin{remark}\label{rmk:autoeqKuz}
For actual K3 surfaces, the description of the autoequivalences group of their bounded derived categories is a challenging problem (see, for example, \cite{BB,HMS:K3,Mu,Or}). One could of course try to describe the autoequivalences group of a noncommutative K3 surface as well. We will go back to this issue later in the paper. For later use, we content ourselves with the observation from \cite{Kuz:V14} that, for a cubic fourfold $X$, the category $\Ku(X)$ has always an autoequivalence
\[
\fO_X\colon\Ku(X)\to\Ku(X)\qquad E\mapsto \iota^*(E\otimes\OO_X(H)),
\]
where $\iota^*$ is the left adjoint of the inclusion $\iota\colon\Ku(X)\hookrightarrow\Db(X)$ and $H$ is a hyperplane class. Such an autoequivalence is called \emph{degree shift} functor. It is not difficult to see that, by definition, $\fO_X$ is of Fourier\textendash Mukai type.
\end{remark}

Assume now that $\K=\C$. It was observed in \cite{AT} (see also \cite[Section 3.4]{MSLectNotes}) that the Kuznetsov component $\Ku(X)$ of a cubic fourfold $X$ is equipped with an even unimodular lattice $\wH(\Ku(X),\Z)$ with a weight-$2$ Hodge structure induced by the Hodge decomposition of $H^4(X,\C)$. Such a lattice is usually referred to as the \emph{Mukai lattice} of $X$.

For the convenience of the reader let us spell out some details in the construction. First observe that the topological K-theory $K_{\mathrm{top}}(X)$ of $X$ comes equipped with the pairing
\[
\chi(v_1,v_2):=p_*(v_1^\vee\otimes v_2)\in K_{\mathrm{top}}(\mathrm{pt})\cong\Z,
\]
where $p\colon X\to\mathrm{pt}$. As a group $\wH(\Ku(X),\Z)$ is defined as the set of classes in $K_{\mathrm{top}}(X)$ which are orthogonal, with respect to $\chi(-,-)$ to the classes of the three line bundles $\OO_X(iH)$, with $i=0,1,2$. This is nothing but the topological K-theory $K_{\mathrm{top}}(\Ku(X))$ of the admissible subcategory $\Ku(X)$.
Then the restriction of $(-,-):=-\chi(-,-)$ to $K_{\mathrm{top}}(\Ku(X))$ defines a pairing on it, called the \emph{Mukai pairing}.

\begin{remark}\label{rmk:cubicK3}
It was proved in \cite{AT} that the lattice $\wH(\Ku(X),\Z)$ is deformation invariant. Moreover, when $\Ku(X)$ is equivalent to the derived category of a K3 surface, $\wH(\Ku(X),\Z)$ is Hodge isometric to the Mukai lattice of the K3 surface. In particular, as a lattice, $\wH(\Ku(X),\Z)\cong U^{\oplus 4}\oplus E_8(-1)^{\oplus 2}$, where $U$ is the hyperbolic lattice and $E_8(-1)$ is the twist by $-1$ of the lattice corresponding to the root system $E_8$.
\end{remark}

The lattice $\wH(\Ku(X),\Z)$ comes with a weight-$2$ Hodge structure defined as follows. Consider the natural map
\begin{equation}\label{altrarotturadiballe}
\mathbf{v}\colon K_{\mathrm{top}}(X)\to H^*(X,\Q)
\end{equation}
which is usually called \emph{Mukai vector}. See \cite[Section 2]{AT} for more details. Taking its complexification, we then set
\[
\begin{split}
\wH^{1,1}(\Ku(X)):=&\mathbf{v}^{-1}\left(\bigoplus_p H^{p,p}(X)\right)\\\wH^{2,0}(\Ku(X)):=&\mathbf{v}^{-1}(H^{3,1}(X))\\\wH^{0,2}(\Ku(X)):=&\mathbf{v}^{-1}(H^{1,3}(X))
\end{split}
\]
Set $\wH_\mathrm{Hodge}(\Ku(X),\Z):=\wH^{1,1}(\Ku(X))\cap\wH(\Ku(X),\Z)$.

For later use, we are interested in describing special classes in $\wH_\mathrm{Hodge}(\Ku(X),\Z)$. To this extent, if we denote by $\iota^*\colon\Db(X)\to\Ku(X)$ the projection functor as above and by $\ell$ any line in $X$, we can consider the classes
\begin{equation}\label{eqn:lambdas}
\llambda_1:=[\iota^*(\OO_\ell(1))]\qquad\llambda_2:=[\iota^*(\OO_\ell(2))]
\end{equation}
in the numerical Grothendieck group $\cN(\Ku(X))$ of $\Ku(X)$.

\begin{remark}\label{rmk:cubicHodgeconj}
The integral Hodge conjecture for cubic fourfolds was originally proved in \cite[Theorem 18]{VoiHodge} and then reproved in \cite[Corollary 29.8]{BLMNPS} (see also \cite{PeHodge} for a general treatment). As for the Kuznetsov component, one can prove that $\wH_\mathrm{Hodge}(\Ku(X),\Z)$ is naturally isometric to the numerical Grothendieck group $\cN(\Ku(X))$ of the Kuznetsov component and thus not only it contains interesting classes as observed above but it entirely consists of algebraic classes (see \cite[Theorem 29.2]{BLMNPS}).
\end{remark}

It is then clear that the classes $\llambda_1$ and $\llambda_2$ are in $\wH_\mathrm{Hodge}(\Ku(X),\Z)\subseteq\wH(\Ku(X),\Z)$ as well. We set
\begin{equation}\label{eqn:A2}
A_2=\langle\llambda_1,\llambda_2\rangle
\end{equation}
to be the primitive sublattice of $\wH_\mathrm{Hodge}(\Ku(X),\Z)$ generated by the classes $\llambda_1$ and $\llambda_2$. With the choice of these generators, $A_2$ is the free $\Z$-module $\Z\oplus\Z$ with intersection form given by the matrix
\[
\left(\begin{array}{cc}2&-1\\-1&2\end{array}\right).
\]

\section{Enriques surfaces}\label{sect:EnriquesSurf}

In this section we want to state and prove the Categorical Torelli theorem for Enriques surfaces as in \cite{LNSZ,LSZ}. As we will see, we need to consider the semiorthogonal decompositions in \autoref{subsect:Enriquesgeom}. The key idea is to prove and use an extension result for Fourier\textendash Mukai equivalences between admissible subcategories. In particular, no stability conditions are needed here.

\subsection{Extending Fourier\textendash Mukai equivalences}\label{subsect:extending}
In this section we illustrate a general criterion which was proved in \cite{LNSZ}. It allows us to extend Fourier\textendash Mukai equivalences between admissible subcategories under some assumptions on the nature of the semiorthogonal decompositions.

\begin{prop}[{\cite[Propositions 2.4 and 2.5]{LNSZ}}]\label{prop:extension}
	Let $\alpha_1\colon\cD_1\hookrightarrow\Db(X_1)$ be an admissible embedding and let $E\in \!^\perp\!\cD_1$ be an exceptional object.
	Let $\Phi_\cE \colon \Db(X_1)\rightarrow \Db(X_2)$ be a Fourier\textendash Mukai functor with the property that $\Phi_\cE(^\perp\!\cD_1)\cong 0$.  Suppose further that
	\begin{itemize}
		\item[{\rm (a)}] $\Phi_\cE|_{\cD_1}$ is an equivalence onto an admissible subcategory $\cD_2$ with embedding $\alpha_2 \colon \cD_2\hookrightarrow\Db(X_2)$, and
		\item[{\rm (b)}] there is an exceptional object $F\in{}^\perp\!\cD_2$ and an isomorphism $\rho \colon \Phi_\cE(\alpha_1\alpha_1^!(E))\isomor\alpha_2\alpha^!_2(F)$, where $\alpha_i^!$ is the right adjoint of $\alpha_i$.
	\end{itemize}
	Then there exists a Fourier\textendash Mukai functor $\Phi_{\tilde{\cE}}\colon \Db(X_1)\rightarrow \Db(X_2)$ satisfying
	\begin{itemize}
		\item [\rm{(1)}] $\Phi_{\tilde{\cE}}(^\perp\langle \cD_1,E\rangle)\cong \mathsf 0$;
		\item [\rm{(2)}] $\Phi_{\tilde{\cE}}|_{\cD_1}\cong \Phi|_{\cD_1}$ and $\Phi_{\tilde{\cE}}(E)\cong F$;
		\item[\rm{(3)}] $\Phi_{\tilde{\cE}}|_{\langle \cD_1, E\rangle}$ is an equivalence onto $\langle \cD_2,F\rangle$.
	\end{itemize}
\end{prop}

In general, such a criterion is not easy to apply. We will try to clarify this with a brief vague discussion. Assumption (b) is hard to verify in concrete examples and, in some cases, it might happen that (b) is not satisfied by a given Fourier\textendash Mukai equivalence $\Phi_\EE$ which must then be composed with some additional autoequivalence of $\cD_2$ (or $\cD_1$). Heuristically, we should expect to be able to apply such a result either when we have a good grip on the autoequivalence groups of $\cD_1$ and $\cD_2$ or when we have a good understanding of the object $\alpha_1\alpha_1^!(E)$. We will show in the next section that for Enriques surfaces we are in the second scenario.

Of course, the criterion can be iterated when the orthogonal complement of $\cD_i$ consists of more than one exceptional object. As we will see, this makes computations more complicated.

\begin{remark}\label{rmk:LKextending}
We conclude this section by pointing out that the above criterion is very much related to the gluing theory for dg categories and dg functors developed in \cite{KL}. The reader can have a look at \cite[Section 2.3]{LNSZ} for an extensive discussion.
\end{remark} 

\subsection{Special objects and their classification}\label{subsect:specialobjects}
As we mentioned, the Categorical Torelli theorem for Enriques surfaces will be obtained as an application of \autoref{thm:dertordercatEnr} and \autoref{prop:extension}. And, as we commented above, this is made possible by a complete understanding of the projection into the Kuznetsov component of the $10$ exceptional line bundles in $\cL$.

Let us begin with a general discussion.

\begin{definition}\label{def:sphpseudosph}
Let $\cT$ be a triangulated category that is linear over a field $\K$ and with Serre functor $\mathsf{S}_\cT$.
\begin{enumerate}

\item[(a)] An object $E$ in $\cT$ is \emph{$n$-spherical} if:
\begin{enumerate}
\item[(i)] There is an isomorphism of graded vector spaces $\RHom(E,E)\cong\K \oplus \K[-n]$;
\item[(ii)] $\mathsf{S}_\cT(E)\cong E[n]$.
\end{enumerate}
\item[(b)]An object $E$ in $\cT$ is \emph{$n$-pseudoprojective} if:
\begin{enumerate}
\item[(i)] There is an isomorphism of graded vector spaces $\RHom(E,E)=\K\oplus \K[-1]\oplus\dots\oplus \K[-n]$;
\item[(ii)] $\mathsf{S}_\cT(E)\cong E[n]$.
\end{enumerate}
\end{enumerate}
\end{definition}

Spherical objects were introduced and studied in \cite{ST}. They often appear in triangulated categories of Calabi\textendash Yau type and, more specifically, in the derived category of smooth projective Calabi\textendash Yau varieties. They naturally define special autoequivalences which are called spherical twists and which correspond, under Mirror Symmetry, to Dehn twists in the mirror Fukaya category.

The notions of spherical object and spherical twist have been widely extended and generalized. Actually, $n$-pseudoprojective objects are part of this more general picture. Indeed, the graded vector space of derived endomorphisms is, up to multiplying by $2$ the degree, the same as the graded vector space of the total cohomology of an $n$-dimensional complex projective space. Hence, $n$-pseudoprojective objects are slight generalizations of the kind of objects studied in \cite{HT,Krug}.

Let us go back to the geometric setting and let us assume that $X$ is an Enriques surface. As we explained in \autoref{subsect:Enriquesgeom}, we have a semiorthogonal decomposition
\begin{equation}\label{eqn:use1}
\Db(X)=\langle\Ku(X,\cL),\cL\rangle
\end{equation}
as in \eqref{eqn:sodEnriques}, where the $10$ exceptional line bundles in $\cL$ are as in \autoref{prop:exceptionalcollectionexists}. Without loss of generality, we can reorganize these exceptional objects to get a semiorthogonal decomposition
\begin{equation}\label{eqn:use2}
\cL=\langle\cL_1,\dots,\cL_c\rangle
\end{equation}
into blocks such that, if $c\neq 10$, then there is a positive integer $1\leq d\leq c$ such that $\cL_j$ consists of more than one object if $1\leq j\leq d$ and of just one object if $d<j\leq c$.

Consider now the corresponding objects $S_i\in\Ku(X,\cL)$ defined in \eqref{eqn:sphericalEnriques}. They provide a complete classification of $3$-spherical and $3$-pseudoprojective objects in $\Ku(X,\cL)$ according to the following result.

\begin{thm}[\cite{LNSZ}, Proposition 4.10 \& \cite{LSZ}, Theorem 2.7]\label{thm:Enrsphericals}
In the setting above, if $F$ is an object in $\Ku(X,\cL)$, then
\begin{enumerate}
    \item[{\rm (1)}] $F$ is $3$-spherical if and only if $F\cong S_j[k]$ for some $d<j\leq c$ and $k\in \Z$; \item[{\rm (2)}] $F$ is $3$-pseudoprojective if and only if $F\cong S_j[k]$ for some $1\leq j\leq d$ and $k\in \Z$.
\end{enumerate}
Furthermore, all these $3$-spherical and $3$-pseudoprojective objects are not isomorphic.
\end{thm}

As we pointed out in \cite[Remark 2.9]{LSZ}, this result (see also its easy consequence \cite[Corollary 2.8]{LSZ}) together with a positive answer to \autoref{qn:KuzEnriques} (ii) would provide another counterexample to the \emph{Jordan\textendash H\"older property} for semiorthogonal decompositions. Roughly, such a property predicts that if $X$ is a smooth projective variety then the semiorthogonal decompositions of $\Db(X)$ are essentially unique, up to a reordering of the components and up to equivalence. The fact that this property does not hold in general is not new due to the counterexamples in \cite{BBS,KuzJH}.

\subsection{The categorical Torelli theorem}\label{subsect:RDTTEnriques}
We are now ready to state the main result of this section which is a combination of \cite[Theorem A]{LNSZ} and \cite[Theorem A]{LSZ} (where we refer to the result below as the Refined Derived Torelli theorem).

\begin{thm}[Categorical Torelli theorem for Enriques surfaces]\label{thm:derived_torreli}
Let $X_1$ and $X_2$ be Enriques surfaces over an algebraically closed field $\K$ of characteristic different from $2$. If they possess semiorthogonal decompositions
\begin{equation*}\label{cond2}
\Db(X_i)=\langle \Ku(X_1,\cN_1),\cN_i\rangle,
\end{equation*}
where $i=1,2$ and $\cN_i$ is defined as in \autoref{subsect:specialobjects} (with $\cL$ replaced by $\cN_i$), and there exists an exact equivalence $\mathsf{F}: \Ku(X_1,\cN_1)\xrightarrow{\sim}\Ku(X_2,\cN_2)$ of Fourier\textendash Mukai type, then $X_1\cong X_2$.
\end{thm}

For simplicity, we illustrate the proof of \autoref{thm:derived_torreli} when $\cN_1$ and $\cN_2$ consist of $10$ completely orthogonal line bundles. Hence,
\[
\cN_i=\langle L_{i,1},\dots,L_{i,10}\rangle,
\]
for $i=1,2$. By \autoref{thm:Enrsphericals} (1), the projections $S_{i,j}$ of $L_{i,j}$ into $\Ku(X_i,\cN_i)$ are, up to shift and isomorphism, the only $3$-spherical objects in $\Ku(X_i,\cN_i)$. The more general case where $\cN_i$ contains blocks with more then one object is dealt similarly by using \autoref{thm:Enrsphericals} (2).

The idea is to extend the Fourier\textendash Mukai equivalence $\fF\colon\Ku(X_1,\cN_1)\to\Ku(X_2,\cN_2)$ step by step by adding all the $10$ exceptional objects in $\cN_i$. It is not difficult to see that, since these objects are completely orthogonal, it is enough to show how to add one of them. Thus, let us consider $L_{1,1}$ and its projection $S_{1,1}$. Since being $3$-spherical is a property which is invariant under equivalence, the object $\fF(S_{1,1})$ is $3$-spherical as well. By \autoref{thm:Enrsphericals} (1), up to shift, there is $j\in\{1,\dots,10\}$ and an isomorphism
\[
\fF(S_{1,1})\cong S_{2,j}.
\]
Again, by orthogonality, we can permute the exceptional objects $L_{2,j}$'s and assume, without loss of generality, that $j=1$.

Now, a direct application of \autoref{prop:extension} implies that the Fourier\textendash Mukai equivalence $\fF$ extends to a Fourier\textendash Mukai equivalence 
\[
\fF_1\colon\langle\Ku(X_1,\cN_1),L_{1,1}\rangle\to\langle\Ku(X_2,\cN_2),L_{2,1}\rangle.
\]
The same argument can be applied again for the other exceptional objects in $\cN_1$ and, in a finite number of steps, we get an equivalence $\Db(X_1)\cong\Db(X_2)$. Now we can just invoke \autoref{thm:dertordercatEnr} and deduce that $X_1\cong X_2$. The careful reader might have noticed that, a priori, the argument gets more complicated when we add more exceptional objects: one should classify $3$-spherical objects in categories which are larger that the original Kuznetsov component. But this is in reality much simpler because the line bundles $L_{i,j}$ are completely orthogonal and the projection of $L_{i,j}$ onto
\[
\langle\Ku(X_i,\cN_i),L_{1,i},\dots,L_{i,j-1}\rangle
\]
is the same as the projection onto $\Ku(X_i,\cN_i)$. Furthermore the situation gets even more involved when we deal with Enriques surfaces whose Kuznetsov component contains pseudoprojective objects. The detailed explanation about how these problems can be overcome are not suited for this survey and we refer to the original papers \cite{LNSZ} and \cite{LSZ}.

\begin{remark}\label{rmk:Enrconcl}
(i) The Categorical Torelli theorem above has a trivial converse: if $X_1$ and $X_2$ are isomorphic Enriques surfaces and $\Db(X_1)$ has a semiorthogonal decomposition as in \autoref{subsect:Enriquesgeom}, then $\Db(X_2)$ has a semiorthogonal decomposition of the same type and there is a Fourier\textendash Mukai equivalence $\Ku(X_1,\cL_1)\cong\Ku(X_2,\cL_2)$, for appropriate $\cL_1$ and $\cL2$. This is simply because any isomorphism $X_1\cong X_2$ induces a Fourier\textendash Mukai equivalence between the whole derived categories which then trivially restricts to a Fourier\textendash Mukai equivalence between the Kuznetsov components.

(ii) The technique used in the proof is quite powerful and it was also used in \cite{LNSZ} to give a new and simple proof of \cite[Conjecture 4.2]{IK} (see \cite[Theorem B]{LNSZ}).
\end{remark}

We conclude this section by observing that our proof of \autoref{thm:derived_torreli} is almost completely characteristic free. The only point where we use that $\K$ is not only algebraically closed but also such that $\mathrm{char}(\K)\neq 2$ is when we invoke \autoref{thm:dertordercatEnr}. Thus, if we can answer \autoref{qn:extDerTorChar2} in the positive, then we can extend \autoref{thm:derived_torreli} to Enriques surfaces over fields of characteristic $2$. Furthermore, if we provide a positive answer to \autoref{qn:KuzEnriques} (iii), we can avoid assuming that the equivalence between the Kuznetsov components is of Fourier\textendash Mukai type.

\section{A brief introduction to (weak) stability conditions}\label{sect:introstab1}

Stability conditions on triangulated categories have been defined for the first time by Bridgeland in \cite{Bri}, generalizing the notion of slope stability for sheaves on curves. Since then, the development of the theory has led to applications in classical algebraic geometry and in the study of moduli spaces of stable objects in admissible subcategories of the bounded derived category. In this section, we review the definition of (weak) stability conditions and the construction in the case of the bounded derived category of a smooth projective variety via tilt stability. Our main references are \cite{BMS,BLMNPS,Bri}.

\subsection{Definitions}\label{subsect:def}
Let $\K$ be an algebraically closed field of arbitrary characteristic. Let $X$ be a smooth projective variety over $\K$ and $\cD$ be a full admissible subcategory of $\Db(X)$. A (weak) stability condition on $\cD$ is essentially the data of the heart of a bounded t-structure and of a (weak) stability function detecting the semistable objects, satisfying certain compatibility conditions.

\begin{definition} \label{def:heart}
The {\it heart of a bounded $t$-structure} on $\cD$ is a full subcategory $\cA \subset \cD$ such that
\begin{enumerate}
\item[(i)] for $E, F\in \cA$ and $k<0$ we have $\Hom(E, F[k])=0$, and
\item[(ii)] for every object $E\in \cD$ there is a sequence of morphisms
$$
0 =E_0 \xrightarrow{\phi_1} E_1 \to \dots \xrightarrow{\phi_m} E_{m} =E
$$
such that $\mathrm{Cone}(\phi_{i})$ is of the form $A_{i}[k_i]$ for some sequence $k_1>k_2>\cdots>k_m$
of integers and objects $0 \neq A_i\in \cA$.
\end{enumerate} 
\end{definition}
\noindent Note that the heart of a bounded t-structure $\cA$ is not a triangulated category. By \cite{BBD} we have that $\cA$ is an abelian category. We denote by $K(\cA)$ the Grothendieck group of $\cA$. As an example, the abelian category $\Coh(X)$ of coherent sheaves on $X$ is the heart of a bounded t-structure on $\cD=\Db(X)$.

\begin{remark}
Given $E \in \cD$, the objects $A_i$ in Definition \ref{def:heart} are uniquely determined and functorial (also the integers $k_i$ are unique). They are called the cohomology objects of $E$ in the heart $\cA$.
\end{remark}

\begin{definition} \label{def:stabfctn}
Let $\cA$ be the heart of a bounded t-structure on $\cD$. A group homomorphism $Z \colon K(\cA) \rightarrow \C$ is a
\emph{weak stability function} on $\cA$ if for any $0 \neq E \in \cA$ we have $\Im Z(E) \geq 0$, and in the case that $\Im Z(E) = 0$, we have $\Re Z(E) \leq 0$. A \emph{stability function} on $\cA$ is a weak stability function $Z$ such that for any $0 \neq E \in \cA$ with $\Im Z(E) = 0$, we have $\Re Z(E) < 0$.
\end{definition}

Given a (weak) stability function $Z$, the \emph{slope} of $E \in \cA$ is 
$$\mu_{Z}(E)= 
\begin{cases}
-\frac{\Re Z(E)}{\Im Z(E)} & \text{if } \Im Z(E) > 0, \\
+ \infty & \text{otherwise,}
\end{cases}$$
and the \emph{phase} of $E$ is 
$$\phi(E)= 
\begin{cases}
\frac{1}{\pi}\text{Arg}(Z(E)) & \text{if } \Im Z(E)>0, \\
1 & \text{otherwise}. 
\end{cases}
$$
We point out that if $Z(E)=0$, then $\mu_Z(E)=+\infty$ and $\phi(E)=1$.  If $F=E[k]$ for $E \in \cA$, then $\phi(F)=\phi(E)+k$. 

Let $K(\cD)$ be the Grothendieck group of $\cD$. It is not difficult to see that $K(\cD)=K(\cA)$. Fix a finite rank free abelian group $\Lambda$ and a surjective morphism $\vv \colon K(\cD) \twoheadrightarrow \Lambda$. 
\begin{definition} \label{def_stabcond}
A \emph{weak stability condition} (with respect to $\vv$) on $\cD$ is a pair $\sigma=(\cA,Z)$, where $\cA$ is the heart of a bounded t-structure on $\cD$ and $Z \colon \Lambda \to \C$ is a group morphism called \emph{central charge}, satisfying the following properties:
\begin{enumerate}
    \item [(a)] The composition $K(\cA)=K(\cD) \xrightarrow{\vv} \Lambda \xrightarrow{Z} \C$ is a weak stability function on $\cA$ (we will write $Z(-)$ instead of $Z(\vv(-))$ for simplicity). 
    
    \noindent We say that an object $E \in \cD$ is $\sigma$-\emph{(semi)stable} if $E[k] \in \cA$ for some $k \in \Z$, and for every proper subobject $F \subset E[k]$ in $\cA$ we have $\mu_{Z}(F) < (\leq) \ \mu_{Z}(E[k]/F)$.
    \item [(b)] \emph{Harder--Narasimhan property}: Every object $E \in \cA$ has a filtration 
    $$0=E_0 \hookrightarrow E_1 \hookrightarrow \dots E_{m-1} \hookrightarrow E_m=E$$
    where $A_i:=E_i/E_{i-1} \neq 0$ is $\sigma$-semistable and $\mu_Z(A_1) > \dots > \mu_Z(A_m)$. 
    \item [(c)] \emph{Support property}: There exists a quadratic form $Q$ on $\Lambda \otimes \R$ such that the restriction of $Q$ to $\ker Z_{\R} \subset \Lambda \otimes \R$ is negative definite and $Q(E) \geq 0$ for all $\sigma$-semistable objects $E$ in $\cA$.
    \end{enumerate}
If $Z \circ \vv$ is a stability function, we say that $\sigma$ is a \emph{Bridgeland stability condition} on $\cD$ (with respect to $\vv$). \end{definition}

\begin{remark}
(i) It is possible to verify that the filtration in Definition \ref{def_stabcond}(b) is unique and functorial. Moreover, the Harder--Narasimhan property and \autoref{def:heart}(ii) imply that every object in $\cD$ has a filtration in $\sigma$-semistable ones, which are called HN factors. We denote by $\phi^+(E)$ (resp.\ $\phi^-(E)$) the largest (resp.\ smallest) phase of the HN factors of $0\neq E\in\cD$.

(ii) A (weak) stability condition $\sigma=(\cA, Z)$ determines a \textit{slicing}, i.e.\ a collection of full additive subcategories $\cP(\phi) \subset \cD$ for $\phi \in \R$, defined as follows:
\begin{enumerate}
\item[(1)] for $\phi \in (0,1]$, the subcategory $\cP(\phi)$ is the union of the zero object and all $\sigma$-semistable objects with phase $\phi$;
\item[(2)] for $\phi+n$ with $\phi \in (0,1]$ and $n \in \Z$, set $\cP(\phi+n):=\cP(\phi)[n]$.
\end{enumerate}
We will use the notation $\cP(I)$, where $I \subset \R$ is an interval, to denote the extension-closed subcategory of $\cD$ generated by the subcategories $\cP(\phi)$ with $\phi \in I$. By \autoref{def_stabcond} we have $\cP((0, 1])= \cA$.
\end{remark}

The notion of weak stability condition is very useful for the construction of Bridgeland stability conditions, as we will explain in \autoref{sec:tiltstab}. 

A Bridgeland stability condition is a stability condition in the sense of \cite{Bri}. Note that it is not clear whether there exist moduli spaces parametrizing semistable objects with a fixed class in $\Lambda$, since they do not have a GIT description. Following the recent developments in \cite{BLMNPS} about the theory of families of stability conditions and \cite{AH-LH} about the existence of good moduli spaces, we introduce the notion of stability condition with moduli spaces, which is a Bridgeland stability condition with ``well-behaved'' moduli functors. 

Assume that the base field $\K$ is of characteristic $0$. Given a Bridgeland stability condition $\sigma$ on $\cD$ with respect to $\vv$, fix $v \in \Lambda$ and $\phi \in \R$ such that $Z(v) \in \R_{>0}e^{i\pi\phi}$. Consider the functor 
$$\cM_\sigma(\cD, v) \colon (\text{Sch})^{\text{op}} \to \text{Gpd}$$
from the category of schemes over $\K$ to the category of groupoids, which associates to $T \in\Ob(\text{Sch})$ the groupoid $\cM_\sigma(\cD, v)(T)$ of all perfect complexes $E \in \D(X \times T)$ such that, for every point $t \in T$, the restriction $E_t$ of $E$ to the fiber $X \times \Spec(k(t))$ belongs to $\cD$, is $\sigma$-semistable of phase $\phi$ and $\vv(E_t)=v$.

\begin{definition} \label{def:stabcond}
A \emph{stability condition with moduli spaces} on $\cD$  (with respect to $\vv$) is a Bridgeland stability condition $\sigma=(\cA, Z)$ satisfying:
\begin{enumerate}
    \item [(d)] \emph{Openness:} For every $\K$-scheme $T$ and every perfect complex $E \in \D(X \times T)$, the set of points $\lbrace t \in T: E_t \in \cD \text{ and is } \sigma\text{-semistable} \rbrace$ is open.
    \item [(e)] \emph{Boundedness:} For every $v \in \Lambda$ the functor $\cM_\sigma(\cD, v)$ is bounded, i.e.\ there exists a pair $(B, \cE)$, where $B$ is a scheme of finite type over $\K$ and $\cE$ is an object in $\cM_{\sigma}(\cD, v)(B)$, such that for every $E \in \cM_{\sigma}(\cD, v)(\K)$ there exists a $\K$-rational point $b \in B$ satisfying $\cE_b \cong E$.
\end{enumerate}
\end{definition}

If $\sigma$ is a stability condition with moduli spaces on $\cD$ with respect to $\vv$, then by \cite[Theorem 21.24(3)]{BLMNPS}, which makes use of \cite{AH-LH}, for every $v \in \Lambda$ it follows that $\cM_{\sigma}(\cD, v)$ admits a good moduli space $M_\sigma(\cD, v)$ which is a proper algebraic space over $\K$.

A natural choice for the lattice $\Lambda$ is the numerical Grothendieck group $\cN(\cD)$ of $ \cD$. In fact, the numerical Grothendieck group is a free abelian group of finite rank. This follows from the fact that $\cD$ is an admissible subcategory of $\Db(X)$, thus we have the semiorthogonal decomposition $\Db(X)= \langle \cD, {}^\perp\cD \rangle$. Since $\cN(-)$ is additive, we have that $\cN(\cD)$ is a subgroup of the numerical Grothendieck group of $X$, which is a free abelian group of finite rank \cite[19.3.2]{Fulton}. This motivates the following definition.
\begin{definition}
A \emph{numerical} stability condition on $\cD$ is a stability condition with respect to the numerical Grothendieck group $\cN(\cD)$ of $ \cD$.
\end{definition}

\begin{ex}(Slope stability) \label{ex:slopestab}
Let $X$ be a smooth projective variety of dimension $n$ with ample class $H$. Define
$$\vv \colon K(\Db(X)) \to \Z^2,\quad \vv(E)=(H^n\ch_0(E), H^{n-1}\ch_1(E)),$$
where $\ch_0(E)$ and $\ch_1(E)$ stand for the rank and the first Chern class of $E$, respectively.
Set $\Lambda_H:=\text{Im}(\vv)$. Then the pair $\sigma_H=(\Coh(X), Z_H)$, where
$$Z_H \colon \Lambda_H \to \C, \quad Z_H(-)=-H^{n-1}\ch_1(-) + \sqrt{-1}H^n\ch_0(-),$$
defines a weak stability condition on $\Db(X)$ with respect to $\Lambda_H$. Indeed, if $E$ is a sheaf on $X$, then $H^{n}\ch_0(E) \geq 0$ and if it is $0$ (i.e.\ $E$ is a torsion sheaf), then $H^{n-1}\ch_1(E) \geq 0$. Moreover, by \cite[Lemma 2.4]{Bri} the HN property holds, and by \cite[Remark 2.6]{BLMS} the trivial form $Q=0$ fulfills the support property.

If $n=1$, i.e.\ $X$ is a curve, then $\sigma_H$ is a numerical stability condition on $\Db(X)$, recovering the classical notion of slope stability.
\end{ex}

\subsection{Stability manifold and actions}
Assume $\K$ is an algebraically closed field of arbitrary characteristic\footnote{In this section we could work in more generality without assuming $\K$ is algebraically closed, see \cite[Theorem 1.2]{BLMNPS}. However, we prefer to keep this condition to be compatible with the following sections.}. We denote by $\text{Stab}_{\Lambda}(\cD)$ the set of stability conditions on $\cD$ with respect to $\vv$. We consider on $\text{Stab}_{\Lambda}(\cD)$ the coarsest topology such that the maps $(\cA, Z) \mapsto Z$, $(\cA, Z) \mapsto \phi^+(E)$, $(\cA, Z) \mapsto \phi^-(E)$ are continuous for every $0\neq E\in \cD$.
A celebrated result of Bridgeland states that $\text{Stab}_{\Lambda}(\cD)$ has the structure of complex manifold.

\begin{thm}[Bridgeland Deformation Theorem, \cite{Bri}, \cite{BLMNPS}, Theorem 1.2] \label{thm:BrDefo}
The continuous map $\mathcal{Z}: \emph{Stab}_{\Lambda}(\cD) \to \Hom(\Lambda,\C)$ defined by $(\cA,Z) \mapsto Z$, is a local homeomorphism. In particular, the topological space $\emph{Stab}_{\Lambda}(\cD)$ is a complex manifold of dimension $\emph{rk}(\Lambda)$.
\end{thm}

The support property implies that if we fix an element $v \in \Lambda$, then there is a locally-finite set of real codimension one submanifolds with boundary in $\text{Stab}_{\Lambda}(\cD)$, called \emph{walls}, where the set of semistable objects with class $v$ changes. The connected components of the complement in $\text{Stab}_{\Lambda}(\cD)$ of the union of walls for $v$ are called \emph{chambers}.

On $\Stab_{\Lambda}(\cD)$ we have the following group actions:

(i) (Right action of $\widetilde{\mathrm{GL}}^+_2(\R)$) Consider the connected group $\mathrm{GL}^+_2(\R)$ of $2 \times 2$ real matrices with positive determinant. Note that $\mathrm{GL}^+_2(\R)$ acts on the right by multiplication on $\Hom(\Lambda, \C)$ via the identification $\C \cong \R^2$. In order to lift this action to the stability manifold, we consider the universal covering space $\widetilde{\mathrm{GL}}^+_2(\R)$ of $\mathrm{GL}^+_2(\R)$, whose objects are pairs $(M, g)$ with $M \in \mathrm{GL}^+_2(\R)$, $g \colon \R \to \R$ an increasing function satisfying $g(\phi+1)=g(\phi)+1$, such that the induced actions of $M$ and $g$ on $(\R^2 \setminus \lbrace 0 \rbrace) /\R_{>0}= S^1$ are the same. For $\sigma=(\cA, Z) \in \Stab_{\Lambda}(\cD)$ and $(M, g) \in \widetilde{\mathrm{GL}}^+_2(\R)$, we define $\sigma \cdot (M, g)$ as the stability condition with heart $\cP((g(0), g(1)])$ and central charge $Z'=M^{-1} \circ Z$ (see \cite[Lemma 8.2]{Bri}). Concretely, the stability conditions $\sigma$ and $\sigma \cdot (M,g)$ have the same set of semistable objects, but with different phases.

(ii) (Left action of $\Aut_{\Lambda}(\cD)$) Consider the group $\Aut_{\Lambda}(\cD)$ of pairs $(\Phi, \Phi_\Lambda)$, where $\Phi$ is an exact autoequivalence of $\cD$ and $\Phi_\Lambda$ is an endomorphism of $\Lambda$ such that $\Phi_\Lambda \circ \vv = \vv \circ \Phi_*$. Here $\Phi_*$ is the automorphism of $K(\cD)$ induced by $\Phi$. For $(\Phi, \Phi_\Lambda) \in \Aut_{\Lambda}(\cD)$ and $\sigma \in \Stab_{\Lambda}(\cD)$, we define the stability condition $(\Phi, \Phi_\Lambda) \cdot \sigma=(\Phi(\cA), Z \circ \Phi_{\Lambda}^{-1})$. Note that if $\Lambda=\cN(\cD)$, then the endomorphism $\Phi_\Lambda$ is determined uniquely by $\Phi$, hence $\Aut_{\Lambda}(\cD)=\Aut(\cD)$ and one can talk about an action of $\Aut(\cD)$.

Assume that $\sigma$ is a stability condition with moduli spaces. We observe that, by definition, moduli spaces with respect to $\sigma \cdot (M, g)$ and $\Phi \cdot \sigma$ are isomorphic to moduli spaces with respect to $\sigma$.

\begin{ex}
(i) Let $X$ be a smooth projective curve of genus $\geq 1$ and set $\Lambda:=\cN(\Db(X))$. Then
$$\Lambda \cong H^0(X, \Z) \oplus H^2(X, \Z) \cong \Z^{\oplus 2}.$$ 
By \cite{Bri,Macri} the action of $\widetilde{\mathrm{GL}}^+_2(\R)$ is free and transitive, thus there is a unique orbit of numerical stability conditions with respect to the $\widetilde{\mathrm{GL}}^+_2(\R)$-action, i.e.\
\begin{equation} \label{eq:curvesemanuele}
\text{Stab}_{\Lambda}(\Db(X)) \cong \sigma_H \cdot \widetilde{\mathrm{GL}}^+_2(\R)    
\end{equation}
where $\sigma_H$ is the slope stability defined in \autoref{ex:slopestab}.

(ii) (Action of the Serre functor) The Serre functor $\fS_{\cD}$ on $\cD$ (see \autoref{ex:Serre}) defines an element in $\Aut_{\cN(\cD)}(\cD)=\Aut(\cD)$, thus we can consider its action on numerical stability conditions. In \autoref{subsect:introstab2} we will introduce the notion of Serre-invariant stability conditions, which are numerical stability conditions preserved by the Serre functor, up to the $\widetilde{\mathrm{GL}}^+_2(\R)$-action. In fact, this notion plays an important role in a proof of the Categorical Torelli theorem for cubic threefolds and in the study of the geometry of moduli spaces of stable objects in the Kuznetsov component, as we will see in \autoref{subsect:cubic3folds2}.
\end{ex}

\subsection{Tilt stability} \label{sec:tiltstab}
The construction of stability conditions is a difficult task, even in the case of $\Db(X)$ for a smooth projective variety $X$ of dimension $n \geq 3$ over an algebraically closed field $\K$. A conjectural approach is via the notion of tilt stability, whose definition is summarized in this section and which works perfectly for $X$ of dimension $n=2$ (see \cite{BriK3, AB}).

Let $(X,H)$ be a polarized smooth projective variety. We have seen in \autoref{ex:slopestab} that slope stability $\sigma_H=(\Coh(X), Z_H)$ defines a stability condition on $\Db(X)$ when $X$ is a curve. However, in higher dimension this is no longer true, as $Z_H$ vanishes on torsion sheaves supported in codimension $\geq 2$, and in higher dimensions slope stability is only a weak
stability condition. Actually, the choice of $\Coh(X)$ as heart is not the correct one, since by \cite[Lemma 2.7]{Toda} it cannot be the heart of a numerical stability condition if $\dim(X) \geq 2$. 

To overcome this problem, one can consider a new heart by tilting $\Coh(X)$. More precisely, let $\sigma=(\cA, Z)$ be a weak stability condition on a triangulated $\K$-linear category. Fix $s \in \R$ and define the following subcategories of $\cA$:
$$\cT^{s}_{\sigma}:= \lbrace E \in \cA: \text{ all HN factors } F \text{ of } E \text{ satisfy } \mu_Z(F)> s \rbrace,$$
$$\cF^{s}_{\sigma}:= \lbrace E \in \cA: \text{ all HN factors } F \text{ of } E \text{ satisfy } \mu_Z(F) \leq s \rbrace.$$
As proved in \cite{HRS}, the category
\[
\cA^s_{\sigma}:=\langle \cT^s_\sigma, \cF^s_\sigma[1] \rangle
\]
is the heart of a bounded t-structure on $\cT$ and is called the tilting of $\cA$ with respect to $\sigma$ at slope $s$. In this context $\langle-,-\rangle$ means the smallest full subcategory closed under extensions and containing the two additive subcategories $\cT^s_\sigma$ and $\cF^s_\sigma[1]$. Note that if $s > s'$, then we have 
\begin{equation} \label{eq_relationhearts}
\cA_\sigma^s \subset \langle \cA_\sigma^{s'}, \cA_\sigma^{s'}[1] \rangle.    
\end{equation}
Indeed, consider $F \in \cA$ and $\sigma$-semistable with $\mu_Z(F)> s$, which is an object in $\cA_\sigma^s$. Then $\mu_Z(F)>s'$, so $F \in \cA_\sigma^{s'}$. Otherwise, consider $F \in \cA$ and $\sigma$-semistable with $\mu_Z(F) \leq s$, so $F[1] \in \cA_\sigma^{s}$. If $\mu_Z(F) \leq s'$, then $F[1] \in \cA_\sigma^{s'}$, while if $\mu_Z(F) > s'$, then $F[1] \in \cA_\sigma^{s'}[1]$. By the definition of $\cA_\sigma^{s}$, we deduce the desired property. 

In the case of slope stability, we consider the heart
$$ \Coh^s(X):=\Coh(X)^s_{\sigma_H}$$
obtained by tilting $\Coh(X)$ with respect to $\sigma_H$ at slope $s$. In analogy to the curve case, define
\begin{equation}\label{eq:latticeLambda}
\vv \colon K(\Db(X)) \to \Q^3,\quad \vv(E)=(H^n\ch_0(E), H^{n-1}\ch_1(E), H^{n-2}\ch_2(E))    
\end{equation}
where $\ch_2(E)$ denotes the degree-$2$ part of the Chern character of $E$, and set $\Lambda_H:=\text{Im}(\vv)$. For $s, q \in \R$, define $Z_{s, q} \colon \Lambda_H \to \C$ by
$$Z_{s,q}(E)=-(H^{n-2}\ch_2(E) - qH^n\ch_0(E))+ \sqrt{-1}(H^{n-1}\ch_1(E)- sH^n\ch_0(E)).$$
The slope associated to $Z_{s,q}$ is denoted by $\mu_{s,q}$. We can now state the main result of this section.

\begin{thm}[\cite{AB,BMS,BMT,Bri}]
\label{thm:familytiltstability}
Let $X$ be a smooth projective variety of dimension $n \geq 2$. There is a continuous family of weak stability conditions with respect to $\vv$, parametrized by $$\Delta:=\left\{ (s, q) \in \R^2 : q> \frac{1}{2}s^2 \right\},$$ defined as
$$(s, q) \mapsto \sigma_{s, q}=(\emph{Coh}^{s}(X), Z_{s, q}),$$
with a locally-finite wall and chamber structure.

If $n=2$, i.e.\ $X$ is a surface, then $\lbrace \sigma_{s,q} \rbrace$ is a continuous family of stability conditions on $X$.
\end{thm}

The proof of \autoref{thm:familytiltstability} is slightly involved and we refer to \cite[Sections 6.2, 6.3]{MS} for a detailed summary. We only mention that a key ingredient is the classical Bogomolov Inequality, which implies that every $\sigma_H$-semistable sheaf $E$ satisfies
\begin{equation}
\label{eq_BI}
\Delta_H(E):=(H^{n-1}\ch_1(E))^2 -2(H^n\ch_0(E))(H^{n-2}\ch_2(E)) \geq 0.
\end{equation}
Moreover, one can choose $\Delta_H$ as the quadratic form fulfilling the support property for $\sigma_{s,q}$. This also implies that $\sigma_{s,q}$ satisfies well-behaved wall-crossing: for fixed $v \in \Lambda_H$, the boundary of the locus in $\Delta$ where an object of class $v$ is stable is defined by a locally-finite set of submanifolds of real codimension one.

The weak stability conditions $\sigma_{s,q}$ defined in Theorem \ref{thm:familytiltstability} are called \emph{tilt-stability conditions}. A useful way to visualize tilt stability conditions has been introduced in \cite[Section 1]{LZ_birational}. Note that given $E \in \Db(X)$ such that $Z_{s,q}(E) \neq 0$, its truncated at degree $2$ Chern character $$(\ch_0(E), \ch_1(E), \ch_2(E))$$ defines the point
$$(H^n\ch_0(E):H^{n-1}\ch_1(E): H^{n-2}\ch_2(E))$$
in a projective plane $\mathbb{P}^2_{\R}$. If $\ch_0(E) \neq 0$, we consider the affine coordinates
$$\left ( s(E):=\frac{H^{n-1}\ch_1(E)}{H^n\ch_0(E)},\, q(E):= \frac{H^{n-1}\ch_2(E)}{H^n\ch_0(E)} \right ) \in \mathbb{A}^2_{\R}.$$
By \eqref{eq_BI} and \cite[Theorem 3.5]{BMS}, we have that $\sigma_{s,q}$-semistable objects correspond to points below the parabola $q= \frac{1}{2}s^2$ and points above parametrize tilt stability conditions. Furthermore, the phase of a $\sigma_{s,q}$-semistable object $E \in \Coh^s(X)$ is the angle between the line connecting $(s,q)$ and $(s(E), q(E))$ and the vertical half-ray from $(s,q)$ to $-\infty$ divided by $\pi$ (see Figure \ref{fig_0}).

\begin{figure}[htb]
\centering
\begin{tikzpicture}[domain=-6:6]
\draw[->] (-6.3,0) -- (6.3,0) node[right] {$s$};
\draw[->] (0,-2) -- (0, 5.5) node[above] {$q$};
\draw plot (\x,{0.125*\x*\x}) node[above] {$q = \frac{1}{2}s^{2}$};
\coordinate (O) at (0,0);

\coordinate (E) at (4, 0.9);
\node  at (E) {$\bullet$};
\draw (E) node [above right]  {$E=(s(E),q(E))$}; 

\coordinate (F) at (3, -1);
\node  at (F) {$\bullet$};
\draw (F) node [above right]  {$F=(s(F),q(F))$}; 

\node  at (1,1) {$\cdot$};
\draw (1,1) node [above]  {$(s,q)$}; 

\draw (E) -- (1,1);
\draw (F) -- (1,1);
\draw[dashed] (1,1) -- (1, -2);

\coordinate (D) at (-3, 3);
\draw (D) node [above]  {$\Delta$}; 
 
\end{tikzpicture}

\caption{We can compare the $\sigma_{s,q}$-slope of $\sigma_{s,q}$-semistable objects $E$ and $F$, using the picture: $E$ has larger slope than $F$ if and only if the ray connecting $E$ to $(s,q)$ is after the ray connecting $F$ with $(s,q)$ moving from the dashed ray in counterclockwise direction (see \cite[Lemma 2]{LZ_poisson}). \label{fig_0}}   
\end{figure}
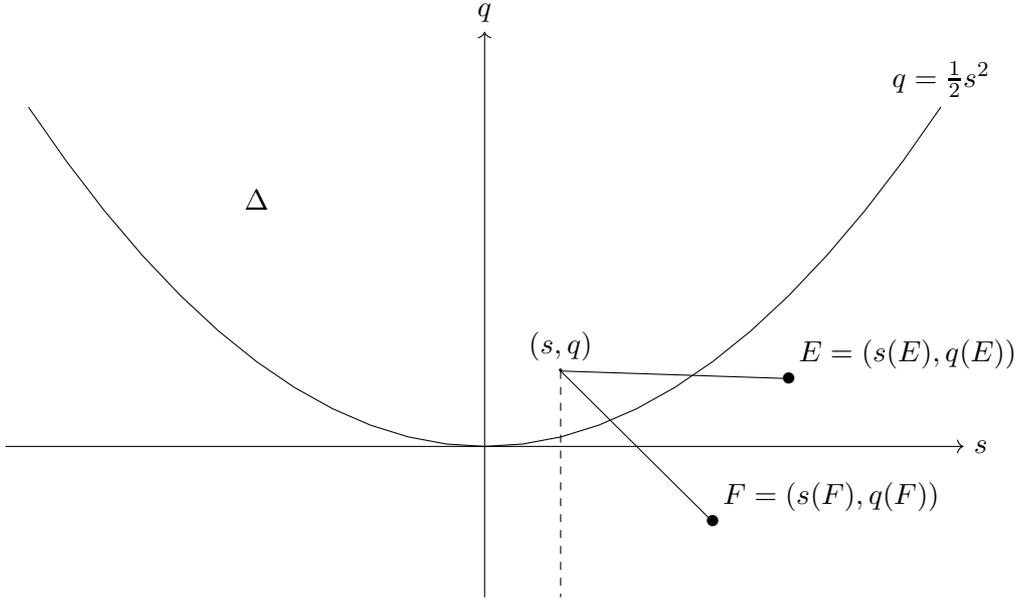

It is now natural to wonder whether the above procedure can be generalized to construct stability conditions when $X$ has dimension $>2$. In this case, it is easy to see that $Z_{s,q}$ does not define a stability function, as it vanishes on objects supported in codimension $\geq 3$. 

Assume $X$ is three-dimensional. In analogy to the surface case, in \cite{BMT} the authors consider a new heart obtained by tilting $\Coh^s(X)$ with respect to the tilt stability $\sigma_{s,q}$ and a weak stability function on it, involving the third Chern character of the objects. The key observation, made in \cite{BMS,BMT}, is that this new weak stability condition is a stability condition on $\Db(X)$ if and only if tilt semistable objects satisfy some quadratic inequality, called generalized Bogomolov inequality, involving the Chern characters till degree $3$. Following this approach, stability conditions have been constructed on the bounded derived categories of Fano threefolds  \cite{BMT,BMSZ,LiFano,Macri2,Schmidt}, abelian threefolds \cite{BMS, MP, MP2}, some
resolutions of finite quotients of abelian threefolds \cite{BMS}, and threefolds with nef tangent bundle \cite{Koseki_nef, BMSZ}. 

A further progress in this direction has been made in \cite{Li}, where a stronger Bogomolov inequality, namely an inequality involving the Chern characters up to degree $2$ of slope semistable torsion free sheaves (see \cite[Theorem 1.1]{Li} for the precise statement), has been proved in the case of quintic threefolds, giving the first highly nontrivial example of stability conditions on a strict (i.e.\ simply-connected) Calabi--Yau threefold. The involved strategy makes use (among other things) of a restriction lemma, first appeared in \cite{Fey}, which allows one to reduce to show a Clifford type bound for the dimensions of global sections of stable vector bundles on curves defined as complete intersections of two quadrics and a quintic hypersurface in $\P^4$. The existence of stability conditions is then obtained by using this stronger Bogomolov inequality to prove the following statement (see \cite[Conjecture 4.1]{BMS}): if $E$ is $\sigma_{s,q}$-semistable for a certain choice of $(s,q) \in \Delta$, then $E$ satisfies 
\begin{equation} \label{eq_conjBMS}
(2q-s^2)\Delta_H(E) + 4(H\ch_2^{s}(E))^2 -6H^2 \ch_1^s(E) \ch_3^s(E) \geq 0,
\end{equation}
where $\ch_1^s(E)=\ch_1(E)-sH\ch_0(E)$, $\ch_2^s(E)=\ch_2(E)- s H \ch_1(E) + \frac{s^2}{2}H^2 \ch_0(E)$, $\ch_3^s(E)=\ch_3(E) -s H \ch_2(E) + \frac{s^2}{2}H^2 \ch_1(E) -\frac{s^3}{6}H^3\ch_0(E)$.

This method has been recently generalized to the case of some three-dimensional weighted hypersurfaces in weighted projective spaces \cite{Koseki}, and to three-dimensional complete intersections of quartic and quadric hypersurfaces in $\P^5$ \cite{SLiu}. An interesting problem would be trying to adapt this argument to treat other examples of Calabi--Yau complete intersections in $\P^N$. For instance, the next case to address could be the following:

\begin{qn}
Let $X$ be a complete intersection of a cubic and two quadric hypersurfaces in $\P^6$. Is it possible to show a stronger Bogomolov inequality for slope semistable torsion free sheaves on $X$, similarly to \cite[Theorem 1.1]{Li}, \cite[Proposition 1.1]{SLiu}, which implies \eqref{eq_conjBMS}?
\end{qn}

We end this section with a sort of negative twist. Indeed, while the possibility to prove the generalized Bogomolov inequality \eqref{eq_conjBMS} is still plausible for threefolds with trivial canonical bundle, we know that \eqref{eq_conjBMS} is false in general. The first counterexample was provided in \cite{SchmidtCounter}. A replacement for this too optimistic guess was recently provided in \cite{BMICM}. In order to state the revised conjecture, we proceed as follows. Let $(X,H)$ be a polarized projective variety of dimension $n$. We denote by $\mathrm{CH}^{\bullet}_\mathrm{num}(X)$ the Chow ring of $X$ modulo numerical equivalence. Let $B\in\mathrm{NS}(X)\otimes\R$, where $\mathrm{NS}(X)=\mathrm{CH}^{1}_\mathrm{num}(X)$, let $\Gamma\in\mathrm{CH}^2_\mathrm{num}(X)\otimes\R$ be such that $\Gamma\cdot H^{n-2}=0$, and pick $\gamma_3,\dots\gamma_n$ to be arbitrary classes with $\gamma_i\in\mathrm{CH}_\mathrm{num}^i(X)\otimes\R$. Let
\begin{equation}\label{eqn:conjnew}
\gamma:=e^{-B}\cdot(1,0,\Gamma,\gamma_3,\dots,\gamma_n) \in \mathrm{CH}_\mathrm{num}^{\bullet}(X)\otimes\R
\end{equation}
and set $\ch^\gamma(-):=\gamma\cdot\ch(-)$. Using the class $\gamma$ and the modified Chern character, one can repeat the discussion above, define $\gamma$-slope semistable objects and thus the notion of $\gamma$-tilt stability $\sigma^\gamma_{s,q}$ for $(s,q) \in \Delta$ (see \cite[Section 4]{BMICM} for more details).

For an object $E\in\Db(X)$ we then set
\[
v^\gamma_H(E):=(H^n\cdot\ch^\gamma_0(E),H^{n-1}\cdot\ch^\gamma_1(E),\dots,H\cdot\ch^\gamma_{n-1}(E),\ch^\gamma_n(E)).
\]
Set $v^\gamma_i$ to be the $i$-th component of $v^\gamma_H(E)$. We can then state the following which is \cite[Conjecture 4.7]{BMICM}:

\begin{conj}[Bayer\textendash Macr\`i]\label{conj_BMICM}
Let $(X,H)$ be a smooth complex projective polarized variety. There exists a class $\gamma$ as in \eqref{eqn:conjnew} and an upper semicontinuous function $f:=f^\gamma\colon\R\to\R$ such that, for any $\sigma_{s,q}^\gamma$-semistable object $E$,
\[
\overline Q_{s,q}(E):=q((v^\gamma_1)^2-2v^\gamma_0v^\gamma_2)+s(3v^\gamma_0v^\gamma_3-v^\gamma_1v^\gamma_2)+(2(v^\gamma_2)^2-3v^\gamma_1v^\gamma_3)\geq 0,
\]
for all $(s,q)\in \Delta$, such that $q>f^\gamma(s)$.
\end{conj}

We do not explain the origin of this conjecture here. We only remark that it is proved for the above mentioned cases (prime Fano threefolds, abelian threefolds, quintic threefolds, complete intersections of quadratic and quartic hypersurfaces in $\P^5$, the blow-up of $\P^3$ at a point, threefolds with nef tangent bundle) which justifies the way it is formulated. For an extensive and very interesting discussion, one can have a look at \cite[Section 4]{BMICM}. The key point is that, if $n=3$, the conjecture implies the existence of stability conditions on $\Db(X)$. Finally, it is clear that the original conjecture in \cite{BMS,BMT} implies \autoref{conj_BMICM} for some particular function $f^\gamma$ while the counterexamples mentioned above do not apply in this case.

\section{Cubic threefolds and beyond}\label{sect:cubic3folds}

As we mentioned in \autoref{subsect:Fano3foldsgeo}, if $X$ is a cubic threefold, we have a semiorthogonal decomposition
\[
\Db(X)=\langle\Ku(X),\OO_X,\OO_X(H)\rangle,
\]
where $H$ is a hyperplane section. In this section we investigate the existence of stability conditions on $\Ku(X)$ and discuss several applications including, of course, the Categorical Torelli theorem for these hypersurfaces.

\subsection{Cubic threefolds: first approach}\label{subsect:cubic3folds1}

In this section we want to recall the approach in \cite{BMMS} to the construction of stability conditions on the Kuznetsov component of a cubic threefold. Even though a much more modern technology is, at the moment, available (see \autoref{subsect:introstab2} and \autoref{subsect:cubic3folds2} below), this circle of ideas have interesting applications to geometric problems and it is an important illustration of what will be proven later about cubic fourfolds.

Assume that $X$ is defined over a field $\K$ which is algebraically closed with $\mathrm{char}(\K)\neq 2$. The starting point is the following observation. Fix a line $\ell_0$ in a cubic threefold $X\hookrightarrow\P^4$ and consider a plane $\P^2\subseteq\P^4$ which is skew with respect to $\ell_0$. The rational map $\pi_0\colon X\dashrightarrow\P^2$ given by the projection from $\ell_0$ can be resolved as follows. Consider the blow-up $\widetilde\P^4$ of $\P^4$ along $\ell_0$ and the strict transform $\widetilde X$ of $X$ inside $\widetilde\P^4$. They all sit in the commutative diagram
\[
\xymatrix{
\widetilde X\ar@{^{(}->}[rr]\ar[dr]_-{\widetilde\pi_0}&&\widetilde\P^4\ar[dl]^-{q}\\
&\P^2,
}
\]
where $q$ is the $\P^2$-bundle induced by the projection from $\ell_0$ and its restriction $\widetilde\pi_0$ to $\widetilde X$ is the corresponding conic fibration. We denote by $h$ the pull-back to $\widetilde X$ of a hyperplane section of $\P^2$, by $H$ the pull-back to $\widetilde X$ of a hyperplane section of $\P^4$ and by $e$ the exceptional divisor.

As a beautiful application of \cite{Kuz2} (see also \cite{ABB} for the case of arbitrary odd characteristic), the conic bundle $\widetilde\pi_0\colon\widetilde X\to \P^2$ yields a sheaf $\cB_0$ on $\P^2$ of even part of Clifford algebras. The explicit description of $\cB_0$ is not needed here but it is worth remembering that we have an isomorphism
\begin{equation}\label{eqn;B0threefolds}
\cB_0\cong\OO_{\P^2}\oplus\OO_{\P^2}(-h)\oplus\OO_{\P^2}(-2h)^{\oplus 2}
\end{equation}
of $\OO_{\P^2}$-modules. Moreover, $\cB_0$ is noncommutative and is an Azumaya algebra
away from the discriminant of the conic bundle. One can further consider the abelian category $\Coh(\P^2,\cB_0)$ of coherent $\cB_0$-modules and the corresponding derived category $\Db(\P^2,\cB_0)$.

It is clear from the discussion above that we can regard $\widetilde X$ both as a blow-up of $X$ and as a conic fibration on $\P^2$. In the first case, \cite{Or:projbun} yields a semiorthogonal decomposition
\begin{equation}\label{eqn:semiorlblow}
\Db(\widetilde X)=\langle\Ku(X),\OO_{\widetilde X}(h-H),\OO_{\widetilde X},\OO_{\widetilde X}(h),\OO_{\widetilde X}(H)\rangle,
\end{equation}
where $\Ku(X)$ is embedded via the pullback along the blowup morphism $\widetilde{X} \to X$. On the other hand, \cite[Theorem 4.2]{Kuz2} provides another semiorthogonal decomposition
\begin{equation}\label{eqn:semikuquad}
\Db(\widetilde X)=\langle\Db(\P^2,\cB_0),\OO_{\widetilde X},\OO_{\widetilde X}(h),\OO_{\widetilde X}(H)\rangle
\end{equation}
(see \cite[(2.7)]{BMMS}). These decompositions are obtained from the standard semiorthogonal decompositions by a sequence of mutations, using the equalities $e=H-h$ and $K_{\widetilde{X}}=-H-h$. We should note here that, in general, it is important to give an explicit description of the way $\Ku(X)$ and $\Db(\P^2,\cB_0)$ are embedded in $\Db(\widetilde X)$ in \eqref{eqn:semiorlblow} and \eqref{eqn:semikuquad} respectively. Here we can safely ignore this and just observe that a direct comparison of the semiorthogonal decompositions in \eqref{eqn:semiorlblow} and \eqref{eqn:semikuquad} yields a new one:
\begin{equation}\label{eqn:P2twist}
\Db(\P^2,\cB_0)=\langle\Ku(X),E\rangle,
\end{equation}
where $E:=\OO_{\widetilde X}(h-H)$ is an additional exceptional object. Again, we are harmlessly ignoring the fact that $\Ku(X)$ is embedded in $\Db(\P^2,\cB_0)$ in a nontrivial way.

Now, the semiorthogonal decomposition \eqref{eqn:P2twist} and the observation that $\Db(\P^2,\cB_0)$ behaves like the derived category of a (twisted) surface are the keys to apply the following general (and as such, at the moment, vague) idea:

\medskip

\noindent{\it Dimension reduction trick:} to study the geometric/homological properties of the Kuznetsov component $\Ku(X)$  of a Fano threefold/fourfold $X$, embed $\Ku(X)$ in the derived category of a smaller dimensional (twisted) variety.

\medskip

This is particularly rewarding when we want to construct stability conditions on $\Ku(X)$. Indeed, as we explained in the previous sections, we have standard strategies to construct stability conditions on surfaces, the situation is less clear but, by now, manageable for threefolds and mostly obscure in higher dimension. We show now how this works for the Kuznetsov component of cubic threefolds. Then in \autoref{subsect:introstab2} we will see that there is a more direct way to contruct stability conditions in the cubic threefold case, without relying on a dimension reduction trick. On the other hand, we will observe later that the dimension reduction trick is crucial while dealing with cubic fourfolds.

Let us now go back to $X$ a cubic threefold. As we observed in the previous section, a standard way to construct (weak) stability conditions on surfaces is by using slope stability and tilt with respect to it.

In the case of the abelian category $\Coh(\P^2,\cB_0)$, the fact that $\cB_0$ is an $\OO_{\P^2}$-algebra (see \eqref{eqn;B0threefolds}), allows us to define a forgetful functor
\[
\mathrm{Forg}\colon\Coh(\P^2,\cB_0)\to\Coh(\P^2)
\]
and thus the functions
\begin{align*}
\rk&:\cN(\P^2,\cB_0)\to\Z,\qquad\rk(F):=\rk(\mathrm{Forg}(F))\\
\deg&:\cN(\P^2,\cB_0)\to\Z,\qquad\deg(F):=\mathrm{c}_1(\mathrm{Forg}(F)) \cdot h,
\end{align*}
where $\cN(\P^2,\cB_0)$ stands for the numerical Grothendieck group of $\Db(\P^2,\cB_0)$.
For $F\in\Coh(\P^2,\cB_0)$ with $\rk(F)\neq0$, we define the slope
\[
\mu_h(F):=\deg(F)/\rk(F).
\]
As usual we can recast $\mu_h$ in terms of a central charge
\[
Z_h(-):=-\deg(-)+\sqrt{-1}\rk(-).
\]
Given $F\in\Coh(\P^2,\cB_0)$, when we say that $F$ is either torsion-free or torsion of dimension $d$, we always mean that $\mathrm{Forg}(F)$ has this property. It is then not difficult to show (see \cite[Section 2.3]{BMMS}) that the pair $\sigma_h=(\Coh(\P^2,\cB_0), Z_h)$ is a weak stability condition on $\Db(\P^2,\cB_0)$. 

Note that if $F \in \Coh(\P^2, \cB_0)$ has rank $0$, then $\mu_h(F)=+\infty$. Given this, we can apply the tilting procedure described in the previous section and define the two full additive subcategories
\begin{align*}
\cT_0:=\cT_{\sigma_h}^{-\frac{5}{4}}&= \left\{ F\in\Coh(\P^2,\cB_0)\,:\,\mu^-_h(F)>\mu_h(\cB_0)=-\frac{5}{4}\right\}\\
\cF_0:=\cF_{\sigma_h}^{-\frac{5}{4}}&= \left\{ F\in\Coh(\P^2,\cB_0)\,:\,\mu^+_h(F)\leq\mu_h(\cB_0)=-\frac{5}{4}\right\}.
\end{align*}
Tilting with respect to the pair $(\cF_0,\cT_0)$ yields the abelian category $\Coh^{-\frac{5}{4}}(\P^2,\cB_0)$, which is the heart of a bounded t-structure. Set
\[
\cA_0:=\Coh^{-\frac{5}{4}}(\P^2,\cB_0)\cap\Ku(X).
\]
It turns out that $\cA_0$ is the heart of a bounded t-structure on $\Ku(X)$ (see \cite[Lemma 3.4]{BMMS}).

\begin{remark}\label{rmk:linesFano3folds}
A useful computation shows that if $\ell\subseteq X$ is any line, then the ideal sheaf $I_\ell$ is contained in $\cA_0$ (see \cite[Proposition 3.3]{BMMS}).
\end{remark}

Consider now the function
\[
Z(-):=\rk(-)+\sqrt{-1}\left(\deg(-)+\frac{5}{4}\rk(-)\right)
\]
on the numerical Grothendieck group $\cN(\P^2,\cB_0)$ of $\Db(\P^2,\cB_0)$. The restriction $Z_0:=Z\rest{\cN(\Ku(X))}$ of $Z$ to the numerical Grothendieck group of $\Ku(X)$ 
is a stability function. For this, see \cite[Lemma 3.5]{BMMS}. This means that it satisfies the condition in \autoref{def:stabfctn}.

These preliminary observations allow us to conclude that we have stability conditions on $\Ku(X)$. Assume that the base field $\K$ has characteristic $0$. Applying the recent results in \cite[Proposition 25.3]{BLMNPS, AH-LH} one can further show that $\sigma_0$ is a stability condition with moduli spaces. For later use, we are also interested in studying special moduli spaces of stable objects. The precise result is the following:

\begin{thm}[\cite{BMMS}, Theorems 3.1 and 4.1]\label{thm:stab1cubic3folds}
In the assumptions above
\begin{enumerate}
    \item[{\rm (1)}] The pair $\sigma_0:=\left(Z_0,\cA_0\right)$ is a stability condition with moduli spaces on $\Ku(X)$.
    \item[{\rm (2)}] Let $\ell$ be any line in $X$. The moduli space $M_{\sigma_0}(\Ku(X),[I_\ell])$ of $\sigma_0$-stable objects in $\cA_0$ and with numerical class $[I_\ell]$ is isomorphic to the Fano surface of lines $F_1(X)$ in $X$.
\end{enumerate}
\end{thm}

\begin{remark}\label{rmk:ulrich1}
The construction of (weak) stability conditions on $\Db(\P^2,\cB_0)$ and on $\Ku(X)$, for $X$ a cubic threefold, was realized in \cite{LMS} in a similar fashion as above, and it has been used to answer some very geometric questions. In particular, in \cite[Theorem B]{LMS} the authors reprove and generalize the main result of \cite{CH} about the nonemptiness of moduli spaces of Ulrich bundles of arbitrary rank on cubic threefolds. Let us recall that an \emph{Ulrich bundle} $E$ is 
an aCM (arithmetically Cohen\textendash Macaulay) bundle whose graded module $\bigoplus_{m\in\Z}H^0(X,E(m))$ has $3\rk(E)$ generators in degree $1$. On the other hand, a vector bundle $E$ on a cubic threefold $X$ is \emph{aCM} if $\mathrm{dim}\; H^i(X,E(j))= 0$, for $i = 1,2$ and all $j\in\Z$. 
We will come back to related questions in \autoref{subsec:cubictthreefolds3} where the results mentioned above are stated in a precise form (see also \cite[Remark 2.20]{LMS} for comments on the used normalization in the definition of Ulrich bundle).
\end{remark}

The result above is the key to prove the following which is indeed \cite[Theorem 1.1]{BMMS}.

\begin{thm}[Categorical Torelli theorem for cubic threefolds]\label{thm;CTTcubic3folds}
Let $X_1$ and $X_2$ be cubic threefolds defined over an algebraically closed field $\K$ with $\mathrm{char}(\K)\neq 2, 3$. Then there exists an exact equivalence $\Ku(X_1)\cong\Ku(X_2)$ if and only if $X_1\cong X_2$.
\end{thm}

\begin{proof}[Idea of proof]
One implication is trivial. Indeed, if $X_1\cong X_2$, then the isomorphism preserves the line bundles $\OO_{X_i}$ and $\OO_{X_i}(H_{X_i})$. Thus we get an induced exact equivalence $\Ku(X_1)\cong\Ku(X_2)$.

Let us now start with an exact equivalence $\fF\colon\Ku(X_1)\to\Ku(X_2)$. Let $\ell\subseteq X_1$ be a line, seen as a point in $F_1(X_1)$. By \autoref{thm:stab1cubic3folds} (2), the object $I_\ell$ is $\sigma_0$-stable in $\Ku(X_1)$. We would like to conclude that $\fF(I_\ell)$ is again the ideal sheaf of a line in $X_2$, but this is not true in general. Nevertheless, by the result of the discussion in \cite[Section 5.1]{BMMS} based on a delicate argument in \cite[Section 4]{BMMS}, we can compose $\fF$ with a suitable power of the Serre functor $\fS_{\Ku(X_2)}$ and a shift, in order to get another exact equivalence $\fF'\colon\Ku(X_1)\to\Ku(X_2)$ such that:
\begin{itemize}
    \item The class $[\fF'(I_\ell)]$ is in $\cN(\Ku(X_2))$ the class of the ideal of a line in $X_2$;
    \item $\fF'(I_\ell)$ is $\sigma_0$-stable and contained in the heart $\cA_0$ in $\Ku(X_2)$.
\end{itemize}
Thus, \autoref{thm:stab1cubic3folds} (2) implies that $\fF'$ induces a bijection between $F_1(X_1)$ and $F_1(X_2)$. By \cite[Section 5.2]{BMMS} such a bijection can be turned into an isomorphism $f\colon F_1(X_1)\to F_1(X_2)$. It was observed in \cite{CG} that $F_1(X_i)$ is a surface of general type whose canonical bundle coincides with the ample class induced by the Pl\"ucker embedding of the Grassmannian of lines in $\P^4$. Thus, the isomorphism $f$ must preserve the induced ample class given by the Pl\"ucker embedding. The trick in \cite[Proposition 4]{Ch} applies and thus $f$ induces an isomorphism $X_1\cong X_2$.
\end{proof}

One interesting consequence is the following.

\begin{cor}\label{cor:eqbecFM}
\autoref{qn:FM2} has a positive answer for cubic threefolds.
\end{cor}

\begin{proof}
Let $X_1$ and $X_2$ be cubic threefolds with an exact equivalence $\fF\colon\Ku(X_1)\to\Ku(X_2)$. By \autoref{thm;CTTcubic3folds}, there is an isomorphism $f\colon X_1\to X_2$ and, as observed in the proof of the result above, $f$ yields an exact equivalence between the Kuznetsov components. For the same reason as in \autoref{rmk:Enrconcl} (i), such an equivalence is of Fourier\textendash Mukai type.
\end{proof}

Let us conclude this section with the following question:

\begin{qn} \label{qn:cubic3foldsdiff}
Is it possible to prove \autoref{thm;CTTcubic3folds} by using the extension techniques that we discussed in \autoref{subsect:extending}?
\end{qn}

\subsection{Inducing stability conditions}\label{subsect:introstab2}

In this section we aim at explaining the methods introduced in \cite{BLMS}, which allow us to induce stability conditions on the Kuznetsov component of prime Fano threefolds defined in \autoref{subsect:Fano3foldsgeo}, so in particular of cubic threefolds, and, with a dimension reduction trick as in \autoref{subsect:cubic3folds1}, of cubic fourfolds. Then we introduce the notion of Serre-invariant stability conditions, whose existence in the context of cubic threefolds will have interesting consequences on the study of moduli spaces.

Let $\cT$ be a proper triangulated category which is linear over an algebraically closed field $\K$ and with Serre functor $\fS_{\cT}$. Assume $\cT$ has an exceptional collection $E_1, \dots, E_m$. Then we have a semiorthogonal decomposition of the form
$$\cT= \langle \cD, E_1, \dots, E_m \rangle$$
where $\cD:= \langle E_1, \dots, E_m \rangle^\perp$. Note that $\rk(\cN(\cD))= \rk(\cN(\cT))-m$ by additivility of the numerical Grothendieck group. This could vaguely suggest that $\cD$ has ``smaller dimension'' than $\cT$, so it might be easier to construct Bridgeland stability conditions on $\cD$ than on $\cT$.

In fact, the following key result provides a criterion which guarantees that a weak stability condition $\sigma$ on $\cT$ restricts to a Bridgeland stability condition on $\cD$.

\begin{prop}[\cite{BLMS}, Proposition 5.1] \label{prop:inducedstab}
Let $\sigma=(\cA, Z)$ be a weak stability condition on $\mathscr{T}$. 
Assume that the exceptional collection $\{E_1,\cdots ,E_m\}$ satisfies the following conditions: 
\begin{enumerate}
\item $E_i\in \cA$;
\item $\fS_{\mathscr{T}}(E_i)\in \cA[1]$; 
\item $Z(E_i)\neq0$ for all $i=1, \cdots, m$;
\item there are no objects $0\neq F\in \cA':=\cA\cap \cD$ with $Z(F)=0$, i.e., $Z':=Z|_{K(\cA')}$ is a stability function on $\cA'$.
\end{enumerate}
Then the pair $\sigma'=(\cA',Z')$ is a Bridgeland stability condition on $\cD$.
\end{prop}

Note that, if the weak stability condition $\sigma$ in the statement is with respect to a lattice $\Lambda$, then $\sigma'$ is defined over the sublattice $\Lambda'$ of $\Lambda$ determined by the image of $K(\cA')$. 

We point out that the conditions (1),(2) are used to show that the restriction $\cA'=\cA \cap \cD$ is the heart of a bounded t-structure on $\cD$, since they imply that the cohomology factors with respect to $\cA$ of an object in $\cD$ belong to $\cD$ as well. By (4) the restriction $Z'$ is a stability function on $\cA'$. For the proof of the Harder\textendash Narasimhan property and the support property the authors make use of the notions of Harder\textendash Narasimhan polygon and mass of an object in the heart. We suggest the interested reader to consult \cite[Section 5]{BLMS} for more details.

\begin{remark}\label{rmk:nostabEnriques}
We should not expect \autoref{prop:inducedstab} to be applicable in all situations where $\Db(X)$ has semiorthogonal decomposition with a Kuznetsov component which is residual to finitely many exceptional objects. One crucial example is when $X$ is an Enriques surface with a semiorthogonal decomposition as in \eqref{eqn:sodEnriques}. Indeed, more should be true: if that $X$ is very general, then $\Ku(X,\cl)$ should not carry Bridgeland stability conditions at all. The heuristic reason is the following. The objects $S_i\in\Ku(X,\cL)$ in \eqref{eqn:sphericalEnriques} are, at the same time, $3$-spherical and numerically trivial. The latter condition implies that the stability function of any stability condition on $\Ku(X,\cL)$ would map the class of $S_i$ to $0$. On the other hand, motivated by the case of K3 surfaces, one would expect that, if a stability condition $\sigma$ on $\Ku(X,\cL)$ exists, then the fact that $S_i$ is spherical implies that it is also $\sigma$-stable. These two facts together would immediately lead to a contradiction and thus to the fact that such a $\sigma$ cannot exist. Of course, it would be very interesting to make the previous argument rigorous and show that $S_i$ is stable in any stability condition on the Kuznetsov component.
\end{remark}

An interesting consequence of \autoref{prop:inducedstab} is the existence of Bridgeland stability conditions on the whole category $\cT$, using results in \cite{CP} about gluings of t-structures.

\begin{prop}[\cite{BLMS}, Proposition 5.13]\label{prop:gluing}
With the assumptions of \autoref{prop:inducedstab}, the pair $\sigma''=(\cA'', Z'')$ on $\cT$, where
$$\cA''= \langle \cA', E_1[1], \dots, E_m[m] \rangle$$
is the extension-closure and 
$$Z''|_{K(\cD)}=Z', \quad Z''(E_i)=(-1)^{i+1} \quad \text{for }i=1, \dots m,$$
is a Bridgeland stability condition on $\cT$.
\end{prop}

When $\cT=\Db(X)$, slope stability and tilt-stability, defined in \autoref{sec:tiltstab}, are weak stability conditions and we could ask whether they satisfy the conditions in \autoref{prop:inducedstab}. We will actually use the following tilting of $\sigma_{s,q}$ defined in Theorem \ref{thm:familytiltstability}. For $\mu \in \R$, let $u$ be the unit vector in the upper half plane with $\mu=- \frac{\Re(u)}{\Im(u)}$. We denote by 
$$\Coh^\mu_{s,q}(X):=\Coh^s(X)^\mu_{\sigma_{s,q}}$$
the heart of a bounded t-structure obtained by tilting $\Coh^s(X)$ with respect to $\sigma_{s,q}$ at $\mu_{s,q}=\mu$. Set 
$$Z^\mu_{s,q}:= \frac{1}{u}Z_{s,q}.$$
If we were working with Bridgeland stability conditions, the above process would correspond to consider $u$ as an element of $\text{GL}_1(\C) \subset \text{GL}_2^+(\R)$, lift it to $\tilde{u} \in \widetilde{\text{GL}}^+_2(\R)$ and act on $\sigma_{s,q}$ by $\tilde{u}$. On the other hand, the $\widetilde{\text{GL}}^+_2(\R)$-action is not a priori well-defined on weak stability conditions, because of the existence of objects with vanishing central charge. For instance, consider an extension $E$ of the form
$$F[1] \to E \to T,$$
where $F$ is $\sigma_{s,q}$-semistable in $\Coh^s(X)$ with $\mu_{s,q}(F) \leq \mu$ and $T \in \Coh(X)$ supported in codimension $>2$. Then $E \in \Coh^\mu_{s,q}(X)$, but $E$ is not in $\Coh^s(X)$ and thus cannot be $\sigma_{s,q}$-semistable. This suggests that there could be objects which are semistable with respect to $\sigma_{s,q}^{\mu}$, but not with respect to $\sigma_{s,q}$ (see \cite[Lemma 5.4]{PetRota} for a concrete example). This prevents us to have a well-defined $\widetilde{\text{GL}}^+_2(\R)$-action on these weak stability conditions. Nevertheless, we have the following result.

\begin{prop}[\cite{BLMS}, Proposition 2.15] \label{prop:doubletiltstab}
Let $X$ be a smooth projective variety. The pair $\sigma^\mu_{s,q}=(\Coh^\mu_{s,q}(X), Z^\mu_{s,q})$ is a weak stability condition on $\Db(X)$ with respect to $\vv$ defined in \eqref{eq:latticeLambda}.
\end{prop}

We end this paragraph by introducing a notion which will be very useful in the case of prime Fano threefolds.

\begin{definition} \label{def_Serreinv}
A (Bridgeland) stability condition $\sigma$ on a triangulated category $\cT$ with Serre functor $\fS_{\cT}$ is \emph{Serre-invariant} (or \emph{$\fS_{\cT}$-invariant}), if there exists $\widetilde{g} \in \widetilde{\text{GL}}^+_2(\R)$ such that
$$\fS_{\cT} \cdot \sigma = \sigma \cdot \widetilde{g}.$$
\end{definition}

\begin{remark}
Note that the property of $\fS_{\cT}$-invariance is stable under the $\widetilde{\text{GL}}^+_2(\R)$-action.
\end{remark}

\begin{remark}
One may wonder why we use the word invariant in the previous definition, which would be more suitable to identify a stability condition which is fixed by the Serre functor. However, recall that stability conditions in the same orbit with respect to the $\widetilde{\text{GL}}^+_2(\R)$-action have the same set of semistable objects. In particular, (if they exist!) the corresponding moduli spaces are isomorphic. From this viewpoint, we are interested in distinguishing or identifying stability conditions, up to the $\widetilde{\text{GL}}^+_2(\R)$-action.
\end{remark}

\begin{ex} \label{ex:curvesSi}
If $X$ is a curve of genus $g(X)$, then the slope stability $\sigma_H$ (see \autoref{ex:slopestab}) is a Serre-invariant stability condition on $\Db(X)$. Actually, if $g(X) \geq 1$, then there is a unique $\widetilde{\text{GL}}^+_2(\R)$-orbit of stability conditions on $\Db(X)$, so the condition in Definition \ref{def_Serreinv} is trivially satisfied. Anyway, the fact that slope stability is Serre-invariant can be checked directly as follows. The Serre functor of $\Db(X)$ is 
$$\fS_{X}(-)= (-) \otimes \omega_X[1].$$
Tensoring by the line bundle $\omega_X$ preserves the slope stability of a coherent sheaf and $\fS_X(\Coh(X))=\Coh(X)[1]$. On the level of central charges, since $(\rk(\omega_X), \deg(\omega_X))=(1, 2g(X)-2)$, we have
$$Z_H \circ (\fS_X^{-1})_*= \deg(-) + (2-2g(X))\rk(-) - \sqrt{-1} \rk(-).$$
Then the matrix $M= 
\begin{pmatrix}
-1 & 2g(X)-2 \\
0 & -1
\end{pmatrix}
$ satisfies $Z_H \circ (\fS_X^{-1})_*=M^{-1} \circ Z_H$. Now note that
$$M \begin{pmatrix}
1 \\
0
\end{pmatrix}= \begin{pmatrix}
-1 \\
0
\end{pmatrix}, \quad M \begin{pmatrix}
0 \\
1
\end{pmatrix}= \begin{pmatrix}
2g(X)-2 \\
-1
\end{pmatrix}.$$
This implies that, if $\cP$ is the slicing defined by $\sigma_H$, then there exists a lifting $(g,M) \in \widetilde{\text{GL}}^+_2(\R)$ of $M$, such that $\cP((g(0), g(1)]) \subset \Coh(X)[1]$ (since $\Coh(X)=\cP((0, 1])$ and $M e^{i\pi\phi} \in \R_{>0}e^{i \pi g(\phi)}$). Since they are both hearts of bounded t-structures, we have the equality $\cP((g(0), g(1)])=\Coh(X)[1]$.
\end{ex}

One could first wonder whether an admissible subcategory $\cD$ of $\Db(X)$ of a smooth projective variety $X$ admits Serre-invariant stability conditions. This has been recently proved to be false in \cite{KP}, using the notion of Serre dimension, in the case of the Kuznetsov component of almost all Fano complete intersections. Anyway, we could focus on prime Fano threefolds and consider the following less general question:

\begin{qn} \label{quest:existeSi}
When does the Kuznetsov component of a prime Fano threefold with index $1$ or $2$ admit Serre-invariant stability conditions?
\end{qn}

We will see in the next section that \autoref{quest:existeSi} has a positive answer for the Kuznetsov components of cubic threefolds and some other prime Fano threefolds. In these cases, we will also see that there is a unique $\widetilde{\text{GL}}^+_2(\R)$-orbit of Serre-invariant stability conditions. Motivated by this, we could also ask the following: 

\begin{qn} \label{quest:uniqueSi}
Assume that $\cD$ is an admissible subcategory of $\Db(X)$ which admits Serre-invariant stability conditions. If $\cN(\cD)$ has rank $2$, is there a unique orbit of Serre-invariant stability conditions with respect to the $\widetilde{\text{GL}}^+_2(\R)$-action?  
\end{qn}

The question above is very much related to several expectations concerning the topology of the stability manifold of $\cD$. In particular, one my wonder whether such a manifold is connected. A positive answer to \autoref{quest:uniqueSi} may give a positive indication in the direction of this expectation as well.

\subsection{Cubic threefolds: a modern view}\label{subsect:cubic3folds2}

Let us now apply the methods introduced in \autoref{subsect:introstab2} to the Kuznetsov component of a cubic threefold $X$ defined over an algebraically closed field $\K$. 

Consider first the slope stability $\sigma_H=(\Coh(X), Z_H)$. Unfortunately, it does not satisfy the assumptions in \autoref{prop:inducedstab} using the exceptional collection $\OO_X$, $\OO_X(H_X)$. Indeed, we see for instance that 
$$\fS_X(\OO_X)=\OO_X(-2H_X)[3] \in \Coh(X)[3].$$
Let us try with the weak stability conditions $\sigma_{s,q}=(\Coh^s(X), Z_{s,q})$ on $\Db(X)$ for $(s,q)$ in the set $\Delta=\left\{ (s, q) \in \R^2 : q> \frac{1}{2}s^2 \right\}$ (see \autoref{sec:tiltstab}). Note that $\OO_X$, $\OO_X(H_X)$ are slope stable, as they are line bundles, with slope $0$ and $1$, respectively. Choosing $s< 0$, we have $\OO_X$, $\OO_X(H_X) \in \Coh^s(X)$. On the other hand, we observe that
if $s< -1$, then $\fS_X(\OO_X(H_X)) = \OO_X(-H_X)[3] \in \Coh^s(X)[3]$, while if $-1 \leq s < 0$, then $\OO_X(-H_X)[3] \in \Coh^s(X)[2]$. For later use, we point out that we have shown:
\begin{equation} 
\label{rotturadiballe}
\text{if } -1 \leq s < 0, \text{ then } \OO_X, \OO_X(H_X), \OO_X(-2H_X)[1], \OO_X(-H_X)[1] \in \Coh^s(X).
\end{equation}

Even if we cannot apply directly \autoref{prop:inducedstab}, this computation suggests that if we tilt another time, we could find a suitable heart. In fact, consider the weak stability condition $\sigma^{\mu}_{s,q}$ defined in \autoref{prop:doubletiltstab}. 

\begin{thm}[{\cite[Theorem 6.8]{BLMS}}]
\label{thm_BLMSstability}
Assume that $(s,q) \in \Delta$ and $q < -\frac{1}{2}s$. Then the pair
\begin{equation*}
\sigma^\mu_{s,q}|_{\Ku(X)}=(\Coh^{\mu}_{s,q}(X) \cap \Ku(X), Z^\mu_{s,q}|_{\cN(\Ku(X))})
\end{equation*}
is a stability condition with moduli spaces on $\Ku(X)$ 
with respect to the lattice $\cN(\Ku(X))$, for every $\mu \in \R$ satisfying 
\begin{equation} \label{eq_valuesmu}
\mu_{s,q}(\OO_X(-H_X)[1]) \leq \mu < \mu_{s,q}(\OO_X).    
\end{equation}
Moreover, for $\mu$, $\mu'$ satisfying \eqref{eq_valuesmu}, the stability conditions $\sigma^\mu_{s,q}|_{\Ku(X)}$, $\sigma^{\mu'}_{s,q}|_{\Ku(X)}$ belong to the same orbit with respect to the $\widetilde{\emph{GL}}^+_2(\R)$-action.
\end{thm}

\begin{figure}[htb]
\centering
\begin{tikzpicture}[domain=-6:6]
\draw[->] (-6.3,0) -- (6.3,0) node[right] {$s$};
\draw[->] (0,-2) -- (0, 5.5) node[above] {$q$};
\draw plot (\x,{0.125*\x*\x}) node[above] {$q = \frac{1}{2}s^{2}$};
\coordinate (O) at (0,0);
\node  at (O) {$\bullet$};
\draw (O) node [above right]  {$\mathcal{O}_{X}$};

\coordinate (OH) at (-4,2);
\node  at (OH) {$\bullet$};
\draw (-4,2) node [above right]  {$\mathcal{O}_{X}(-H_X)$};

\coordinate (OK) at (4,2);
\node  at (OK) {$\bullet$};
\draw (4,2) node [above left]  {$\mathcal{O}_{X}(H_X)$};

\coordinate (OH2) at (-6,4.5);
\node  at (OH2) {$\bullet$};
\draw (-6,4.5) node [above right]  {$\mathcal{O}_{X}(-2H_X)$};

\draw[line width=0.55mm, black] (OH)--(O);
\draw[line width=0.55mm,domain=-4:0,color=black] plot (\x,{0.125*\x*\x});

\node  at (-2,0.8) {$\bullet$};
\draw (-2,0.8) node [below left]  {$(s,q)$};

\draw[dashed] (-2,0.8) -- (-2, -2);

\draw (-2,0.8) -- (O);
\draw[dashed] (OH) -- (-2,0.8);
\draw (-2,0.8) -- (0,-0.4);

\draw[fill=gray!42] (-2,0.8) -- (O) -- (0,-0.4);

\end{tikzpicture}

\caption{The boundary of the region of points $(s,q)$ inducing stability conditions on $\Ku(X)$ is represented in bold. For $(s,q)$ in this region the slope with respect to $\sigma_{s,q}$ of $\OO_X$ is bigger than the slope of $\OO_X(-H_X)[1]$. Note that the sharp angle in grey between the interval connecting the point $(s,q)$ to the point $\OO_X$ and continuation of the interval connecting $\OO_X(-H)$ to $(s,q)$ beyond $(s,q)$ is the region that corresponds to the values of $\mu$ for which $\sigma_{s,q}^\mu|_{\Ku(X)}$ is a stability condition on $\Ku(X)$. \label{fig_1}}  
\end{figure}
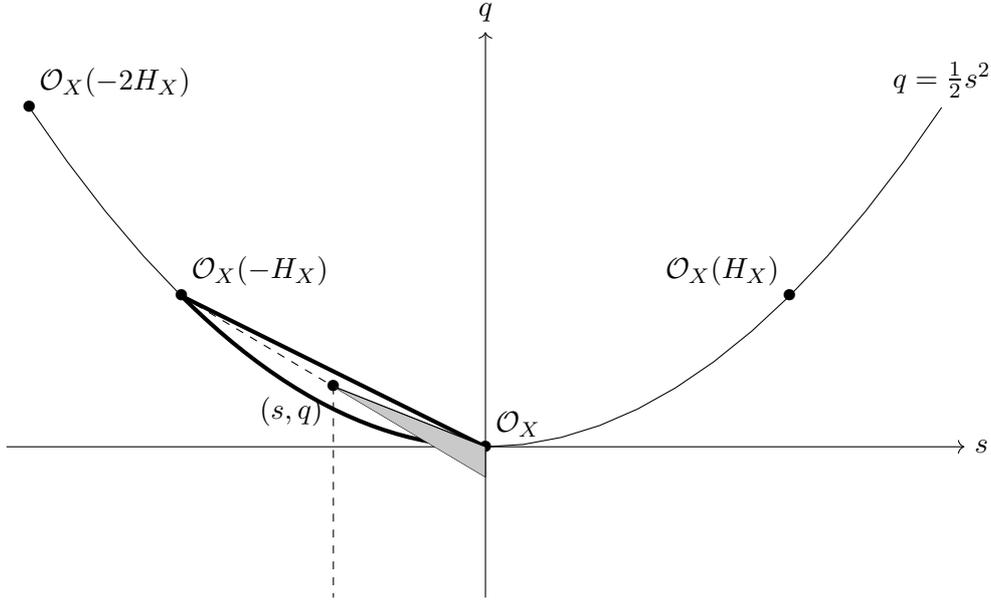

\begin{proof}
By \eqref{rotturadiballe} we have that $\OO_X$, $\OO_X(H_X)$, $\OO_X(-2H_X)[1]$, $\OO_X(-H_X)[1]$ belong to $\Coh^s(X)$ for $-1 \leq s < 0$. Since they define points on the parabola $q=\frac{1}{2}s^2$, by \cite[Corollary 3.11(a)]{BMS}, they are $\sigma_{s,q}$-stable for every $(s,q) \in \Delta$. Comparing the slopes with respect to $\sigma_{s,q}$ via a picture as in Figure \ref{fig_1} (or by a direct computation), we deduce that 
$$\mu_{s,q}(\OO_X(-2H_X)[1]) < \mu_{s,q}(\OO_X(-H_X)[1]) < \mu_{s,q}(\OO_X) < \mu_{s,q}(\OO_X(H_X))$$
for $(s,q)$ as in the statement. Thus if $\mu$ satisfies \eqref{eq_valuesmu}, then $\OO_X$, $\OO_X(H_X)$, $\OO_X(-2H_X)[2]$, $\OO_X(-H_X)[2]$ belong to $\Coh^\mu_{s,q}(X)$. Note that an object $E \in \Coh^\mu_{s,q}(X)$ with $Z_{s,q}(E)=0$ is a torsion sheaf supported in codimension $3$ by \cite[Lemma 2.16]{BLMS}, which is not in $\Ku(X)$ since $\Hom(\OO_X, E) \neq 0$. We can then apply \autoref{prop:inducedstab} which proves that $\sigma_{s,q}^\mu|_{\Ku(X)}$ is a Bridgeland stability condition on $\Ku(X)$ with respect to 
$$\mbox{Im}(K(\Ku(X)) \to \cN(X)) \cong \cN(\Ku(X)) \cong \Z^{2}.$$ 
Then $\sigma_{s,q}^\mu|_{\Ku(X)}$ is a numerical stability condition on $\Ku(X)$ and \cite{AH-LH,BLMNPS} (see, in particular, Proposition 25.3 in the latter paper) implies that it is a stability condition with moduli spaces. We leave to the reader the check of the second part of the statement, using \cite[Lemma 4.3]{BLMS} and considering the rotation between two sides of the sharp angle on Figure \ref{fig_1} which is provided by the segment connecting the point $(s,q)$ to $\OO_X$ and the segment which is the continuation of the one connecting $\OO_X(-H_X)$ to $(s,q)$.
\end{proof}

We denote the induced stability conditions on $\Ku(X)$ by
\begin{equation} \label{eq_stabsigmasq}
\sigma(s,q, \mu)=(\cA(s,q, \mu), Z(s,q, \mu)),    
\end{equation}
where
$$\cA(s,q, \mu):=\Coh^{\mu}_{s,q}(X) \cap \Ku(X), \quad Z(s,q, \mu):=Z^\mu_{s,q}|_{\cN(\Ku(X))}.$$
As we explain below, these stability conditions satisfy the nice property of being Serre-invariant, answering positively \autoref{quest:existeSi} for cubic threefolds.

From now on, we assume $\K=\C$ as in the papers we are referring to, although the same proof should apply in more generality over an algebraically closed field of characteristic $0$. 

\begin{thm}[\cite{PY}, Corollary 5.5] \label{thm_withsong}
Let $\sigma$ be a stability condition on $\Ku(X)$ in the same orbit of $\sigma(s,q,\mu)$ defined in \eqref{eq_stabsigmasq}, with respect to the $\widetilde{\emph{GL}}^+_2(\R)$-action. Then $\sigma$ is Serre-invariant.
\end{thm}
\begin{proof}[Idea of proof]
Since the Serre functor commutes with the $\widetilde{\text{GL}}^+_2(\R)$-action, it is enough to show the statement for $\sigma(-\frac{1}{2},q, -\frac{1}{2})$ with $\frac{1}{8} < q < \frac{1}{4}$. By \cite[Lemma 4.1]{Kuz:V14} the Serre functor of $\Ku(X)$ satisfies the relation
\begin{equation}
\label{eq_Serrefunctor}
\fS_{\Ku(X)}^{-1}=(\L_{\OO_X} \circ \Phi) \circ (\L_{\OO_X} \circ \Phi)[-3],
\end{equation}
where $\Phi(-)=(-) \otimes \OO_X(H_X)$. Thus we can further reduce to prove the statement for the autoequivalence $\L_{\OO_X} \circ \Phi$. Recall that $\L_{\OO_X}$ is the \emph{left mutation in the exceptional object $\OO_X$}. This means that if $\alpha$ denotes the embedding of the admissible subcategory generated by $\OO_X$ into $\Db(X)$, the functor $\L_{\OO_X}$ is defined by the canonical triangle of exact functors
\[
\alpha\alpha^!\to\id_{\Db(X)}\to\L_{\OO_X}
\]
where the first arrow is given by adjunction.

The second step is to show that the heart $\L_{\OO_X}(\Phi(\cA(-\frac{1}{2},q, -\frac{1}{2})))$ is a tilt of $\cA(-\frac{1}{2},q, -\frac{1}{2})$. This is the hardest part in the proof. One key ingredient is \cite[Lemma 3]{LZ_poisson}, which allows us to control the slope of the semistable factors with respect to $\sigma_{s',q'}$ of a $\sigma_{s,q}$-semistable object in $\Coh^s(X)$, when deforming $(s,q)$ to $(s',q')$.

Finally, it is not difficult to show that there exists $M \in \text{GL}^+_2(\R)$ such that $ Z(s,q, -\frac{1}{2}) \circ (\L_{\OO_X} \circ \Phi)_*^{-1}= M^{-1}Z(s,q, -\frac{1}{2})$. Then we can use a similar argument to the one in the proof of the second part of \autoref{thm_BLMSstability}, to get the statement.
\end{proof}

In this case we can also give a positive answer to \autoref{quest:uniqueSi}, as a consequence of the following general criterion, which provides sufficient conditions in order for a fractional Calabi\textendash Yau category to admit a unique $\widetilde{\text{GL}}^+_2(\R)$-orbit of Serre-invariant stability conditions.

\begin{thm}[\cite{FP}, Theorem 3.2] \label{thm_critSi}
Let $\cT$ be a proper $\K$-linear triangulated category over a field $\K$. Assume $\cT$ satisfies the following conditions: 
\begin{enumerate}
	\item\label{C1}Its Serre functor $\fS_{\cT}$ satisfies $\fS_{\cT}^r = [k]$ when $0 < k/r < 2$; 
	\item\label{C2} The numerical Grothendieck group $\cN(\cT)$ is of rank 2 and 
	$$\ell_{\cT} := \max\{\chi(v, v) \colon 0 \neq v \in \cN(\cT) \} < 0,$$
	where $\chi$ is the Euler form defined in \eqref{eq_defeulerform}.
	\item\label{C3} There is an object $Q \in \cT$ satisfying 
	 	\begin{equation}\label{minimal}
	   -\ell_{\cT} +1 \leq  \mathrm{dim}\Ext^1(Q, Q) <  -2\ell_{\cT} +2.
	\end{equation}
\end{enumerate}
Let $\sigma_1$ and $\sigma_2$ be Serre-invariant numerical stability conditions on $\cT$. Then there exists $\tilde{g} \in \widetilde{\emph{GL}}^+_2(\mathbb{R})$ such that $\sigma_1= \sigma_2 \cdot \tilde{g}$. 
\end{thm}

\begin{remark}
Note that \cite[Theorem 3.2]{FP} also deals with the case $r=2$ and $k=4$. In that situation, we need the additional condition that there are two objects $Q_1, Q_2 \in \cT$ satisfying \eqref{minimal} such that $Q_1$ is not isomorphic to $Q_2$ or $Q_2[1]$, $\Hom(Q_2, Q_1) \neq 0$ and $\Hom(Q_1, Q_2[1]) \neq 0$. 
\end{remark}

The assumptions in the statement could seem unnatural at a first glance, but actually they are not. Indeed, in \eqref{C1} we require $\cT$ to be fractional Calabi\textendash Yau of dimension $< 2$, while in \eqref{C2} we require $\cN(\cT)$ to have rank $2$ as a noncommutative curve (although the analogy with curves is not totally correct, as for instance in one case
the numerical Grothendieck group is negative definite, while in the latter it is not, as pointed out by the referee). Condition \eqref{C3} is also inspired by the case of curves: the objects $Q$ and $\fS_{\cT}(Q)$ for $\cT$ play a similar role as the one of the skyscraper sheaves and line bundles on a curve in Macr\`i's proof \cite{Macri}, e.g.\ they are stable with controlled phase.

We now see how to apply the above criterion to $\Ku(X)$ of a cubic threefold.

\begin{cor} \label{cor_uniqueness} 
\autoref{quest:uniqueSi} has a positive answer for the Kuznetsov component $\Ku(X)$ of a cubic threefold $X$.
\end{cor}

\begin{proof}
By  \autoref{prop:cubic3foldseq} we know that $\Ku(X)$ is fractional Calabi\textendash Yau of dimension $5/3 < 2$. By \cite[Proposition  2.7]{BMMS} and \cite{Kuz:Fano3folds}, the numerical Grothendieck group $\cN(\Ku(X))$ is a rank-$2$ lattice and a basis is given by
\begin{equation*}
  \cN(\Ku(X)) = \Z\, [\II_{\ell}] \oplus \Z\, [\fS(\II_{\ell})],
\end{equation*}
where $\II_{\ell}$ is the ideal sheaf of a line $\ell$ in $X$, with respect to which the Euler form $\chi$ is represented by
\begin{equation}
\label{eq_eulerform}
\begin{pmatrix}
 -1 & -1\\0 & -1   
\end{pmatrix}.
\end{equation}
An easy computation shows that $\chi(v,v) \leq -1$ for every $0\neq v\in \cN(\Ku(X))$, thus $\ell_{\Ku(X)}=-1$. Since $\mathrm{dim}\Ext^1(\II_\ell, \II_\ell)=2$, we see that
$$\mathrm{dim}\Ext^1(\II_\ell, \II_\ell)=1- \ell_{\Ku(X)}=2.$$
Thus \autoref{thm_critSi} applies to $\Ku(X)$ and implies the statement.
\end{proof}

In the next section, we will see how to apply these results on Serre-invariant stability conditions to study the properties of moduli spaces of stable objects in $\Ku(X)$ and to reprove the Categorical Torelli theorem.

\subsection{Cubic threefolds: applications} \label{subsec:cubictthreefolds3}

Assume $\K=\C$. The first application of the results presented in the previous section concerns the relation between the stability conditions $\sigma(s,q, \mu)$, defined in \eqref{eq_stabsigmasq}, and the stability condition $\sigma_0$, introduced previously in \autoref{thm:stab1cubic3folds}. Similarly to \autoref{thm_withsong}, we can show the following statement.

\begin{thm}[\cite{FP}, Theorem 5.4] \label{thm_withsoheyla}
The stability condition $\sigma_0$ on $\Ku(X)$ introduced in \autoref{thm:stab1cubic3folds} is Serre-invariant.
\end{thm}

\begin{cor} \label{cor_relationoldandnewstability}
The stability conditions $\sigma_0$ and $\sigma(s,q, \mu)$ defined in \eqref{eq_stabsigmasq} on $\Ku(X)$ are in the same $\widetilde{\emph{GL}}^+_2(\R)$-orbit. 
\end{cor}
\begin{proof}
It is a consequence of \autoref{thm_withsong} and \autoref{thm_withsoheyla} and \autoref{cor_uniqueness}.
\end{proof}

\begin{remark}
Consider the continuous family of Bridgeland stability conditions on $\Db(\P^2,\cB_0)$, parameterized by 
\begin{equation*}
    U := \left\{(b,w) \in \mathbb{R}^2 \colon \ w > \frac{b^2}{2} + \frac{11}{32}  \right\},
\end{equation*}
given by  
\begin{equation*}
(b,w) \in U \mapsto \overline{\sigma}_{b,w} =(\Coh^b(\P^2, \cB_0) ,\ \overline{Z}_{b,w} = -\ch_2 + w \ch_0 +\sqrt{-1} (\ch_1 -b\ch_0)).
\end{equation*}
By using \autoref{prop:inducedstab}, one can show \cite[Proposition 5.3]{FP} that for any $(b,w) \in U$ with $-\frac{5}{4}\leq b < -\frac{3}{4}$, the pair
\begin{equation}\label{bar-1}
\overline{\sigma}(b,w) = (\cA_b:=\Coh^b(\P^2, \cB_0) \cap \Ku(X) ,\ \overline{Z}_{b,w}|_{\cN(\cA_b)})
\end{equation}
is a stability condition on $\Ku(X)$ which is in the $\widetilde{\text{GL}}^+_2(\R)$-orbit of $\sigma_0$ (see Figure \ref{fig_U}).

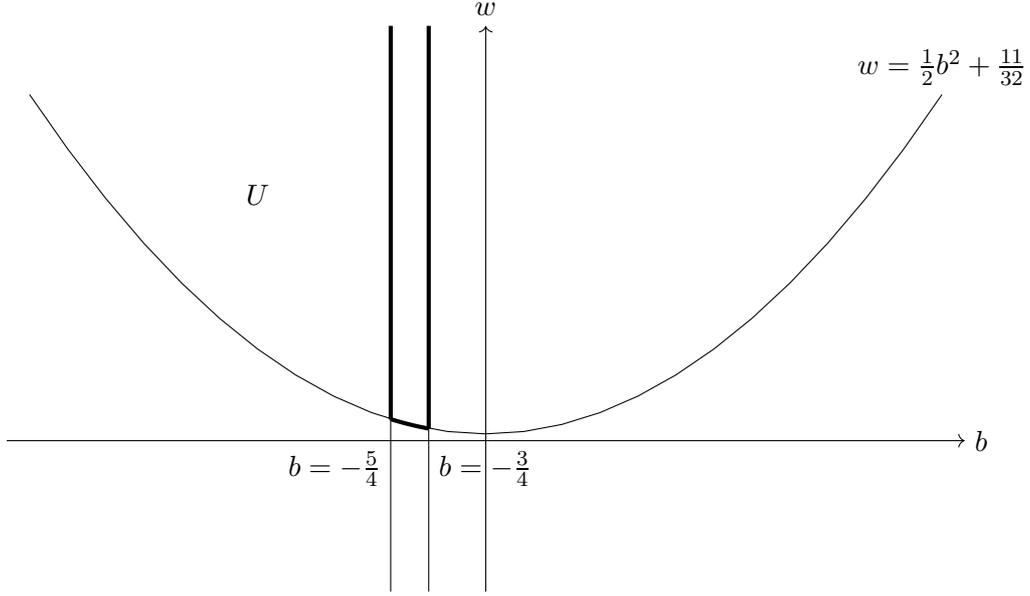
\begin{figure}[htb]
\centering
\begin{tikzpicture}[domain=-6:6]
\draw[->] (-6.3,0) -- (6.3,0) node[right] {$b$};
\draw[->] (0,-2) -- (0, 5.5) node[above] {$w$};
\draw plot (\x,{0.125*\x*\x+0.09}) node[above] {$w = \frac{1}{2}b^{2}+\frac{11}{32}$};
\coordinate (O) at (0,0);

\coordinate (D) at (-3, 3);
\draw (D) node [above]  {$U$}; 
 
\draw (-0.75,-2) -- (-0.75, 5.5);
\draw (-1.25,-2) -- (-1.25, 5.5);

\coordinate (A) at (-0.75, 0);
\draw (A) node [below right] {$b=-\frac{3}{4}$};

\coordinate (B) at (-1.25, 0);
\draw (B) node [below left] {$b=-\frac{5}{4}$};

\draw[line width=0.55mm,domain=-1.25:-0.75,color=black] plot (\x,{0.125*\x*\x+0.09});

\draw[line width=0.55mm, black] (-1.25,0.28) -- (-1.25, 5.5);

\draw[line width=0.55mm, black] (-0.75,0.16) -- (-0.75, 5.5);
\end{tikzpicture}

\caption{We represent the region $U$ parametrizing the weak stability conditions $\overline{\sigma}_{b,w}$, which is the region $w>\frac{1}{2}b^2+\frac{11}{32}$. The strip with bold boundary corresponding to the region above the parabola in the interval $-\frac{5}{4} \leq b < -\frac{3}{4}$ represents those inducing the stability conditions $\overline{\sigma}(b,w)$ on $\Ku(X)$. \label{fig_U}}   
\end{figure}
\end{remark}

An interesting consequence of these results is that all the known stability conditions on $\Ku(X)$ are Serre-invariant, hence belong to the same $\widetilde{\text{GL}}^+_2(\R)$-orbit. A positive answer to the following (hard) question would complete the analogy of $\Ku(X)$ with the curve case in \cite{Macri}.

\begin{qn}
Let $\sigma$ be a stability condition on $\Ku(X)$ of a cubic threefold $X$. Is $\sigma$ in the same orbit of $\sigma(s,q, \mu)$? Equivalently, is $\sigma$ Serre-invariant? 
\end{qn}

The second application, which makes use of \autoref{cor_relationoldandnewstability}, concerns the study of Ulrich bundles on $X$, whose definition is recalled in \autoref{rmk:ulrich1}. Note that an Ulrich bundle $E$ belongs to $\Ku(X)$. Indeed, by \cite[Lemma 2.19]{LMS} we have
$$\ch(E)=\left(d, 0, -\frac{d}{3}H_X^2, 0\right)=d\ch(I_\ell),$$
where $d$ is the rank of $E$. Note that Ulrich bundles are Gieseker semistable \cite[Proposition 2.8]{CH}. The Gieseker stability of $E$ implies that $\Hom(\OO_X(kH_X), E) = 0$ for $k \geq 0$. Since $\chi(\OO_X(kH_X), E) = 0$ for $k =0, 1$, the definition of Ulrich bundles implies that $\Hom(\OO_X(kH_X), E[3]) = 0$. We conclude that $E \in \Ku(X)$.

By using the embedding of $\Ku(X)$ in $\Db(\P^2, \cB_0)$, the following existence result has been proved.

\begin{thm}[\cite{LMS}, Theorem B]
The moduli space of stable Ulrich bundles of rank $d \geq 2$ on $X$ is non-empty and smooth of dimension $d^2+1$. 
\end{thm}

The property of Serre-invariance of $\sigma_0$ and $\sigma(s,q, \mu)$ and the uniqueness result allows us to prove the irreducibility of the moduli space of Ulrich bundles. 

\begin{thm}[\cite{FP}, Theorem 6.1]
\label{thm_ulrich}
The moduli space $\mathfrak{M}^{U}_{d}$ of Ulrich bundles of rank $d \geq 2$ on $X$ is irreducible.
\end{thm}

\begin{proof}[Idea of proof]
The first step consists in proving that all Ulrich bundles $E$ of rank $d$ are $\sigma(s,q,\mu)$-semistable. Since the condition of being Ulrich is open, we get an open embedding 
$$\mathfrak{M}^{U}_{d} \hookrightarrow M_{\sigma(s,q, \mu)}(\Ku(X), d[I_{\ell}]),$$
where the latter is the moduli space of $\sigma(s,q,\mu)$-semistable objects in $\Ku(X)$ with numerical class $d[I_{\ell}]$ (see \autoref{subsect:def}). Now using that $\sigma(s,q,\mu)$ is Serre-invariant and that $\sigma_0$ is in the same $\widetilde{\text{GL}}^+_2(\R)$-orbit of $\sigma(s,q,\mu)$, we have the isomorphisms between moduli spaces 
$$M_{\sigma(s,q,\mu)}(\Ku(X), d[I_{\ell}]) \cong M_{\sigma(s,q,\mu)}(\Ku(X), d[\fS_{\Ku(X)}(I_{\ell})]) \cong M_{\sigma_0}(\Ku(X), dv),$$
where $v:=[\fS_{\Ku(X)}(I_{\ell})]$ in $\cN(\Ku(X)) \subset \cN(\Db(\P^2, \cB_0))$.

The last step is to show that $M_{\sigma_0}(\Ku(X), dv)$ is identified to the moduli space $\mathfrak{M}_d$ of Gieseker semistable sheaves in $\Coh(\P^2, \cB_0)$ with class $dv$. A standard argument (see \cite[Theorem 2.12]{LMS}, \cite[Lemma 4.1]{KLS}) whose key point is the vanishing of $\Ext^2$ for sheaves in this moduli space, allows us to show that $\mathfrak{M}_d$ is irreducible, and thus $\mathfrak{M}^{U}_{d}$ is such (see \cite[Proposition 6.4]{FP} for more details).
\end{proof}

Let us now discuss some interesting properties of moduli spaces of stable objects in $\Ku(X)$ with respect to a Serre-invariant stability condition $\sigma$, e.g.\ for $\sigma(s,q,\mu)$ and $\sigma_0$. We have already seen in \autoref{thm:stab1cubic3folds} that the Fano surface of lines is identified with a moduli space of $\sigma_0$-stable objects in $\Ku(X)$. We mention the following more general statement.

\begin{thm}[\cite{FP}, Theorem 4.5, \cite{PY}, Lemma 5.16]\label{thm-moduli-spaces}
Let $\sigma$ be a Serre-invariant stability condition on $\Ku(X)$.  
\begin{enumerate}
    \item The moduli space $M_{\sigma}(\Ku(X), [I_{\ell}])$ is isomorphic to the moduli space of slope-stable sheaves on $X$ with Chern character $\ch(I_{\ell}) = 1 -\frac{H_X^2}{3}$.
    \item The moduli space $M_{\sigma}(\Ku(X), [\fS_{\Ku(X)}(I_{\ell})])$ is isomorphic to the moduli space of slope-stable sheaves on $X$ with Chern character $2 -H_X-\frac{H_X^2}{6}+ \frac{H_X^3}{6}$.
    \item The moduli space $M_{\sigma}(\Ku(X), [\fS_{\Ku(X)}^2(I_{\ell})])$ is isomorphic to the moduli space of  large volume limit stable complexes of character $\ch(\II_\ell)-\ch(\fS_{\Ku(X)}(I_{\ell}))$ \footnote{Inspired by \cite{Toda_bogomolov} and to simplify the exposition, we give the following definition. A two-term complex $E \in \Db(X)$ supported in degree $0$ and $-1$ is said to be large volume limit stable if $\HH^{-1}(E)$ is a line bundle, $\HH^0(E)$ is a sheaf supported in dim $\leq 1$, and $\Hom(T, E) = 0$ for every torsion sheaf $T$ supported in dimension $\leq 1$. By \cite[Lemma 3.12 and Lemma 3.13(ii)]{Toda_bogomolov}, a complex $E \in \Db(X)$ is large volume limit stable if and only if $E$ lies in $\Coh^{s}(X)$ and is $\sigma_{s,q}$-stable for $s > \mu_H(E)$ and $q \gg 0$. The condition on the rank of $\HH^{-1}(E)$ could be relaxed, but we restrict to this case which is the one relevant to our discussion.}.
\end{enumerate}
The moduli spaces $M_{\sigma}(\Ku(X), [I_{\ell}])$, $M_{\sigma}(\Ku(X), [\fS_{\Ku(X)}(I_{\ell})])$, $M_{\sigma}(\Ku(X), [\fS_{\Ku(X)}^2(I_{\ell})])$ are isomorphic to the Fano surface of lines $F_1(X)$ in $X$.
\end{thm}

Note that $\pm [I_\ell]$, $\pm [\fS_{\Ku(X)}(I_{\ell})]$, $\pm [I_\ell]-[\fS_{\Ku(X)}(I_{\ell})]$ are the only vectors of square $-1$ in $\cN(\Ku(X))$. We also have the following result.

\begin{thm}
Let $\sigma$ be a Serre-invariant stability condition on $\Ku(X)$. Then non-empty moduli spaces of $\sigma$-stable objects in $\Ku(X)$ are smooth.
\end{thm}

\begin{proof}
Let $E$ be a $\sigma$-stable object in $\Ku(X)$. Up to shifting, we can assume that $E$ is in the heart of $\sigma$. Since $\sigma$ is Serre-invariant, we have that $\fS_{\Ku(X)}(E)$ is $\sigma$-stable. By using the relation $\fS_{\Ku(X)}^3=[5]$, it is not difficult to show (see \cite[Lemma 5.9]{PY}) that the phases with respect to $\sigma$ satisfy
\begin{equation} \label{eq_fasi}
\phi(E)< \phi(\fS_{\Ku(X)}(E)) < \phi(E)+2.    
\end{equation}
As a consequence, by Serre duality we have
\[
\Ext^i(E,E)=\Hom(E, E[i]) \cong \Hom(E[i], \fS_{\Ku(X)}(E))=0,
\]
for $i \geq 2$, $i<0$. Indeed, the objects $E[i]$, $\fS_{\Ku(X)}(E)$ are $\sigma$-stable and if $i \geq 2$, then $\phi(E[i]) > \phi(\fS_{\Ku(X)}(E))$ by \eqref{eq_fasi}. Thus there are no nontrivial morphisms from $E[i]$ to $\fS_{\Ku(X)}(E)$. The vanishing for $i<0$ follows from $E$ being in the heart. Since $E$ is stable, we have that $\mathrm{dim}\Hom(E,E)=1$, so $\mathrm{dim}\Ext^1(E,E)=1-\chi(E,E)$ is constant. Since $\Ext^1(E,E)$ is identified with the tangent space to the moduli space at the point corresponding to $E$ and the obstruction is given by $\Ext^2(E,E)$ (see \cite{Lieblich}, \cite[Lemma 4.8]{PeHodge}) which is vanishing, we conclude that the moduli space $M_\sigma(\Ku(X),[E])$, where $[E] \in \cN(\Ku(X))$, is smooth of dimension $1-\chi(E,E)$ at the point corresponding to $E$.
\end{proof}

Last but not least, we focus on the Categorical Torelli theorem. Using the same strategy as in \autoref{thm;CTTcubic3folds}, it is possible to reprove this theorem by using any Serre-invariant stability condition and by applying the modular description of the Fano surface of lines by \autoref{thm-moduli-spaces} (see \cite[Theorem 5.17]{PY}). We end this section by explaining an alternative beautiful proof, given in \cite{soheylaetal}, which is based on a description of the desingularization of the theta divisor of the intermediate Jacobian of $X$ as a moduli space of stable sheaves on $X$, and then of semistable objects in $\Ku(X)$. 

Recall that the intermediate Jacobian of $X$ is the complex torus
$$J(X):= H^1(X, \Omega^2_X) / H_3(X, \Z).$$
It has the structure of a principally polarized abelian variety of dimension $5$ and plays a key role in the seminal paper \cite{CG} for the proof of the (classical) Torelli Theorem and the nonrationality of $X$. If $T$ denotes the closure in the Hilbert scheme of the subvariety parametrizing twisted cubic curves in $X$, then we have the Abel\textendash Jacobi map $\alpha \colon T \to J(X)$ defined by $\alpha(t)=t-H_X^2$ (which is an element of $J(X)$ through the cycle class map). Then by \cite[Proposition 4.2]{Beau} the image of $\alpha$ in $J(X)$ is a theta divisor $\Theta \subset J(X)$ and its generic fiber is isomorphic to $\P^2$. In fact, the linear span of a twisted cubic $C$ is $\langle C \rangle\cong \P^3$, so $C$ is contained in the cubic surface $S=\langle C \rangle \cap X$ for $C$ general. Then the generic fiber of $\alpha$ is the $\P^2$ of twisted cubic curves which are linearly equivalent to $C$ on $S$.

Let $M_X(v')$ be the moduli space of Gieseker stable sheaves on $X$ with Chern character $3 -H -\frac{1}{2}H^2+ \frac{1}{6}H^3$, hence having numerical class $v':=[I_\ell]+ [\fS_{\Ku(X)}(I_\ell)]$. Note that the projection $K_x$ in $\Ku(X)$ of the skyscraper sheaf of a point $x \in X$ has numerical class $v'$ and it turns out to be $\sigma$-stable with respect to any Serre-invariant stability condition $\sigma$ on $\Ku(X)$, which motivates the choice of this particular moduli space to recover the isomorphism class of $X$. Consider the Abel\textendash Jacobi morphism $\beta \colon M_X(v') \to J(X)$ defined by $\beta(E)=\tilde{c}_2(E)-H_X^2$, where $\tilde{c}_2(E)$ corresponds to the second Chern class $c_2(E)$ of $E$ via the cycle map. The key ingredient is the following result.

\begin{thm}[\cite{soheylaetal}, Theorem 7.1] \label{thm_soheyla1}
The moduli space $M_X(v')$ is smooth and irreducible of dimension $4$. The Abel-Jacobi morphism $\beta$ maps $M_X(v')$ birationally onto the theta divisor $\Theta$. More precisely, $M_X(v')$ is the blow-up of $\Theta$ in its singular point. The exceptional divisor is isomorphic to the cubic threefold $X$, and parametrizes non-locally free sheaves in $M_X(v')$.
\end{thm}
The proof of the above theorem makes use of classical results on the properties of $\Theta$ together with techniques of wall-crossing to describe the objects in $M_X(v')$. In particular, the embedding of $X$ in $M_X(v')$ is obtained by sending a point $x \in X$ to the projection $K_x$ in $\Ku(X)$ of the corresponding skyscraper sheaf. The objects $K_x$ sweep a divisor in this moduli space. Also note that the Abel--Jacobi map $\alpha$ factors through $\beta$.

The second key ingredient is the identification of $M_X(v')$ with a moduli space of semistable objects in $\Ku(X)$.

\begin{thm}[\cite{soheylaetal}, Theorem 8.7] \label{thm_soheyla2}
Let $\sigma$ be a Serre-invariant stability condition on $\Ku(X)$. Then we have the isomorphism
$$M_\sigma(\Ku(X), v') \cong M_X(v').$$
\end{thm}

\begin{proof}[Idea of proof]
Since $\sigma$ is Serre-invariant, the Serre functor $\fS_{\Ku(X)}$ induces an isomorphism
\[
M_\sigma(\Ku(X), 2[I_\ell]-[\fS_{\Ku(X)}(I_\ell])) \cong M_\sigma(\Ku(X), v'), \quad E \mapsto \fS_{\Ku(X)}(E),
\]
where $2[I_\ell]-[\fS_{\Ku(X)}(I_\ell]=\fS^{-1}_{\Ku(X)}(v')$. The moduli space $M_\sigma(\Ku(X), 2[I_\ell]-[\fS_{\Ku(X)}(I_\ell]))$ has been fully described in \cite[Theorem 1.2]{APR} (see also \cite[Proposition 8.5]{soheylaetal}). The proof then follows by computing the image via $\fS_{\Ku(X)}$ of the objects.
\end{proof}

As a consequence, the authors get a new proof of the Categorical Torelli \autoref{thm;CTTcubic3folds} (see \cite[Theorem 8.1]{soheylaetal}):

\begin{proof}[New proof of \autoref{thm;CTTcubic3folds}]
As usual, there is only one nontrivial implication to prove. Assume that there is an exact equivalence $\Phi \colon \Ku(X_1) \isomor \Ku(X_2)$ (not necessarily of Fourier--Mukai type). Up to composing with a power of the Serre functor and shifting, we can assume $\Phi_*(v')=v'$, namely that $\Phi$ maps the numerical class of the projection $K_{x_1}$ of the skyscraper sheaf of a point $x_1 \in X_1$ to the numerical class of $K_{x_2}$ for a point $x_2 \in X_2$, using that the Serre functor acts transitively on all classes that have the same square as $v'$ with respect to $\chi$ (see \cite[Lemma 8.3]{soheylaetal}). Since $\sigma$ is Serre invariant and the Serre functor commutes with equivalences, we have that $\Phi \cdot \sigma$ is Serre invariant. \autoref{thm_soheyla2} applied two times gives the isomorphisms between moduli spaces
$$M_{X_1}(v') \cong M_\sigma(\Ku(X_1), v') \cong M_{\Phi \cdot \sigma}(\Ku(X_2), v') \cong M_{X_2}(v').$$
By \cite[Lemma 7.5]{soheylaetal} we have that $X_1$ and $X_2$ are the unique rationally connected divisors in $M_{X_1}(v')$ and $M_{X_2}(v')$, respectively, which get contracted by any morphism to a complex abelian variety. Thus the above chain of isomorphisms implies $X_1 \cong X_2$.
\end{proof}

\subsection{More prime Fano threefolds}\label{subsect:moreFano}

After the success of the above techniques in the case of cubic threefolds, it is natural to wonder whether we can get similar results for other prime Fano threefolds of index $1$ and $2$ (see \autoref{subsect:Fano3foldsgeo} for the classification and properties). The goal of this section is to summarize the state of art in these cases concerning the existence and uniqueness of Serre-invariant stability conditions and the Categorical Torelli theorem. We assume the base field is $\K= \C$ for simplicity.

Let us start by considering the index $2$ case. We denote by $Y_d$ a prime Fano threefold of index $2$ and degree $1 \leq d \leq 5$ \footnote{Note that these prime Fano threefolds of index $2$ were denoted by $X_d$ in \autoref{subsect:Fano3foldsgeo}. The sudden change of notation is motivated by the fact that soon we will need discuss the relation between prime Fano threefolds of index $1$ and $2$.} and take its Kuznetsov component $\Ku(Y_d):= \langle \OO_{Y_d}, \OO_{Y_d}(H_{Y_d}) \rangle^{\perp}$. The strategy of \cite{BLMS}, reviewed in \autoref{subsect:cubic3folds2} for cubic threefolds, works more generally for $Y_d$ and allows one to construct stability conditions on $\Ku(Y_d)$, which we denote by $\sigma(s,q,\mu)$. We have the following results:
\begin{enumerate}
\item By \cite[Proposition 5.7]{PY} the stability conditions $\sigma(s,q,\mu)$ on $\Ku(Y_d)$ are Serre-invariant for every $1 \leq d \leq 5$.
\item By \cite[Theorem 4.25]{JLLZ} (see also \cite[Remark 3.7]{FP}) there is a unique $\widetilde{\text{GL}}^+_2(\R)$-orbit of Serre-invariant stability conditions on $\Ku(Y_d)$ for every $2 \leq d \leq 5$.
\item The Categorical Torelli theorem holds for general $Y_2$ by \cite[Theorem 1.3]{APR}, and for every $Y_2$ under the assumption that the equivalence is of Fourier\textendash Mukai type by \cite[Corollary 6.1]{BT}. Note that by \cite{LPZStrong} every equivalence is of Fourier\textendash Mukai type in this case.
\item The Categorical Torelli theorem holds for $Y_4$ by \cite{BO}, and for $Y_5$ since it is unique up to isomorphism by \cite{Isk}.
\end{enumerate}
The case $d=1$ remains mysterious. Recall that $Y_1$ is a hypersurface of degree $6$ in the weighted projective space $\P(1,1,1,2,3)$, known as the Veronese double cone. The Serre functor of $\Ku(Y_1)$ satisfies $\fS_{\Ku(Y_1)}^3=[7]$ by \cite[Corollary 4.2]{Kuz_CY}. As a consequence, we cannot apply the criterion of \autoref{thm_critSi} and the homological dimension of the heart of a Serre-invariant stability condition is $3$. This makes this case quite different from the others. A remarkable difference is that this time the Fano surfaces of lines is an irreducible component of a moduli space of stable objects in $\Ku(Y_1)$ and in \cite{PetRota} the authors classify all the objects in this moduli space. We can then formulate the following questions:
\begin{qn}
Is there a unique $\widetilde{\text{GL}}^+_2(\R)$-orbit of Serre-invariant stability conditions on $\Ku(Y_1)$?
\end{qn}

\begin{qn}
Does the Categorical Torelli theorem holds for $Y_1$?
\end{qn}

We now focus on the index $1$ case and denote by $X_d$ such a prime Fano threefold of degree $d=2g-2$, where $2 \leq g \leq 12$, $g \neq 11$. We first consider $g \geq 6$, so that the Kuznetsov component is defined by $\Ku(X_d):=\langle \EE_r, \OO_{X_d} \rangle^\perp$, where $\EE_r$ is an exceptional vector bundle of rank $r$ on $X_d$ (see Section \ref{subsect:Fano3foldsgeo}). Again by \cite{BLMS}, there are stability conditions $\sigma(s,q,\mu)$ on $\Ku(X_d)$ defined following the same procedure explained in Section \ref{subsect:cubic3folds2} for cubic threefolds. In these cases we have the following results:
\begin{enumerate}
\item By \cite[Theorem 3.18]{PR} the stability conditions $\sigma(s,q,\mu)$ on $\Ku(X_d)$ are Serre-invariant for every $d=10, 14, 18, 22$. Among all, the most interesting case is $d=10$, i.e.\ when $X_d$ is a Gushel\textendash Mukai threefold.
\item Denote by $\MM^i_d$ the moduli space of Fano threefolds of index $i$ and degree $d$ for $i=1,2$. By \cite[Theorem 3.3]{Kuz:Fano3folds}, for $d=3, 4, 5$ there is a correspondence $\ZZ_d \subset \MM^2_d \times \MM^1_{4d+2}$, which we call Kuznetsov's correspondence, dominant over each factor, such that for every point $(Y_d, X_{4d+2}) \in \ZZ_d$, there is an equivalence 
$$\Ku(Y_d) \cong \Ku(X_{4d+2}).$$
By this equivalence and the results for $Y_3$, $Y_4$, $Y_5$, we deduce that $\Ku(X_d)$ has a unique orbit of Serre-invariant stability conditions for $d=14, 18, 22$. Moreover, by \cite[Theorem 4.25]{JLLZ} (see also \cite[Corollary 4.5]{PR}) this uniqueness result holds also for $\Ku(X_{10})$. In the remaining cases, i.e.\ $d=12, 16$, we have $\Ku(X_d)$ is equivalent to the bounded derived category of a curve of genus $7$ and $3$, respectively. In particular, by \eqref{eq:curvesemanuele} the stability manifold of $\Ku(X_d)$ is identified with the $\widetilde{\text{GL}}^+_2(\R)$-orbit of the slope stability on the curve. Thus there is a unique orbit of Serre-invariant stability conditions and actually every stability condition $\sigma$ is Serre-invariant, using that the action of any autoequivalence commutes with the action of $\widetilde{\text{GL}}^+_2(\R)$. We conclude that for every $d \geq 10$, there exists a non-empty and unique orbit of Serre-invariant stability conditions. 
\item The Categorical Torelli theorem does not hold for $X_d$ in the form stated above. For instance, by \cite{KP_cones} it is known that there are Gushel\textendash Mukai threefolds with equivalent Kuznetsov components but which are not isomorphic (only birational). Nevertheless, a refined version of the Categorical Torelli theorem has been proved in \cite{JLLZ} for Gushel\textendash Mukai threefolds and more generally in \cite{JLZ} for every $X_d$ with $d \geq 10$, making use (among all) of the existence and uniqueness of Serre-invariant stability conditions. The precise statement is the following. Consider $X_d$ and $X_d'$ two prime Fano threefolds of index $1$ and same degree $d \geq 10$. Denote  by $i^!$ and $i^{'!}$ the right adjoints of the embedding functors $\Ku(X_d) \subset \Db(X_d)$ and $\Ku(X_d') \subset \Db(X_d')$, respectively. Assume there is an equivalence $\Phi \colon \Ku(X_d) \cong \Ku(X_d')$ such that 
$$\Phi(i^!(\EE_r)) \cong i^{'!}(\EE_r).$$ 
Then $X_d \cong X_d'$. In other words, (the other direction is trivial) the knowledge of $i^{!}(\EE_r)$ is necessary and sufficient to reconstruct $X_d$ from $\Ku(X_d)$ for $d \in \lbrace 10, 12, 14, 16, 18, 22 \rbrace$.
\end{enumerate}

As noted in the introduction to \cite{JLLZ}, the Categorical Torelli theorem for cubic threefolds and the Kuznetsov's correspondence, discussed in item (2) above, imply that the birational class of $X_{14}$ is determined by its Kuznetsov component. Inspired by \cite[Conjecture 1.7]{KP_cones}, we can formulate the following question, which is nothing but the Birational Categorical Torelli theorem for Gushel\textendash Mukai threefolds:

\begin{qn}[\cite{JLLZ}, Question 1.1(1)]
Let $X_1$ and $X_2$ be prime Fano threefolds of index $1$ and degree $10$. Assume that there is an equivalence $\Phi \colon \Ku(X_1) \isomor \Ku(X_2)$ (of Fourier\textendash Mukai type). Is it true that $X_1$ and $X_2$ are birational?
\end{qn}

As noted by the referee $X_{12}$, $X_{16}$, $X_{18}$, $X_{22}$ are rational (over an algebraically closed field), so the above question is void in these cases. Furthermore, this question has been positively answered in \cite[Theorem 1.5]{JLLZ} for general ordinary Gushel\textendash Mukai threefolds.

Let us consider the remaining Fano threefolds $X_d$ with $d=2, 4, 6, 8$, where $\Ku(X_d):= \langle \OO_{X_d} \rangle^\perp$. Even if the construction in \cite{BLMS} still applies and there are stability conditions $\sigma(s,q,\mu)$ on $\Ku(X_d)$, the situation is far less understood. Note that $\cN(\Ku(X_d))$ has rank $3$ in this case, so the stability manifold of $\Ku(X_d)$ has a bigger dimension than in the previous cases and $\Ku(X_d)$ does not look like a noncommutative curve. Up to now, we know the following results:
\begin{enumerate}
\item By \cite[Remark 1.2]{JLZ} the Fano threefold $X_d$ is a moduli space of stable objects with respect to the stability conditions $\sigma(s,q,\mu)$ on $\Ku(X_d)$. As one could expect, the objects in the moduli spaces are shifts of ideal sheaves of points in $X_d$.
\item If $d=6$, then $X_6$ is a smooth intersection of a quadric and a cubic hypersurfaces. Then by \cite[Corollary 1.9]{KP} there does not exist Serre-invariant stability conditions on $\Ku(X_6)$.
\item If $d=4$, then $X_4$ is either a quartic threefold (see \autoref{rmk:fractCY}) or a double cover of a quadric hypersurface $Q$ in $\P^4$ ramified in the intersection of $Q$ with a quartic. Assume $X_4$ belongs to the second class. Using the recent results in \cite{KP}, we can compute explicitly the Serre functor of $\Ku(X_4)$ and show that the induced stability conditions $\sigma(s,q,\mu)$ cannot be Serre-invariant (see \cite{HMPS}).
\end{enumerate}

We can focus on the case of quartic threefolds and ask the following questions:
\begin{qn}
If $X$ is a quartic threefold, are the stability conditions $\sigma(s,q,\mu)$ Serre-invariant for some $(s,q)$?
\end{qn}

\begin{qn} \label{qn_catTorquartics}
Is there a version of the Categorical Torelli theorem for quartic threefolds? Does the birational Categorical Torelli theorem hold for quartic threefolds?
\end{qn}

In these notes we have seen many methods to show the Categorical Torelli theorem for the Kuznetsov components of Enriques surfaces and cubic threefolds. The first tentative to approach \autoref{qn_catTorquartics} could be trying to adapt one of these techniques to the case of quartic threefolds. This is the line of investigation we are following in the work in progress \cite{HMPS}.

\section{Cubic fourfolds} \label{sec:cubic4folds}

Let us now move back to the case when $X$ is a cubic fourfold defined over a field $\K$ which is algebraically closed with $\mathrm{char}(\K)\neq 2$. Let us recall from \autoref{subsect:cubic4folds} that we have a semiorthogonal decomposition
\begin{equation*}\label{eqn:sodcub4folds}
\Db(X)=\langle\Ku(X),\OO_X,\OO_X(H),\OO_X(2H)\rangle,
\end{equation*}
where $H$ is a hyperplane class. The aim of this section is to show that $\Ku(X)$ carries stability conditions which we will use to prove both a categorical and a classical Torelli theorem for these Fano fourfolds. We conclude this section with a brief discussion on Gushel\textendash Mukai fourfolds (see \autoref{sec_GM4}) which, quite surprisingly, are at the same time very close in spirit to cubic fourfolds but very different for some key features.

\subsection{Stability conditions on the Kuznetsov component}\label{subsect:stabKuz4folds}

In this section we want to prove that $\Ku(X)$ carries stability conditions. Ideally, we would like to apply the techniques discussed in \autoref{subsect:introstab2} but what prevents us from a successful output is the fact that $X$ has dimension $4$ while the inducing strategy works perfectly fine for threefolds.

Thus the idea is to embed $\Ku(X)$ into the derived category of a new threefold whose construction is intimately related to the geometry of $X$ as in \autoref{subsect:cubic3folds1}.

More precisely, we follow \cite[Section 7]{BLMS}. Now, as in the case of cubic threefolds, we pick a line $\ell\subseteq X$ which is not contained in a plane in $X$ and we consider the projection $\pi_\ell\colon X\dashrightarrow\P^3$ onto a skew $3$-dimensional projective plane. We further consider the blow-up $\widetilde X:=\mathrm{Bl}_\ell(X)$ of $X$ in $\ell$ which makes the rational map $\pi_\ell$ into an actual morphism
\[
\widetilde\pi_\ell\colon\widetilde X\to\P^3
\]
whose fibers are conics.

On the categorical side, the conic fibration $\widetilde\pi_\ell$ yields a sheaf $\BB_0$ of even parts of Clifford algebras which is analogous to the one considered in \autoref{subsect:cubic3folds1}
(in particular, it is noncommutative and generically Azumaya). As in  \autoref{subsect:cubic3folds1} the fact that we can either view $\widetilde X$ as a conic fibration or as a blow-up, gives two fully faithful embeddings
\[
\Ku(X)\hookrightarrow\Db(X)\hookrightarrow\Db(\widetilde X)\qquad\Db(\P^3,\BB_0)\hookrightarrow\Db(\widetilde X).
\]
A direct comparison between the two inclusions shows that the embedding of the Kuznetsov component can be
made compatible by mutations with the one of $\Db(\P^3,\BB_0)$. Hence, according to \cite[Proposition 7.7]{BLMS}, we get a semiorthogonal decomposition
\begin{equation}\label{eqn:sodcubci4foldsn1}
\Db(\P^3,\BB_0)=\langle\Ku(X),E_1,E_2,E_3\rangle,
\end{equation}
where $E_i$ is an exceptional $\BB_0$-coherent sheaf. As in \eqref{eqn:P2twist}, we omit the explicit description of the embedding $\Ku(X)\hookrightarrow\Db(\P^3,\BB_0)$ which is indeed relevant for computations but useless for the purposes of this paper.

\begin{remark}\label{rmk:diff4folds3folds}
Despite the analogy between \eqref{eqn:P2twist} for cubic threefolds and \eqref{eqn:sodcubci4foldsn1} for cubic fourfolds, the complexity of the twisted projective space $(\P^3,\BB_0)$ and of its derived category has secretly increased a lot. Indeed, while the numerical Grothendieck group of $\Db(\P^2,\BB_0)$ has always rank $3$ for all cubic threefolds (see \cite[Proposition 2.12]{BMMS}), the rank of the numerical Grothendieck group of $\Db(\P^3,\BB_0)$ varies when the cubic fourfold varies.

This is, indeed, not surprising because the Kuznetsov component behaves like a noncommutative curve (with a rather simple cohomology) for cubic threefolds while it behaves like a noncommutative K3 surface (hence with a rich cohomology) for cubic fourfolds.
\end{remark}

In order to construct stability conditions on $\Ku(X)$ we are now in a good position to apply the dimension reduction trick described in \autoref{subsect:cubic3folds1}: indeed, $\Ku(X)$ is now an admissible subcategory of the derived category of a (twisted) threefold with residue category generated by three exceptional objects.

To make this precise, we proceed as in the cubic threefold case and take the forgetful functor $\mathrm{Forg}\colon\Db(\P^3,\cB_0)\to\Db(\P^3)$. The \emph{twisted Chern character} is then defined as
\[
\mathrm{ch}_{\BB_0} (-) := \mathrm{ch}(\mathrm{Forg}(-)) \left( 1 - \frac{11}{32}L\right),
\]
where $L$ is the class of a line in $\P^3$. We denote by $\mathrm{ch}_{\BB_0,i}$ the degree $i$ component of $\mathrm{ch}_{\cB_0}$. Since this is a cohomology class of $\P^3$, it gets naturally identified with a rational number.

\begin{remark}\label{rmk:11/32}
The mysterious numerical correction $\frac{11}{32}$ is needed to provide a Bogomolov inequality for the twisted derived category $\Db(\P^3,\BB_0)$. The role of such an inequality is to give the correct quadratic inequality in the support property for the Kuznetsov component. This is similar to \eqref{eq_BI}. For the purposes of this paper the numerical correction above can be ignored.
\end{remark}

Next, we define
\[
\vv' \colon K(\Db(\P^3, \BB_0)) \to \Q^3,\qquad \vv'(E):=\left(\mathrm{ch}_{\BB_0,0}(E),\mathrm{ch}_{\BB_0,1}(E),\mathrm{ch}_{\BB_0,2}(E)\right)\in\Q^{\oplus 3}
\]
and consider the lattice $\Lambda_{\BB_0}:=\text{Im}(\vv')$. Finally, we denote by $\Lambda_{\BB_0,\Ku(X)}$ the image of the composition of the natural maps $K(\Ku(X))\to K(\Db(\P^3,\BB^0))\to\Lambda_{\BB_0}$. The main result is then the following.

\begin{thm}[\cite{BLMS}, Theorem 1.2, \cite{BLMNPS}, Proposition 25.3]\label{thm:stab4folds}
If $X$ is a cubic fourfold, then there are stability conditions with moduli spaces on $\Ku(X)$ with respect to $\Lambda_{\BB_0,\Ku(X)}$.
\end{thm}

It is worth pointing out that the result is actually more precise. Following Bridgeland's notation, a \emph{full numerical stability condition} on $\Ku(X)$ is a stability condition on $\Ku(X)$ with respect to the lattice $\widetilde{H}_\mathrm{Hodge}(\Ku(X),\Z) \cong \cN(\Ku(X))$ whose definition is recalled in Section \ref{subsect:cubic4folds} (see \autoref{rmk:cubicHodgeconj}). For later use, a Mukai vector is an element in the image of the restriction $\vv\colon K(\Ku(X))\to \widetilde{H}_\mathrm{Hodge}(\Ku(X),\Z)$ of $\vv'$. We denote by $\mathrm{Stab}^m(\Ku(X))$ the set of full numerical stability conditions with moduli spaces on $\Ku(X)$. Note that the twisted Chern character induces a natural map $u\colon\widetilde{H}_\mathrm{Hodge}(\Ku(X),\Z)\to \Lambda_{\BB_0,\Ku(X)}$.

\begin{ex}\label{ex:fullstab}
It is explained in \cite[Section 9]{BLMS}, that if $\sigma=(\AA,Z)$ is a stability condition constructed in \autoref{thm:stab4folds}, then the pair $\sigma_{\Ku(X)}:=(\AA,Z\circ u)$ is in $\mathrm{Stab}^m(\Ku(X))$ which is then nonempty.
\end{ex}

\autoref{thm:stab4folds} and \autoref{prop:gluing} yield the following

\begin{cor}\label{cor:stabcub}
Let $X$ be a cubic fourfold. Then $\Db(X)$ has a stability condition.
\end{cor}

Furthermore, we consider the natural continuous map
\[
\eta\colon\mathrm{Stab}^m(\Ku(X))\to\widetilde{H}_\mathrm{Hodge}(\Ku(X),\C)
\]
defined in the following way. First, using the pairing on $\widetilde H(\Ku(X),\Z))$ (see \autoref{rmk:cubicK3}), we get a natural identification between $\Hom(\widetilde{H}_\mathrm{Hodge}(\Ku(X),\Z),\C)$ and the vector space $\widetilde{H}_\mathrm{Hodge}(\Ku(X),\C)=\widetilde{H}_\mathrm{Hodge}(\Ku(X),\Z)\otimes_\Z\C$. Then $\eta$ is nothing but the continuous map $\mathcal{Z}$ in \autoref{thm:BrDefo} composed with such an identification.

Set now $\mathcal{P}\subseteq\widetilde{H}_\mathrm{Hodge}(\Ku(X),\C)$ to be the open subset consisting of vectors whose
real and imaginary parts span positive-definite two-planes in $\widetilde{H}_\mathrm{Hodge}(\Ku(X),\R)$. Then we set
\[
\mathcal{P}_0:=\mathcal{P}\setminus\bigcup_{\delta\in\Delta}\delta^\perp,
\]
where $\Delta:=\{\delta\in\widetilde{H}_\mathrm{Hodge}(\Ku(X),\Z):(\delta,\delta)=-2\}$. Note that for $\sigma_{\Ku(X)}$ as in \autoref{ex:fullstab} we have, by \cite[Proposition 9.10]{BLMS},
\begin{equation}\label{eqn:imageeta}
\eta(\sigma_{\Ku(X)})\in(A_2\otimes\C)\cap\mathcal{P}\subseteq\mathcal{P}_0.
\end{equation}
Here $A_2$ is the lattice in \eqref{eqn:A2}. Let $\mathcal{P}_0^+$ be the connected component of $\mathcal{P}_0$ which contains $\eta(\sigma_{\Ku(X)})$, for $\sigma_{\Ku(X)}$ as in \autoref{ex:fullstab}. In addition, let $\mathrm{Stab}^\dagger(\Ku(X))$ be the connected component of $\mathrm{Stab}^m(\Ku(X))$ which contains $\sigma_{\Ku(X)}$.

The complete result is then the following.

\begin{thm}[\cite{BLMNPS}, Theorem 29.1 and \cite{BLMS}, Proposition 9.9]\label{thm:stab4folds2}
The preimage $\eta^{-1}(\cP_0^+)$ contains the connected component $\mathrm{Stab}^\dagger(\Ku(X))$.
Moreover, the restriction $\eta\colon\mathrm{Stab}^\dagger(\Ku(W))\to\cP_0^+$ is a covering map.
\end{thm}

\subsection{Categorical Torelli theorem}\label{subsect:catThm4folds}

In this section we want to discuss the following result which is one of the main results of \cite{HR}.

\begin{thm}[Categorical Torelli theorem for cubic fourfolds]\label{thm:catTT4folds}
Let $X_1$ and $X_2$ be cubic fourfolds. Then $X_1\cong X_2$ if and only if there is an equivalence $\Phi\colon\Ku(X_1)\isomor\Ku(X_2)$ such that $\fO_{X_2}\circ\Phi=\Phi\circ\fO_{X_1}$.
\end{thm}

Here $\fO_{X_i}$ is the degree shift autoequivalence of $\Ku(X_i)$  described in \autoref{rmk:autoeqKuz}. Note that we do not require $\Phi$ being of Fourier--Mukai type. The assumption concerning the compatibility between the equivalence $\Phi$ and the degree shift functors $\fO_{X_i}$ is crucial. Indeed, it was proved in \cite[Theorem 1.1]{PFMCub}, that given any positive integer $N$ one can find $N$ non-isomorphic cubic fourfolds with equivalent Kuznetsov components. This is another striking similarity with the case of K3 surfaces (see, for example, \cite{St}). It is then natural to raise the following question (see \cite[Question 3.25]{MSLectNotes} and the discussion therein):

\begin{qn}[Huybrechts] \label{qn_birCatTorcubic4}
Does a Birational Categorical Torelli theorem for cubic fourfolds hold? Namely, let $X_1$ and $X_2$ be cubic fourfolds. Is it true that the existence of an equivalence $\Ku(X_1)\isomor\Ku(X_2)$ implies that $X_1$ and $X_2$ are birational?
\end{qn}
Note that \autoref{qn_birCatTorcubic4} is compatible with \autoref{conj:Kuzration}.

\begin{remark}\label{rmk_onlyonedirection}
Note that the reverse implication of \autoref{qn_birCatTorcubic4} is not true. More precisely, this amounts to asking whether two birational cubic fourfolds $X_1$ and $X_2$ have equivalent Kuznetsov components. However, as suggested by the referee, two general Pfaffian
cubic fourfolds are birational (as both are rational), but their Kuznetsov components are not equivalent.
\end{remark}

Of course, one could continue the analysis of the analogies with K3 surfaces. Assume that $\K=\C$. In this case, one knows that the derived categories of two K3 surfaces are equivalent if and only if there is an orientation preserving Hodge isometry of the Mukai lattices of the two surfaces (see \cite{Or, HMS:K3}). It is then natural to ask whether the same happens for the Kuznetsov components of cubic fourfolds (see \cite[Question 3.24]{MSLectNotes}).

We will skip this discussion and, in the rest of this section, we will deal with two proofs of \autoref{thm:catTT4folds} based on two different approaches.

\subsubsection*{Idea of proof 1 (Jacobian rings)}

Let us first explain the original approach in \cite{HR} which is close in spirit to the one in \cite{Donagi}. For this we have to stick to the case $\K=\C$.

Let us first introduce the main ingredients in the proof. If $Y$ is a smooth hypersurface in $\P^{n+1}$ described as the zero locus of a homogeneous polynomial $F$, then the \emph{Jacobian ring} of $Y$ is
\[
\mathrm{Jac}(Y):=\C[x_0,\dots,x_{n+1}]/(\partial_iF).
\]
If the degree $d$ of $Y$ is such that $d\leq n+1$ (i.e.\ $Y$ is a Fano manifold), then there is a semiorthogonal decomposition
\[
\Db(Y)=\langle\Ku(Y),\OO_Y,\dots,\OO_Y((n+1-d)H)\rangle,
\]
where $H$ is a hyperplane section. Let $\Ku(Y)(-(n+1-d))$ be the admissible subcategory of $\Db(Y)$ obtained by tensoring $\Ku(Y)$ by $\OO_Y(-(n+1-d)H)$.

Denote by $\Ku(Y)(-(n+1-d))\boxtimes\Ku(Y)$ the full subcategory of $\Db(Y\times Y)$ which is generated by objects of the form $E_1\boxtimes E_2$, where $E_1\in\Ku(Y)(-(n+1-d))$ and $E_2\in\Ku(Y)$\footnote{This means that $\Ku(Y)(-(n+1-d))\boxtimes\Ku(Y)$ is the smallest full triangulated subcategory of $\Db(Y\times Y)$ containing the objects $E_1\boxtimes E_2$ as above and which is closed under taking direct summands.}. Note that it is admissible by \cite[Theorem 5.8]{Kuz11}. Denote by
\[
j_Y\colon\Ku(Y)(-(n+1-d))\boxtimes\Ku(Y)\hookrightarrow\Db(Y\times Y)
\]
its fully faithful embedding and set $P_0:=j^*_Y\OO_\Delta$, where $j^*_Y$ is the left adjoint of $j_Y$. Similarly, for a given $n\geq 1$, we set $P_n$ to be the Fourier\textendash Mukai kernel of the Fourier\textendash Mukai functor obtained by composing $\fO_Y$ with itself $n$ times. It worth pointing out that $\fO_Y$ is a Fourier\textendash Mukai functor with Fourier\textendash Mukai kernel given by the convolution of $P_0$ and $\OO_\Delta(1)$ \cite[Remark 1.10]{HR}. By construction $P_n\in \Ku(Y)(-(n+1-d))\boxtimes\Ku(Y)$. Note that $P_0$ and $P_n$ are related by convolution, namely $P_1 \cong P_0 \circ \OO_\Delta(1) \circ P_0$ and $P_n \cong P_1^{\circ n}$ by \cite[Remark 1.11]{HR}. 

Now, assume that $d>2$, set $N=(n+2)(d-2)$ and
\[
L(Y):=\bigoplus_{i=0}^N\RHom(P_0,P_i).
\]
This comes with the ring structure induced by the composition (see \cite[Section 3.1]{HR}). Indeed, applying the convolution with $P_i$, we have a natural map $\RHom(P_0,P_j) \to \RHom(P_i, P_{i+j})$ and thus
$$\RHom(P_0, P_i) \times \RHom(P_0, P_j) \to \RHom(P_0, P_i) \times \RHom(P_i, P_{i+j}) \to \RHom(P_0, P_{i+j})$$
requiring that the multiplication is trivial if $i+j>N$. Standard arguments show that this endows $L(Y)$ with the structure of a graded ring. We denote by $L_j(Y)$ the graded piece of degree $j$. We set $\mathrm{HH}^*(\Ku(Y),\fO_Y)$ to be the graded subalgebra of $L(Y)$ generated by $L_1(Y)$\footnote{The notation $\mathrm{HH}^*(\Ku(Y),\fO_Y)$ clearly suggests a relation with the Hochschild cohomology of the category $\Ku(Y)$ as defined in \cite{KuzHoch}. Indeed, as explained in \cite{HR}, we should think of $\mathrm{HH}^*(\Ku(Y),\fO_Y)$ as the Hochschild cohomology of the pair $(\Ku(Y),\fO_Y)$. Even though the relation between the Hochschild cohomology of the pair $(\Ku(Y),\fO_Y)$ and the usual $\mathrm{HH}^*(\Ku(Y))$ is not of interest for this paper, the reader may find a discussion about this in \cite[Section 3]{HR}.}.

We are now ready to relate $\mathrm{Jac}(Y)$ and $\mathrm{HH}^*(\Ku(Y),\fO_Y)$.

\begin{thm}[\cite{HR}, Theorem 1.1]\label{thm:HR1}
Let $Y\subseteq\P^{n+1}$ be a smooth hypersurface of degree $d\leq\frac{n+2}{2}$. Then there exists a natural surjection
\[
\pi_Y\colon\mathrm{Jac}(Y)\to\mathrm{HH}^*(\Ku(Y),\fO_Y)
\]
of graded rings which is an isomorphisms if $n+2$ is divisible by $d$.
\end{thm}

Now let $Y_1$ and $Y_2$ be smooth hypersurfaces in $\P^{n+1}$ of degree $2\leq d\leq\frac{n+2}{2}$ and assume that $d$ divides $n+2$ so that $\pi_{Y_i}$ is an isomorphism, for $i=1,2$. If $\Phi\colon\Ku(Y_1)\isomor\Ku(Y_2)$ is a Fourier\textendash Mukai equivalence such that $\fO_{Y_2}\circ\Phi=\Phi\circ\fO_{Y_1}$, by \cite[Proposition 3.9]{HR}, $\Phi$ induces an isomorphism of graded algebras $\mathrm{HH}^*(\Ku(Y_1),\fO_{Y_1})\cong\mathrm{HH}^*(\Ku(Y_2),\fO_{Y_2})$. By \autoref{thm:HR1}, this lifts to an isomorphism of graded algebras $\mathrm{Jac}(Y_1)\cong \mathrm{Jac}(Y_2)$. By Yau\textendash Mather theorem (see \cite[Proposition 1.1]{Donagi}, we get $Y_1\cong Y_2$. Thus we proved:

\begin{cor}\label{cor:catTThyp}
Let $Y_1$ and $Y_2$ be smooth hypersurfaces in $\P^{n+1}$ of degree $2\leq d\leq\frac{n+2}{2}$ and such that $d$ divides $n+2$. Then $Y_1\cong Y_2$ if and only if there is an equivalence of Fourier\textendash Mukai type $\Phi\colon\Ku(Y_1)\isomor\Ku(Y_2)$ such that $\fO_{Y_2}\circ\Phi=\Phi\circ\fO_{Y_1}$
\end{cor}

Since cubic fourfolds satisfy the assumptions of \autoref{cor:catTThyp}, we immediately get \autoref{thm:catTT4folds} in view of \autoref{prop:cubic4foldKu1}.

\subsubsection*{Idea of proof 2 (stability conditions)}

We now outline the strategy of proof via stability conditions which is pursued in the appendix to \cite{BLMS} and, as explained later, allows one to get another proof of the classical Torelli theorem as well.

In this second proof, we can assume that $\K$ is an algebraically closed field with $\mathrm{char}(\K)\neq 2$. As above, there is only one implication in the statement that needs to be proved.

The strategy is very close in spirit to the one in \autoref{subsect:cubic3folds1} and \autoref{subsect:cubic3folds2} where we proved the Categorical Torelli theorem for cubic threefolds. Indeed, let $X$ be a cubic fourfold and let $F_1(X)$ be the Fano varieties of lines in $X$. By \cite{BD}, when $\K=\C$, the variety $F_1(X)$ is a $4$-dimensional smooth and projective irreducible symplectic manifold (i.e.\ a simply-connected manifold such that $H^0(F_1(X),\Omega_{F_1(X)}^2)$ is generated by an everywhere nondegenerate holomorphic $2$-form). In general, as $F_1(X)$ is embedded in the Grassmannian of lines in $\P^5$, it carries a privileged ample polarization which is the restriction of the Pl\"uker polarization. To shorten the notation we will refer to such a polarization on $F_1(X)$ as the Pl\"ucker polarization.

The key point is that one can interpret $F_1(X)$ as a moduli space of stable objects in the Kuznetsov component $\Ku(X)$. This approach was initiated in \cite{MSt} and  pursued in \cite{BLMS,LPZ}. Let $X$ be a cubic fourfold and fix a stability condition $\sigma\in\mathrm{Stab}^\dagger(\Ku(X))$ such that $\eta(\sigma)\in (A_2\otimes\C)\cap\mathcal{P}\subseteq\mathcal{P}_0^+$. As we observed in \eqref{eqn:imageeta} any stability condition constructed in the proof of \autoref{thm:stab4folds} would work. By the general theory of moduli spaces of (semi)stable objects in the Kuznetsov component $\Ku(X)$ which we discussed in \autoref{subsect:def}, we can take the moduli space $M_\sigma(\Ku(X),v)$ for every Mukai vector $v$. By the results in \cite{BLMNPS}, any such moduli space $M_\sigma(\Ku(X),v)$ carries a natural ample polarization $\ell_\sigma$.

The result is then the following (recall the class $\llambda_1$ defined in \eqref{eqn:lambdas} and that $\delta\in\widetilde{H}_\mathrm{Hodge}(\Ku(X),\Z)$ such that $(\delta,\delta)=-2$ is called $(-2)$-class):

\begin{thm}[\cite{BLMS}, Theorem A.8]\label{thm:Fano}
Let $X$ be a cubic fourfold such that $\wH_\mathrm{Hodge}(\Ku(X),\Z)$ does not contain $(-2)$-classes. For any $\sigma \in \Stab^\dag(\Ku(X))$ such that $\eta(\sigma)\in(A_2\otimes\C)\cap\mathcal{P}\subseteq\mathcal{P}_0^+$, the Fano variety of lines in $X$ is isomorphic to the moduli space $M_\sigma(\Ku(X),\llambda_1)$ of $\sigma$-stable objects with numerical class $\llambda_1$. Moreover, the ample line bundle $\ell_\sigma$ on $M_\sigma(\Ku(X),\llambda_1)$ is identified with a multiple of the Pl\"ucker polarization by this isomorphism.
\end{thm}

\begin{remark}
In \cite[Theorem 1.1]{LPZ} the previous result has been generalized by showing that the Fano variety of lines of every cubic fourfold is isomorphic to a moduli space of $\sigma$-stable objects for $\sigma$ as in \autoref{subsect:stabKuz4folds}.
\end{remark}

Now, let $X_1$ and $X_2$ be cubic fourfolds with an equivalence $\Phi\colon\Ku(X_1)\isomor\Ku(X_2)$ commuting with the rotation functors. By \autoref{prop:cubic4foldKu1}, $\Phi$ is of Fourier\textendash Mukai type. Then $\Phi$ induces a Hodge isometry
$$\Phi^H\colon\wH(\Ku(X_1),\Z)\isomor\wH(\Ku(X_2),\Z),$$
between the Mukai lattices. It is a simple exercise, using our assumption $\fO_{X_2}\circ\Phi=\Phi\circ\fO_{X_1}$ and \cite[Proposition 3.12]{Huy:cubics}, to show that $\Phi^H$ sends the $A_2$-lattice of $X_1$ to the corresponding one of $X_2$.

\begin{remark}
It is clear that, for this argument, one can weaken the assumption $\fO_{X_2}\circ\Phi=\Phi\circ\fO_{X_1}$ in \autoref{thm:catTT4folds} to its cohomological version $\fO_{X_2}^H\circ\Phi^H=\Phi^H\circ\fO_{X_1}^H$. Here we have that $\fO_{X_i}^H\colon\wH_{\mathrm{Hodge}}(\Ku(X_i), \Z)\isomor\wH_{\mathrm{Hodge}}(\Ku(X_i),\Z)$ denotes the Hodge isometry induced by $\fO_{X_i}$.
\end{remark}

Recall that if $X$ is a cubic fourfold, the middle cohomology $H^4(X,\Z)$ has a natural lattice and Hodge structure (see \cite{HuyLect} for an excellent introduction). If $H$ is the class of a hyperplane section then the self-intersection $H^2$ is an algebraic class in $H^4(X,\Z)$. Then we denote by $H^4_{\mathrm{prim}}(X,\Z)$ the orthogonal to $H^2$ in $H^4(X,\Z)$. Clearly, such a sublattice inherits a Hodge structure from $H^4(X,\Z)$. By \cite{AT} there is (up to Tate twist) a Hodge-isometry
$$H^4_{\mathrm{prim}}(X,\Z) \cong A_2^\perp \subset \wH(\Ku(X),\Z).$$

In our setting, using the above identification and the fact that $\Phi^H$ preserves the $A_2$-lattices of $X_1$ and $X_2$, it follows that $\Phi^H$ induces a Hodge isometry
\[
\varphi\colon H^4_{\mathrm{prim}}(X_1,\Z)\isomor H^4_{\mathrm{prim}}(X_2,\Z).
\]
\autoref{thm:catTT4folds} then follows from the following beautiful classical result. Its first proof was a masterpiece in Hodge theory due to Voisin \cite{Voi}. Alternative and more recent proofs are due to Looijenga \cite{Loo} and Charles \cite{Ch}. 

\begin{thm}[Classical Torelli theorem]\label{thm:classicaTT}
Two smooth complex cubic fourfolds $X_1$ and $X_2$ are isomorphic if
and only if there exists a Hodge isometry $H^4_{\mathrm{prim}}(X_1,\Z)\cong H^4_{\mathrm{prim}}(X_2,\Z)$.
\end{thm}

Since we promised that the second proof of \autoref{thm:catTT4folds} would have been based on stability conditions, we are going to provide an alternative proof of \autoref{thm:classicaTT} based on these techniques and following the Appendix of \cite{BLMS}.

\begin{proof}[Proof of \autoref{thm:classicaTT}]
Of course, if $X_1\cong X2$, then there is a Hodge isometry $H^4_{\mathrm{prim}}(X_1,\Z)\cong H^4_{\mathrm{prim}}(X_2,\Z)$. For the other implication, we start with a Hodge isometry
\[
\varphi\colon H^4_{\mathrm{prim}}(X_1,\Z)\isomor H^4_{\mathrm{prim}}(X_2,\Z).
\]
The argument proceeds now by taking a local deformation of $X_i$. Indeed, as explained in \cite{HR}, in view of the local Torelli theorem, $\varphi$ extends to the bases of the universal deformation spaces $\mathrm{Def}(X_1)\cong\mathrm{Def}(X_2)$, which are considered as open subsets of the period domain. More precisely, one can find an identification $\mathrm{Def}(X_1)\cong\mathrm{Def}(X_2)$ such that parallel transport induces a Hodge-isometry
\[
\varphi_d\colon H^4_{\mathrm{prim}}(X_{1,d},\Z)\isomor H^4_{\mathrm{prim}}(X_{2,d},\Z),
\]
where $X_{i, d}$ is the local deformation of $X_i$ parametrized by $d \in \mathrm{Def}(X_i)$.
Then a lattice theoretic argument (see \cite[Proposition 4.2]{HR}) shows that for every $d \in \mathrm{Def}(X_1)\cong\mathrm{Def}(X_2)$ the Hodge-isometry $\varphi_d$ lifts to an orientation preserving Hodge isometry $$\phi_d \colon \wH(\Ku(X_{1, d}),\Z)\isomor\wH(\Ku(X_{2, d}),\Z)$$ 
which commutes with the action of the degree shift functors on the Mukai lattices.

Consider the set $D_i$ of points of $\mathrm{Def}(X_i)$ corresponding to cubic fourfolds $X$ such that $\Ku(X)\cong\Db(S,\alpha)$, for $S$ a K3 surface and $\alpha\in\mathrm{Br}(S)$, and $\wH_{\mathrm{Hodge}}(\Ku(X),\Z)$ does not contain $(-2)$-classes. Since the condition of having Kuznetsov component equivalent to the bounded derived category of a twisted K3 surface is determined by the Mukai lattice \cite[Theorem 1.4]{Huy:cubics}, \cite[Proposition 33.1]{BLMNPS}, we see that $\phi_d$ preserves this property. Analogously, $\phi_d$ preserved the property of not having $(-2)$-classes. We thus conclude that $\mathrm{Def}(X_1)\cong\mathrm{Def}(X_2)$ restricts to an isomorphism $D_1 \cong D_2$. Thus we can set $D:=D_1 \cong D_2$. As explained in the appendix to \cite{BLMS}, the set $D$ is dense. Moreover, for all $d\in D$, the Hodge isometry $\phi_d$ can be lifted to an equivalence $$\Phi_d\colon\Ku(X_{1, d})\isomor\Ku(X_{2, d})$$ 
by the derived Torelli theorem for twisted K3 surfaces \cite[Theorem 0.1]{HS}. By construction, the isometry $\Phi_d^H$ commutes with the action of the degree shifts functors in cohomology.

Now, for $d\in D$, pick $\sigma_1\in\mathrm{Stab}^\dagger(\Ku(X_{1, d}))$ such that $\eta(\sigma_1)\in(A_2\otimes\C)\cap\mathcal{P}\subseteq\mathcal{P}_0^+$ and set $\sigma_2:=\Phi_d(\sigma_1)$. By \cite[Theorem 1]{HMS} the stability manifold $\Stab^\dag(\Ku(X_{i,d}))$ has a unique connected component of maximal dimension. Since the action of $\Phi_d$ on the stability manifolds exchanges components of the same dimension, it follows that the image $\sigma_2$ of $\sigma_1$ belongs to $\Stab^\dag(\Ku(X_{2,d}))$. Moreover, by definition, $\eta(\sigma_2)\in(A_2\otimes\C)\cap\mathcal{P}$ as well.
Thus we can apply \autoref{thm:Fano} twice and obtain a string of isomorphisms
\[
F_1(X_{1,d})\cong M_{\sigma_1}(\Ku(X_{1,d}),\llambda_1)\cong M_{\sigma_2}(\Ku(X_{2,d}),\llambda_1)\cong F_1(X_{2,d}).
\]
The isomorphism in the middle is induced by $\Phi_d$ and, as explained in \cite[Appendix A]{BLMS}, it sends the polarization $\ell_{\sigma_1}$ to $\ell_{\sigma_2}$. Thus, by \autoref{thm:Fano}, the whole sequence of isomorphisms sends the Pl\"ucker polarization on $F_1(X_{1,d})$ to the Pl\"ucker polarization on $F_1(X_{2,d})$. The proof then continues as in the one of \autoref{thm;CTTcubic3folds} and it consists in applying Chow's trick (see \cite[Proposition 4]{Ch}) in order to conclude that $X_{1,d}\cong X_{2,d}$ for all $d\in D$. Since $D$ is dense, separatedness of the moduli space of cubic fourfolds implies that $X_1$ and $X_2$ have to be isomorphic (since any open neighborhood of the point corresponding to $X_1$ in $\mathrm{Def}(X_1)$ intersects any open neighborhood of the point corresponding to $X_2$ in $\mathrm{Def}(X_2)$, precisely in the points which belong to $D$). This ends the proof.
\end{proof}

\subsection{Gushel\textendash Mukai fourfolds}\label{sec_GM4}
We could wonder whether the techniques explained in the previous sections may be adapted to other classes of Fano fourfolds. This turns out to be true in the case of Gushel\textendash Mukai fourfolds (more generally for Gushel\textendash Mukai varieties of even dimension). 

Recall that a general complex \emph{Gushel\textendash Mukai (GM) fourfold} $X$ is a smooth four-dimensional intersection of the form
$$X= Q \cap \text{Gr}(2, 5) \subset \P^{9},$$
where $\text{Gr}(2, 5)$ is the Pl\"ucker embedded Grassmannian and $Q$ is a quadric hypersurface in a hyperplane section of $\P^{9}$. GM fourfolds share many similarities with cubic fourfolds. For instance, from a geometric viewpoint, there are known examples of rational GM fourfolds, but it is still unknown whether the very general one is irrational or rational. From the point of view of derived categories, Kuznetsov and Perry proved in \cite{KP_GM} that the bounded derived category of $X$ has a semiorthogonal decomposition of the form
$$\Db(X)= \langle \Ku(X), \OO_X, \UU_X^\vee, \OO_X(1), \UU_X^\vee(1) \rangle,$$
where $\UU_X$ is the restriction of the tautological bundle of rank $2$ on the Grassmannian $\text{Gr}(2,5)$ and $\Ku(X):= \langle \OO_X, \UU_X^\vee, \OO_X(1), \UU_X^\vee(1) \rangle^\perp$. The residual component $\Ku(X)$ is the Kuznetsov component of $X$ and is a noncommutative K3 surface. 

Full numerical stability conditions on $\Ku(X)$ have been constructed in \cite{PPZ}, using a dimension reduction trick. More precisely, the authors show that $X$ is birational to a conic fibration over a quadric threefold $Y$ and then provide an embedding of $\Ku(X)$ in the bounded derived category $\Db(Y, \BB_0)$, where $(Y, \BB_0)$ is a twisted quadric threefold.

On the other hand, it is known that the Torelli Theorem does not hold for GM fourfolds. In fact, in this case the period map has four-dimensional fibers by \cite{DIM}. Nevertheless, we can still wonder whether a (refined) Categorical Torelli theorem holds for GM fourfolds. 

More precisely, note that GM fourfolds in the same fiber of the period map have equivalent Kuznetsov components by \cite[Theorem 1.6]{KP_cones}. On the other hand, by \cite[Theorem 1.3]{BP} there are examples of GM fourfolds with equivalent Kuznetsov components, but defining different period points. These considerations suggest that we have to impose some additional conditions to an equivalence between the Kuznetsov components to recover the period point or the isomorphism class of a GM fourfold. 

Recall that the degree shift functor of $\Ku(X)$ is defined by
\[
\mathsf{O}_X:=\L_{\langle \OO_X, \UU_X^\vee \rangle} (- \otimes \OO_X(1))[-1],
\]
where $\L_{\langle \OO_X, \UU_X^\vee \rangle}$ is the left mutation in $\langle \OO_X, \UU_X^\vee \rangle$ (see, for example, \cite[Section 2]{Kuz_CY} for the precise definition). We can formulate the following question:

\begin{qn}
Let $X_1$ and $X_2$ be Gushel\textendash Mukai fourfolds. Assume that there is an equivalence $\Phi \colon \Ku(X_1) \isomor \Ku(X_2)$ which commutes with the degree shift functors of $X_1$ and $X_2$. Then under which assumptions on $\Phi$ we have that $X_1$ and $X_2$ are isomorphic? 
\end{qn}

To address this question, it may be helpful to use the stability conditions defined on $\Ku(X)$ and the associated moduli spaces having the structure of hyperk\"ahler manifolds \cite{PPZ}, as done for cubic fourfolds in the last part of \autoref{subsect:catThm4folds}. 

\begin{remark}
Let $X_1$ and $X_2$ be Gushel\textendash Mukai fourfolds. Assume that there is an equivalence $\Phi \colon \Ku(X_1) \isomor \Ku(X_2)$ which commutes with the degree shift functors. One could ask the intermediate question whether $\Phi$ induces a Hodge isometry $H^4_{\text{van}}(X_1, \Z) \cong H^4_{\text{van}}(X_2, \Z)$ between the degree-four vanishing cohomologies of $X_1$ and $X_2$ (see \cite[Section 3.3]{DebKuz_GM} for the definition). This is answered positively by \cite[Proposition 1.12]{BP_inprep}. 

Note that if $X_1$ and $X_2$ are \emph{very general}, then the above statement holds for every exact equivalence $\Phi \colon \Ku(X_1) \isomor \Ku(X_2)$. This can be proved exactly as in \cite[Corollary 3.6]{Huy:cubics}. Indeed, by \cite{LPZStrong} every equivalence as above is of Fourier-Mukai type. Then by \cite[Corollary 3.5]{Huy:cubics}, it induces a Hodge isometry $\Phi^H$ between the Mukai lattices of $\Ku(X_1)$ and $\Ku(X_2)$. Note that the condition of being very general means that the algebraic part of the Mukai lattice of $\Ku(X_i)$ is generated by two classes $\llambda_1$, $\llambda_2$ (see \cite{KP_GM} and \cite[Lemma 2.4]{Pert} for the precise definitions). As a consequence, the Hodge isometry $\Phi^H$ restricts to a Hodge isometry $H^4_{\text{van}}(X_1, \Z) \cong H^4_{\text{van}}(X_2, \Z)$. Of course, we cannot say $X_1$ and $X_2$ are isomorphic (and in general they are not), since there is no Torelli theorem.
\end{remark}

We finally recall the following conjecture about Birational Categorical Torelli theorem for Gushel\textendash Mukai varieties. The same comment in \autoref{rmk_onlyonedirection} applies to this setting.

\begin{conj}[\cite{KP_cones}, Conjecture 1.7]
If $X_1$ and $X_2$ are Gushel\textendash Mukai varieties of the same dimension such that there is an equivalence $\Ku(X_1) \isomor \Ku(X_2)$, then $X_1$ and $X_2$ are birational.
\end{conj}

\bigskip

{\small\noindent{\bf Acknowledgements.}} This project started when the second author was visiting the Max Planck Institute (Bonn) whose hospitality is gratefully acknowledged. It is our pleasure to thank Arend Bayer, Chunyi Li, Zhiyu Liu, Alex Perry, Emanuele Macr\`i for carefully reading a preliminary version of this paper and for many insightful comments. We are mostly grateful to the anonymous referee who carefully read this paper and helped us with several interesting comments which allowed us to improve the presentation.



\begin{thebibliography}{999}

\bibitem{AB}
D.\ Arcara, A.\ Bertram,
\emph{Bridgeland-stable moduli spaces for K-trivial surfaces},  with an appendix by Max Lieblich, J.\ Eur.\ Math.\ Soc.\ \textbf{15}, no.\ 1 (2013), 1–38.
	
\bibitem{AT} N.\ Addington, R.\ Thomas, \emph{Hodge theory and derived categories of cubic fourfolds}, Duke Math.\ J.\ {\bf 163} (2014), 1885--1927.

\bibitem{AH-LH}
J.\ Alper, D.\ Halpern-Leistner, J.\ Heinloth,
\emph{Existence of moduli spaces for algebraic stacks}, arXiv:1812.01128.
	
\bibitem{AKW} B.\ Antieau, D.\ Krashen, M.\ Ward, \emph{Derived categories of torsors for abelian schemes}, Adv.\ Math.\ {\bf 306} (2017), 1--23.

\bibitem{APR} M.\ Altavilla, M.\ Petkovic, F.\ Rota, \emph{Moduli spaces on the Kuznetsov component of Fano threefolds of index 2}, to appear in:  \'Epijournal G\'eom.\ Alg\'ebrique, arXiv:1908.10986.

\bibitem{ABB} A.\ Auel, M.\ Bernardara, M.\ Bolognesi, \emph{Fibrations in complete intersections of quadrics, Clifford algebras, derived categories, and rationality problems}, J.\ Math.\ Pures Appl.\ {\bf 102} (2015), 249--291.

\bibitem{BBS} Ch.\ B\"ohning, H.-Ch.\ Graf v.\ Bothmer, P.\ Sosna, \emph{On the Jordan\textendash H\"older property for geometric derived categories}, Adv.\ Math.\ {\bf 256} (2014), 479--492.

\bibitem{soheylaetal}
A.\ Bayer, S.\ Beentjes, S.\ Feyzbakhsh, G.\ Hein, D.\ Martinelli, F.\ Rezaee, B.\ Schmidt,
\emph{The desingularization of the theta divisor of a cubic threefold as a moduli space}, to appear in: Geometry and Topology, arXiv:2011.12240.

\bibitem{BB} A.\ Bayer, T.\ Bridgeland, \emph{Derived automorphism groups of K3 surfaces of Picard rank $1$}, Duke Math.\ J.\ {\bf 166} (2017), 75--124.

\bibitem{BMICM}
A.\ Bayer, E.\ Macrì, 
\emph{The unreasonable effectivness of wall\textendash crossing in algebraic geometry}, to appear in: Proceedings of the ICM 2022, arXiv:2201.03654.

\bibitem{BMS}
A.\ Bayer, E.\ Macrì, P.\ Stellari, 
\emph{The space of stability conditions on abelian threefolds, and on some Calabi-Yau threefolds}, Invent.\ Math.\ \textbf{206} (2016), 869--933.

\bibitem{BMT}
A.\ Bayer, E.\ Macrì, Y.\ Toda,
\emph{Bridgeland stability conditions on threefolds I: Bogomolov-Gieseker type inequalities}, J.\ Algebraic Geom.\ \textbf{23} (2014), 117-163.

\bibitem{BLMNPS} A.\ Bayer, M.\ Lahoz, E.\ Macr\`i, H.\ Nuer, A.\ Perry, P.\ Stellari, \emph{Stability conditions in families}, Publ.\ Math.\ Inst.\ Hautes \'Etudes Sci.\ {\bf 133} (2021), 157--325.

\bibitem{BLMS} A.\ Bayer, M.\ Lahoz, E.\ Macr\`i, P.\ Stellari, \emph{Stability conditions on Kuznetsov components, (Appendix joint also with X. Zhao)}, to appear in: Ann.\ Sci.\ \'Ec.\ Norm.\ Sup\'er., arXiv:1703.10839.

\bibitem{BP_inprep}
A.\ Bayer, A.\ Perry,
\emph{Kuznetsov's Fano threefold conjecture via K3 categories and enhanced group actions}, arXiv:2202.04195.

\bibitem{Beau}
A.\ Beauville,
\emph{Vector bundles on the cubic threefold}, Contemp.\ Math.\ \textbf{312}, Amer.\ Math.\ Soc., Providence, RI, 2002.

\bibitem{BD} A.\ Beauville, R.\ Donagi, \emph{La vari\'et\'e des droites d’une hypersurface cubique de dimension $4$}, C.\ R.\ Acad.\ Sci.\ Paris S\'er.\ I Math.\ {\bf 301} (1985), 703--706.

\bibitem{Bei:Pn} A.\ Beilinson, \emph{Coherent sheaves on $\P^n$ and problems in linear algebra}, Funct.\ Anal.\ Appl. {\bf 12} (1979), 214--216.

\bibitem{BBD}
A.\ Beilinson, J.\ Bernstein, P.\ Deligne, 
\emph{Faisceaux pervers}, Analysis and topology on singular
spaces, I (Luminy, 1981), Astérisque \textbf{100} (1982), 5--171.

\bibitem{BMMS} M.\ Bernardara, E.\ Macr\`i, S.\ Mehrotra, P.\ Stellari, \emph{A categorical invariant for cubic threefolds}, Adv.\ Math.\ {\bf 229} (2012), 770--803.

\bibitem{BMSZ}
M.\ Bernardara, E.\ Macr{\`{\i}}, B.\ Schmidt, X.\ Zhao,
\emph{Bridgeland stability conditions on Fano threefolds}, \'{E}pijournal Geom. Alg\'{e}brique \textbf{1}, 2 (2017).

\bibitem{BT}
M.\ Bernardara, G.\ Tabuada,
\emph{From semi-orthogonal decompositions to polarized intermediate Jacobians via Jacobians of noncommutative motives}, Mosc.\ Math.\ J.\ \textbf{16}
(2016), no.\ 2, 205–235.

\bibitem{BrEq} T.\ Bridgeland, \emph{Equivalences of triangulated categories and Fourier\textendash Mukai transforms}, Bull.\ London Math.\ Soc. {\bf 31} (1999), 25--34.

\bibitem{Bri}
T.\ Bridgeland,
\emph{Stability conditions on triangulated categories}, Ann.\ of Math.\ \textbf{166}, no.\ 2 (2007), 317–-345.

\bibitem{BriK3}
T.\ Bridgeland,
\emph{Stability conditions on K3 surfaces}, Duke Math.\ J.\ \textbf{141}, no.\ 2 (2008), 241–291.

\bibitem{BM01} T.\ Bridgeland, A.\ Maciocia, \emph{Complex surfaces with equivalent derived categories}, Math.\ Z. {\bf 236} (2001), 677--697.

\bibitem{B}
A.\ Bondal, 
\emph{Representations of associative algebras and coherent sheaves}, Math.\ USSR Izvestiya, \textbf{34} (1990), No.\ 1, 23–42.

\bibitem{BLL} A.\ Bondal, M.\ Larsen, V.\ Lunts, \emph{Grothendieck ring of pretriangulated categories}, Int.\ Math.\ Res.\ Not.\ {\bf 29} (2004), 1461--1495.

\bibitem{BO} A.\ Bondal, D.\ Orlov, \emph{Reconstruction of a variety from the derived category and groups of autoequivalences}, Compositio Math.\ {\bf 125} (2001), 327--344.

\bibitem{BondalOrlov:Main} A.\ Bondal, D.\ Orlov, \emph{Semiorthogonal decomposition for algebraic varieties}, arXiv:alg-geom/9506012.

\bibitem{BVdB} A.\ Bondal, M. Van den Bergh, \emph{Generators and representability of functors in commutative and noncommutative geometry}, Moscow Math. J. {\bf 3} (2003), 1--36.

\bibitem{BP}
E.\ Brakkee, L.\ Pertusi,
\emph{Marked and labelled Gushel-Mukai fourfolds}, Rationality of Varieties, Progr.\ Math.\ \textbf{342}, Birkhauser Basel (2021).

\bibitem{Cal:thesis} A.\ C\u{a}ld\u{a}raru, \emph{Derived categories of twisted sheaves on Calabi\textendash Yau manifolds}, PhD-Thesis, Cornell University (2000).

\bibitem{CNS} A.\ Canonaco, A.\ Neeman, P.\ Stellari, \emph{Uniqueness of enhancements for derived and geometric categories}, arXiv:2101.04404.

\bibitem{COS} A.\ Canonaco, D.\ Orlov, P.\ Stellari, \emph{Does full imply faithful?}, J.\ Noncommut.\ Geom.\ {\bf 7} (2013), 357--371.

\bibitem{CS:SurvFM} A.\ Canonaco, P.\ Stellari, \emph{Fourier--Mukai functors: a survey},  EMS Ser.\ Congr.\ Rep., Eur.\ Math.\ Soc.\ (2013), 27--60.

\bibitem{CSSupp} A.\ Canonaco, P.\ Stellari, \emph{Fourier\textendash Mukai functors in the supported case}, Compositio Math.\ {\bf 150} (2014), 1349--1383.

\bibitem{CS} A.\ Canonaco, P.\ Stellari, \emph{Twisted Fourier\textendash Mukai functors}, Adv.\ Math.\ {\bf 212} (2007), 484--503.

\bibitem{CSUni1} A.\ Canonaco, P.\ Stellari, \emph{Uniqueness of dg enhancements for the derived category of a Grothendieck category}, J.\ Eur.\ Math.\ Soc.\ {\bf 20} (2018), 2607--2641.

\bibitem{CH} M.\ Casanellas, R.\ Hartshorne, F.\ Gleiss, F.O. Schreyer, \emph{Stable Ulrich bundles}, Int.\ J.\ Math.\ {\bf 23} (2012), 1250083--1250133.

\bibitem{Ch} F.\ Charles, \emph{A remark on the Torelli theorem for cubic fourfolds}, preprint.

\bibitem{CG} C.H.\ Clemens, P.\ Griffiths, \emph{The intermediate Jacobian of the cubic threefold}, Ann.\ Math.\ {\bf 95} (1972), 281--356.

\bibitem{CP}
J.\ Collins, A.\ Polishchuk, 
\emph{Gluing stability conditions}, 
Adv.\ Theor.\ Math.\ Phys.\ \textbf{14}(2) (2010), 563–607.

\bibitem{CD89} F.\ Cossec, I.\ Dolgachev, \emph{Enriques surfaces. I},
Progress in Mathematics {\bf 86}, Birkh\"auser (1989).

\bibitem{DIM}
O.\ Debarre, A.\ Iliev, L.\ Manivel,
\emph{Special prime Fano fourfolds of degree 10 and index 2}, Recent advances in algebraic geometry, London Math.\ Soc.\ Lecture Note Ser.\ \textbf{417}, Cambridge Univ.\ Press (2015), 123–155.

\bibitem{DebKuz_GM}
O.\ Debarre, A.\ Kuznetsov,
\emph{Gushel--Mukai varieties: Linear spaces and
periods}, Kyoto J.\ Math.\ {\bf 59} (2019), no.\ 4, 897--953.

\bibitem{D} I.\ Dolgachev, \emph{A brief introduction to Enriques surfaces}, in: Development of Moduli (Theory-Kyoto 2013), 1--32, Adv.\
Study in Pure Math.\ {\bf 69}, Math.\ Soc.\ Japan (2016).

\bibitem{DK} I.\ Dolgachev, S.\ Kond\={o}, \emph{Enriques surfaces, II}, {\tt http://www.math.lsa.umich.edu/~idolga/EnriquesTwo.pdf}.

\bibitem{Donagi} R.\ Donagi, \emph{Generic Torelli for projective hypersurfaces}, Compositio Math.\ {\bf 50} (1983), 325--353.

\bibitem{Fey}
S.\ Feyzbakhsh, 
\emph{An effective restriction theorem via wall-crossing and Mercat's conjecture}, arXiv:1608.07825.

\bibitem{FP}
S.\ Feyzbakhsh, L.\ Pertusi,
\emph{Serre-invariant stability conditions and Ulrich bundles on cubic threefolds}, arXiv:2109.13549.

\bibitem{Fulton}
W.\ Fulton,
\emph{Intersection theory}, Springer-Verlag Berlin Heidelberg 1998 XIII, 470, Springer New York, NY.

\bibitem{HRS}
D.\ Happel, I.\ Reiten, S.\ Smalø, 
\emph{Tilting in abelian categories and quasitilted algebras}, Mem.\ Amer.\ Math.\ Soc.\ \textbf{120} (1996), viii+ 88pp.

\bibitem{HLT17} K.\ Honigs, M.\ Lieblich, S.\ Tirabassi, \emph{Fourier\textendash Mukai partners of Enriques and bielliptic surfaces in positive characteristic}, Math.\ Res.\ Lett. {\bf 28} (2021), 65--91.

\bibitem{Huy} D.\ Huybrechts, \emph{Fourier--Mukai transforms in algebraic geometry}, Oxford Mathematical Monographs, Oxford Science Publications (2006).

\bibitem{HuyLect} D.\ Huybrechts,, \emph{Hodge theory of cubic fourfolds, their Fano varieties, and associated K3 surfaces}, In: Birational Geometry of Hypersurfaces, 165--198, Lecture Notes of the Unione Matematica Italiana {\bf 26}, Springer (2019).

\bibitem{Huy:cubics} D.\ Huybrechts, \emph{The K3 category of a cubic fourfold}, Compositio Math.\ {\bf 153} (2017), 586--620.

\bibitem{HMS} D.\ Huybrechts, E.\ Macr\`i, P.\ Stellari, \emph{Stability conditions for generic K3 categories}, Compositio Math.\ \textbf{144} (2008), 134--162.

\bibitem{HMS:K3} D.\ Huybrechts, E.\ Macr\`i, P.\ Stellari, \emph{Derived equivalences of K3 surfaces and orientation}, Duke Math.\ J.\ {\bf 149} (2009), 461--507.

\bibitem{HR} D.\ Huybrechts, J.\ Rennemo, \emph{Hochschild cohomology versus the Jacobian ring, and the Torelli theorem for cubic fourfolds}, Algebr.\ Geom.\ {\bf 6}(2019), 76--99.

\bibitem{HT} D.\ Huybrechts, R.\ Thomas, \emph{$\mathbb{P}$-objects and autoequivalences of derived categories}, Math.\ Res.\ Lett.\ {\bf 13} (2006), 87--98.

\bibitem{HS} D.\ Huybrechts, P.\ Stellari, \emph{Proof of Caldararu's conjecture. An appendix to a paper by K. Yoshioka},  In: The 13th MSJ Inter.\ Research Inst.\ - Moduli Spaces and Arithmetic Geometry, 31--42, Adv.\ Stud.\ Pure Math.\ \textbf{45}, Math.\ Soc.\ Japan, Tokyo (2006).

\bibitem{HMPS}
D.\ Huybrechts, E.\ Macrì, L.\ Pertusi, P.\ Stellari,
\emph{Categorical Torelli theorem for quartic threefolds}, in progress.

\bibitem{IK} C.\ Ingalls, A.\ Kuznetsov, \emph{On nodal Enriques surfaces and quartic double solids}, Math.\ Ann.\ {\bf 361} (2015), 107--133.

\bibitem{Is} V.\ Iskovskikh, \emph{Anticanonical models of three-dimensional algebraic varieties}, in: Current problems in mathematics,
VINITI, Moscow, 12, 59--157 (Russian); translation in J.\ Soviet Math.\ {\bf 13} (1980) 745--814.

\bibitem{Isk}
V.\ Iskovskikh, 
\emph{Fano threefolds I}, Izv.\ Akad.\ Nauk SSSR Ser.\ Mat.\ \textbf{41} (1977), no.\ 3, 516–562, 717.

\bibitem{JLLZ} A.\ Jacovskis, X.\ Lin, Z.\ Liu, S.\ Zhang,
\emph{Hochschild cohomology and categorical Torelli for Gushel-Mukai threefolds}, arXix:2108.02946.

\bibitem{JLLZ2} A.\ Jacovskis, X.\ Lin, Z.\ Liu, S.\ Zhang,
\emph{Infinitesimal categorical Torelli theorems for Fano threefolds}, arXiv:2203.08187.

\bibitem{JLZ}
A.\ Jacovskis, Z.\ Liu, S.\ Zhang,
\emph{Brill-Noether theory for Kuznetsov components and refined Categorical Torelli theorems for index one Fano threefolds}, arXiv:2207.01021.

\bibitem{KLS}
D.\ Kaledin,  M.\ Lehn, Ch.\ Sorger,
\emph{Singular  symplectic  moduli  spaces}, Invent.\ Math.\ {\bf 164}(3) (2006), 591--614.

\bibitem{Kap:Grass} M.\ Kapranov, \emph{On the derived categories of coherent sheaves on some homo-geneous spaces}, Invent.\ Math.\ {\bf 92} (1988), 479--508.

\bibitem{KO} K.\ Kawatani, S.\ Okawa, \emph{Nonexistence of semiorthogonal decompositions and sections of the canonical bundle}, arXiv:1508.00682.

\bibitem{KLOS_book}
J.\ Kollár, M.\ Lieblich, M.\ Olsson, W.\ Sawin,
\emph{The Zariski topology, linear systems, and
algebraic varieties}, \url{https://math.berkeley.edu/~molsson/Reconstructionweb.pdf}.

\bibitem{Koseki_nef} N.\ Koseki, \emph{Stability conditions on threefolds with nef tangent bundle}, Adv.\ Math.\ \textbf{372} (2020).

\bibitem{Koseki} N.\ Koseki, \emph{Stability conditions on Calabi-Yau double/triple solids}, arXiv:2007.00044.

\bibitem{Krug} A.\ Krug, \emph{Varieties with $\mathbb{P}$-units}, Trans.\ Amer.\ Math.\ Soc.\ {\bf 370} (2018), 7959--7983.

\bibitem{KuzJH} A.\ Kuznetsov, \emph{A simple counterexample to the Jordan\textendash H\"older property for derived categories}, arXiv:1304.0903.

\bibitem{Kuz_CY} A.\ Kuznetsov, \emph{Calabi–Yau and fractional Calabi–Yau categories}, J.\ Reine Angew.\ Math.\ {\bf 753} (2019), 239--267.

\bibitem{Kuz11} A.\ Kuznetsov, \emph{Base change for semiorthogonal decompositions}, Compositio Math.\ {\bf 147} (2011), 852--876.

\bibitem{Kuzcubics} A.\ Kuznetsov, \emph{Derived categories of cubic fourfolds}, in: Cohomological and geometric approaches to rationality problems, 219--243, Progr. Math. {\bf 282}, Birkh\"auser (2010).

\bibitem{Kuz:V14} A.\ Kuznetsov, \emph{Derived category of a cubic threefold and the variety {$V_{14}$}}, Tr.\ Mat.\ Inst.\ Steklova {\bf 246} (2004), 183--207.

\bibitem{Kuz:Fano3folds} A.\ Kuznetsov, \emph{Derived categories of Fano threefolds}, Tr.\ Mat.\ Inst.\ Steklova {\bf 264} (2009), 116--128.

\bibitem{Kuz2} A.\ Kuznetsov, \emph{Derived categories of quadric fibrations and intersections of quadrics}, Adv. Math. {\bf 218} (2008), 1340--1369.

\bibitem{KLect} A.\ Kuznetsov, \emph{Derived categories view on rationality problems}, in: Rationality Problems in Algebraic Geometry, Lecture Notes in Math.\ {\bf 2172}, 67--104, Springer (2016).

\bibitem{Kuz07} A.\ Kuznetsov, \emph{Homological projective duality}, Publ.\ Math.\ Inst.\ Hautes \'Etudes Sci.\ {\bf 105} (2007), 157--220.

\bibitem{Kuznetsov:Hyperplane} A.\ Kuznetsov, \emph{Hyperplane sections and derived categories}, Izv.\ Ross.\ Akad.\ Nauk Ser.\ Mat.\ {\bf 70} (2006), 23--128.

\bibitem{KuzHoch} A.\ Kuznetsov, \emph{Hochschild homology and semiorthogonal decompositions}, arXiv:0904.4330.

\bibitem{KL} A.\ Kuznetsov, V.\ Lunts, \emph{Categorical resolutions of irrational singularities}, Int. Math.\ Res.\ Not. IMRN 2015, 4536--4625.

\bibitem{KP_GM} A.\ Kuznetsov, A.\ Perry, \emph{Derived categories of Gushel-Mukai varieties}, Compos.\ Math.\ \textbf{154} (2018), no.\ 7, 1362-1406.

\bibitem{KP_cones} A.\ Kuznetsov, A.\ Perry, \emph{Categorical cones and quadratic homological projective duality}, to appear in: Ann.\ Sci.\ \'Ec.\ Norm.\ Sup\'er., arXiv:1902.09824.

\bibitem{KP} A.\ Kuznetsov, A.\ Perry, \emph{Serre functors and dimensions of residual categories}, arXiv:2109.02026.

\bibitem{LMS} M.\ Lahoz, E.\ Macr\`i, P.\ Stellari, \emph{Arithmetically Cohen\textendash Macaulay bundles on cubic threefolds}, Algebr.\ Geom.\ {\bf 2} (2015), 231--269.

\bibitem{Lieblich}
M.\ Lieblich,
\emph{Moduli of complexes on a proper morphism}, J.\ Algebraic Geom.\ \textbf{15} (2006), no.\ 1, 175–206.

\bibitem{LieOls}
M.\ Lieblich, M.\ Olsson,
\emph{Derived categories and birationality}, arXiv:2001.05995.

\bibitem{LiFano} C.\ Li, \emph{Stability conditions on Fano threefolds of Picard number one}, J.\ Eur.\ Math.\ Soc.\ \textbf{21} (2019), 709--726.

\bibitem{Li} C.\ Li, \emph{On stability conditions for the quintic threefold}, Invent.\ Math.\ \textbf{218} (2019), 301--340.

\bibitem{LNSZ} C.\ Li, H.\ Nuer, P.\ Stellari, X.\ Zhao, \emph{A refined Derived Torelli Theorem for Enriques surfaces}, Math.\ Ann.\ {\bf 379} (2021), 1475--1505.

\bibitem{LPZStrong} C.\ Li, L.\ Pertusi, X.\ Zhao, 
\emph{Derived categories of hearts on Kuznetsov components}, arXiv:2203.13864.

\bibitem{LPZ} C.\ Li, L.\ Pertusi, X.\ Zhao, \emph{Twisted cubics on cubic fourfolds and stability conditions,} arXiv:1802.01134.

\bibitem{LSZ} C.\ Li, P.\ Stellari, X.\ Zhao, \emph{A refined Derived Torelli Theorem for Enriques surfaces, II: the non-generic case}, Math.\ Z.\ {\bf 300} (2022), 3527--3550.

\bibitem{LZ_birational}
C.\ Li, X.\ Zhao,
\emph{Birational models of moduli spaces of coherent sheaves on the projective plane}, Geometry and Topology \textbf{23}(1) (2019), 347-426.

\bibitem{LZ_poisson}
C.\ Li, X.\ Zhao,
\emph{Smoothness and Poisson structures of Bridgeland moduli spaces on Poisson surfaces}, Math.Z. \textbf{291} (2019), 437-447.

\bibitem{SLiu}
S.\ Liu,
\emph{Stability condition on Calabi-Yau threefold of complete intersection of quadratic and quartic hypersurfaces}, arXiv:2108.08934.

\bibitem{Loo} E.\ Looijenga, \emph{The period map for cubic fourfolds}, Invent.\ Math.\ {\bf 177} (2009), 213--233.

\bibitem{LO} V.\ Lunts, D.\ Orlov, \emph{Uniqueness of enhancements for triangulated categories}, J.\ Amer.\ Math.\ Soc.\ {\bf 23} (2010), 853--908.

\bibitem{LS} V.\ Lunts, O.M.\ Schn\"urer, \emph{New enhancements of derived categories of coherent sheaves and applications}, J.\ Algebra {\bf 446} (2016), 203--274.

\bibitem{MP}
A.\ Maciocia, D.\ Piyaratne, 
\emph{Fourier-Mukai transforms and Bridgeland stability conditions on abelian threefolds}, Algebr.\ Geom.\ \textbf{2} (2015), 270--297.

\bibitem{MP2}
A.\ Maciocia, D.\ Piyaratne, 
\emph{Fourier-Mukai transforms and Bridgeland stability conditions on abelian threefolds II}, Internat.\ J.\ Math.\ \textbf{27} (2016), 1650007.

\bibitem{Macri}
E.\ Macrì,
\emph{Stability conditions on curves}, Math.\ Res.\ Lett.\ \textbf{14} (2007), no.\ 4, 657–672.

\bibitem{Macri2}
E.\ Macr{\`{\i}},
\emph{A generalized Bogomolov-Gieseker inequality for the three-dimensional projective space}, 
Algebra Number Theory \textbf{8} (2014), 173--190.

\bibitem{MSLectNotes} E.\ Macr\`i, P.\ Stellari, \emph{Lectures on non-commutative K3 surfaces, Bridgeland stability, and moduli spaces}, In: Birational Geometry of Hypersurfaces, 199--266, Lecture Notes of the Unione Matematica Italiana {\bf 26}, Springer (2019).

\bibitem{MS}
E.\ Macrì, B.\ Schmidt, 
\emph{Lectures on Bridgeland Stability}, Proceedings of the “CIMPA-CIMAT-ICTP School on Moduli of Curves” (Guanajuato, Mexico, 2016), Springer 2017.

\bibitem{MSt} E.\ Macr\`i, P.\ Stellari, \emph{Fano varieties of cubic fourfolds containing a plane}, Math.\ Ann.\ {\bf 354} (2012), 1147--1176.

\bibitem{Muk:ab} S.\ Mukai, \emph{Duality between $D(X)$ and $D(\hat{X})$ with its applications to Picard sheaves}, Nagoya Math.\ J.\ {\bf 81} (1981), 153--175.

\bibitem{Muk92} S.\ Mukai, \emph{Fano 3-folds}, in: Complex projective geometry (Trieste, 1989/Bergen, 1989), London Math.\ Soc.\ Lecture Note Ser. {\bf 179}, 255--263, Cambridge Univ. Press (1992).

\bibitem{Mu} S.\ Mukai \emph{On the moduli space of bundles on K3 surfaces, I}, in: Vector Bundles on Algebraic Varieties, Bombay (1984).

\bibitem{MU} S.\ Mukai, H.\ Umemura, \emph{Minimal rational threefolds}, Algebraic geometry (Tokyo/Kyoto, 1982), 490--518, Lecture
Notes in Math.\ {\bf 1016}, Springer (1983).

\bibitem{Og} K.\ Oguiso, \emph{K3 surfaces via almost-primes}, Math.\ Res.\ Lett.\ {\bf 9} (2002), 47--63.

\bibitem{Ok} S.\ Okawa, \emph{Semi-orthogonal decomposability of the derived category of a curve}, Adv.\ Math.\ {\bf 228} (2011), 2869--2873.

\bibitem{Ol} N.\ Olander, \emph{Orlov's theorem in the smooth proper case}, arXiv:2006.15173v1.

\bibitem{Or} D.\ Orlov, \emph{Equivalences of derived categories and K3 surfaces}, J.\ Math.\ Sci.\ {\bf 84} (1997), 1361--1381.

\bibitem{orlov:Y5} D.\ Orlov, \emph{Exceptional set of vector bundles on the variety {$V_5$}}, Vestnik Moskov.\ Univ.\ Ser.\ I Mat.\ Mekh.\ {\bf 5} (1991), 69--71.

\bibitem{Or:projbun} D.\ Orlov, \emph{Projective bundles, monoidal transformations, and derived categories of coherent sheaves}, Russian Acad.\ Sci.\ Izv.\ Math.\ {\bf 41} (1993), 133--141.

\bibitem{Ottaviani} 
G.\ Ottaviani,
\emph{Spinor bundles on quadrics}, Trans.\ Amer.\ Math.\ Soc.\ \textbf{307} (1988), 301–316.

\bibitem{PeHodge} A.\ Perry, \emph{The integral Hodge conjecture for two-dimensional Calabi\textendash Yau categories}, to appear in: Compositio Math., arXiv:2004.03163.

\bibitem{PPZ}
A.\ Perry, L.\ Pertusi, X.\ Zhao,
\emph{Stability conditions and moduli spaces for Kuznetsov components of Gushel-Mukai varieties}, to appear in: Geometry and Topology, arXiv:1912.06935.

\bibitem{PFMCub} L.\ Pertusi, \emph{Fourier-Mukai partners for very general special cubic fourfolds}, Math.\ Research Letters {\bf 28} (2021), 213--243.

\bibitem{Pert} L.\ Pertusi, \emph{On the double EPW sextic associated to a Gushel\textendash Mukai fourfold}, J.\ London Math.\ Soc.\ (1) \textbf{100} (2019), 83–106.

\bibitem{PY}
L.\ Pertusi, S.\ Yang,
\emph{Some remarks on Fano threefolds of index two and stability conditions}, to appear in: Int.\ Math.\ Res.\ Not.\, arXiv:2004.02798.

\bibitem{PR}
L.\ Pertusi, E.\ Robinett,
\emph{Stability conditions on Kuznetsov components of Gushel–Mukai threefolds and Serre functor}, arXiv:2112.04769.

\bibitem{PetRota}
M.\ Petkovich, F.\ Rota,
\emph{A note on the Kuznetsov component of the Veronese double cone}, arXiv:2007.05555.

\bibitem{R} R.\ Rouquier, \emph{Dimensions of triangulated categories}, J.\ K-theory {\bf 1} (2008), 193--258.

\bibitem{RS1} F.\ Russo, G.\ Staglian\`o, \emph{Congruences of $5$–secant conics and the rationality of some admissible cubic
fourfolds}, Duke Math.\ J.\ {\bf 168} (2019), 849--865.

\bibitem{RS2} F.\ Russo, G.\ Staglian\`o, \emph{Trisecant flops, their associated K3 surfaces and the rationality of some
Fano fourfolds}, to appear in: J.\ Eur.\ Math.\ Soc., arXiv:1909.01263.

\bibitem{Schmidt}
B.\ Schmidt,
\emph{A generalized Bogomolov-Gieseker inequality for the smooth quadric threefold}, Bull.\ Lond.\ Math.\ Soc. \textbf{46} (2014), 915–-923.

\bibitem{SchmidtCounter}
B.\ Schmidt,
\emph{Counterexample to the generalized Bogomolov-Gieseker inequality for threefolds}, Int.\ Math.\ Res.\ Not.\ IMRN 2017, 2562--2566.

\bibitem{ST} P.\ Seidel, R.\ Thomas, \emph{Braid group actions on derived categories of coherent sheaves},
Duke Math.\ J. {\bf 108} (2001), 37--108.

\bibitem{SB} N.I.\ Sherherd\textendash Barron, \emph{Fano threefolds in positive characteristic}, Compositio Math.\ {\bf 105} (1997), 237--265.

\bibitem{ShinderZhang}
E.\ Shinder, Z.\ Zhang, 
\emph{L-equivalence for degree five elliptic curves, elliptic fibrations and K3 surfaces}, Bull.\
Lond.\ Math.\ Soc.\ \textbf{52} (2020), no.\ 2, 395–409.

\bibitem{St} P.\ Stellari \emph{Some remarks about the FM-partners of K3 surfaces with Picard number 1 and 2}, Geom.\ Dedicata {\bf 108} (2004), 1--13.

\bibitem{Toda}
Y. \ Toda, 
\emph{Limit stable objects on Calabi--Yau 3-folds}, Duke Math.\ J.\ {\bf 149} (2009), 157–208.

\bibitem{Toda_bogomolov}
Y. \ Toda, 
\emph{Bogomolov--Gieseker-type inequality and counting invariants}, J. Topol.\ {\bf 6} (2013), 217–250.

\bibitem{To} B.\ To\"en, \emph{The homotopy theory of dg-categories and derived Morita theory}, Invent.\ Math.\ {\bf 167} (2007), 615--667.

\bibitem{VoiHodge} C.\ Voisin, \emph{Some aspects of the Hodge conjecture}, Jpn.\ J.\ Math.\ {\bf 2} (2007), 261--296.

\bibitem{Voi}  C.\ Voisin, \emph{Th\'eor\`eme de Torelli pour les cubiques de $\P^5$}, Invent.\ Math.\ {\bf 86} (1986), 577--601.

\bibitem{Zu} S.\ Zube, \emph{Exceptional vector bundles on Enriques surfaces}, Mathematical Notes {\bf 61} (1997), 693--699.

\end{thebibliography}
\end{document}